\documentclass[11pt]{amsart}
\usepackage[margin=30mm]{geometry}
\usepackage{amsmath,amssymb}
\usepackage{amsthm}
\usepackage{mathrsfs}

\newtheorem{thm}{Theorem}[section]
\newtheorem{prop}{Proposition}[section]
\newtheorem{lem}{Lemma}[section]
\newtheorem{cor}{Corollary}[section]
\newtheorem{defi}{Definition}[section]

\newtheorem{rem}{Remark}[section]
\newtheorem{q}{Question}[section]
\newtheorem*{ques}{\it{Question}}

\begin{document}

\title{Shadowing, transitivity and a variation of omega-chaos}
\author{Noriaki Kawaguchi}
\address{Research Institute of Science and Technology, Tokai University, 4-4-1 Kitakaname, Hiratsuka, Kanagawa 259-1292, Japan}
\email{gknoriaki@gmail.com}

\begin{abstract}
We study a special type of shadowing (DSP) of chain transitive continuous self-maps of compact Hausdorff spaces. We prove some basic properties of DSP.  As application of DSP, we obtain sufficient conditions for a statistical variant of $\omega$-chaos and prove the topological genericity of it. We also consider topological distribution of irregular points under the assumption of DSP.
\end{abstract}

\maketitle

\markboth{NORIAKI KAWAGUCHI}{Shadowing, transitivity and a variation of omega-chaos}

\section{Introduction}

{\em Shadowing} is a phenomenon discovered in the study of hyperbolic differentiable dynamics, especially, in the study of geodesic flows on negatively curved closed Riemannian manifolds by Anosov \cite{An} and Axiom A diffeomorphisms by Bowen \cite{B}. It is one of ``hyperbolic'' features of dynamical systems, which is invariant under the topological conjugacy, and generally refers to a property that coarse orbits, or {\em pseudo-orbits}, can be approximated by true orbits. Since Bowen and until recently, various types of shadowing properties have been defined for different purposes (see \cite{AH,P} for general background).

{\em Chaos} is a central notion in the modern theory of dynamical systems. One of topological definitions of chaos is the  {\em distributional chaos of type 1} (abbrev.\:DC1) introduced by Schweizer and Sm\'ital \cite{SS} as a statistical enhancement of Li--Yorke chaos \cite{LY}. The relationship between the shadowing and DC1 has been a subject of several studies \cite{AC,Ka2,Ka3,Ka4,LLT,OW}. Especially, in \cite{LLT}, Li, Li, and Tu raised the following question:

\begin{ques}
For any continuous self-map of a compact metric space, do the transitivity, the shadowing, and positive topological entropy imply DC1?
\end{ques}

Recently in \cite{Ka4}, the author gave a partial answer to this question. Precisely, it is proved in \cite{Ka4} that for any continuous self-map $f\colon X\to X$ of a compact metric space $X$, if
\begin{itemize}
\item[(1)] $f$ is transitive and has the shadowing property,
\item[(2)] $f$ has positive topological entropy,
\end{itemize}
and if $(X,f)$ has an extension $(Y,g)$ such that
\begin{itemize}
\item[(3)] $Y$ is a zero-dimensional, i.e., totally disconnected, compact metric space,
\item[(4)] $g$ is transitive and has the shadowing property,
\end{itemize}
then $f$ exhibits DC1 in a strong sense. It gives a positive answer to the above question especially when $X$ is zero-dimensional. A key technical observation in the proof of this result is that a {\em constrained} shadowing with respect to a certain closed partition of the phase space, followed from the conditions (1), (3), and (4), allows us to apply a theorem of Mycielski \cite[Theorem 1]{My} to obtain distributionally scrambled Mycielski sets in each part of the partition (see also \cite{Ka2}). The purpose of this paper is to extend the study of this special type of shadowing and provide its applications. Below is a brief summary of the contents of this paper.
\begin{itemize}
\item We generalize the special type of shadowing considered in \cite{Ka4}, so-called {\em $\mathcal{D}$-shadowing property} (abbrev.\:DSP), for chain transitive continuous self-maps of {\em compact Hausdorff spaces} and prove some of its basic properties. The properties include the preservation of DSP with respect to factorization and the continuity, with respect to the Vietoris topology, of the closed partition of the phase space that DSP relies on.  
\item We introduce {\em $\overline{\omega}$-chaos}, a statistical variant of $\omega$-chaos \cite{SHLi}, and prove that for any continuous self-map of a compact metric space, the transitivity, DSP, and positive topological entropy are sufficient conditions for $\overline{\omega}$-chaos. The proof is by a statistical shadowing followed from DSP, which is a generalization of the {\em asymptotic average shadowing property}. As a corollary, we show that the $C^0$-generic dynamics on closed differentiable manifolds is $\overline{\omega}$-chaotic. We also prove that for interval maps, $\overline{\omega}$-chaos is equivalent to positive topological entropy.
\item Under the assumption of DSP, we consider topological distribution of {\em irregular points} and improve a result of \cite{Ka4}.
\end{itemize}

First, we recall the definition of (standard) shadowing. The shadowing is considered, for example, in \cite{GMa}, for continuous self-maps of compact Hausdorff spaces. Let $X$ be a compact Hausdorff space. A finite {\em open cover} of $X$ is a finite set of open subsets of $X$ whose union is $X$. We denote by $\mathscr{U}(X)$ the set of finite open covers of $X$.

\begin{defi}
\normalfont
Given a continuous map $f\colon X\to X$ and $\mathcal{U}\in\mathscr{U}(X)$, a finite sequence $(x_i)_{i=0}^{k}$ of points in $X$, where $k>0$ is a positive integer, is called a {\em $\mathcal{U}$-chain} of $f$ if $\{f(x_i),x_{i+1}\}\subset U_i$ for all $0\le i\le k-1$ for some $U_i\in\mathcal{U}$. A $\mathcal{U}$-chain $(x_i)_{i=0}^{k}$ of $f$ is said to be a {\em $\mathcal{U}$-cycle} of $f$ if $x_0=x_k$. Let $\xi=(x_i)_{i\ge0}$ be a sequence of points in $X$. For $\mathcal{U}\in\mathscr{U}(X)$, $\xi$ is called a {\em $\mathcal{U}$-pseudo orbit} of $f$ if $\{f(x_i),x_{i+1}\}\subset U_i$ for all $i\ge0$ for some $U_i\in\mathcal{U}$. For $\mathcal{V}\in\mathscr{U}(X)$, $\xi$ is said to be {\em $\mathcal{V}$-shadowed} by $x\in X$ if $\{f^i(x),x_i\}\subset V_i$ for all $i\ge 0$ for some $V_i\in\mathcal{V}$. We say that $f$ has the {\em shadowing property} if for any $\mathcal{V}\in\mathscr{U}(X)$, there is $\mathcal{U}\in\mathscr{U}(X)$ such that every $\mathcal{U}$-pseudo orbit of $f$ is $\mathcal{V}$-shadowed by some point of $X$.
\end{defi}

\begin{rem}
\normalfont
The corresponding metric definitions are given in Section 2.2.
\end{rem}

The definition of the special shadowing relies on an equivalence relation, i.e., a partition of the phase space, defined for chain transitive continuous self-maps of compact Hausdorff spaces. It should be noted that the relation was introduced in \cite{Sim} and later rediscovered in \cite{RW} (based on the argument given in \cite[Exercise 8.22]{A}). The relation is generalized in, for example, \cite{AWC} for chain transitive uniformly continuous self-maps of uniform spaces. We deal with the case of compact Hausdorff spaces.

Let $X$ be a compact Hausdorff space.

\begin{defi}
\normalfont
We say that a continuous map $f\colon X\to X$ is {\em chain transitive} if for any $x,y\in X$ and $\mathcal{U}\in\mathscr{U}(X)$, there is a $\mathcal{U}$-chain $(x_i)_{i=0}^k$ of $f$ such that $x_0=x$ and $x_k=y$. 
\end{defi}

Let $f\colon X\to X$ be a chain transitive continuous map. Given any $\mathcal{U}\in\mathscr{U}(X)$, the {\em length} of a $\mathcal{U}$-cycle $(x_i)_{i=0}^k$ of $f$ is defined to be $k$. Let $m=m(\mathcal{U})>0$ be the greatest common divisor of the lengths of all $\mathcal{U}$-cycles of $f$. A relation $\sim_{f,\mathcal{U}}$ in $X^2$ is defined by: for all $x,y\in X$, $x\sim_{f,\mathcal{U}}y$ if and only if there is a $\mathcal{U}$-chain $(x_i)_{i=0}^k$ of $f$ with $x_0=x$, $x_k=y$, and $m|k$. 

\begin{rem}
\normalfont
Under the chain transitivity assumption, the following properties hold:
\begin{itemize}
\item[(P1)] $\sim_{f,\mathcal{U}}$ is an open and closed $(f\times f)$-invariant equivalence relation in $X^2=X\times X$.
\item[(P2)] Any $x,y\in X$ with $\{x,y\}\subset U$ for some $U\in\mathcal{U}$ satisfies $x\sim_{f,\mathcal{U}}y$, so for any $\mathcal{U}$-chain $(x_i)_{i=0}^k$ of $f$, we have $f(x_i)\sim_{f,\mathcal{U}}x_{i+1}$ for each $0\le i\le k-1$, implying $x_i\sim_{f,\mathcal{U}}f^i(x_0)$ for every $0\le i\le k$.
\item[(P3)] For any $x\in X$ and $n\ge0$, $x\sim_{f,\mathcal{U}}f^{mn}(x)$.
\item[(P4)] There exists $N>0$ such that for any $x,y\in X$ with $x\sim_{f,\mathcal{U}}y$ and $n\ge N$, there is a $\mathcal{U}$-chain $(x_i)_{i=0}^k$ of $f$ with $x_0=x$, $x_k=y$, and $k=mn$.
\end{itemize}

Property (P2) can be proved as follows. For any $x,y\in X$ with $\{x,y\}\subset U$ for some $U\in\mathcal{U}$, since $f$ is chain transitive and so surjective, $f(z)=x$ for some $z\in X$. The chain transitivity of $f$ gives a $\mathcal{U}$-chain $(x_i)_{i=0}^k$ of $f$ with $x_0=x$ and $x_k=z$. Then, since
\[
(z,x_0,x_1,\dots,x_k)
\]
is a $\mathcal{U}$-cycle of $f$, by the definition of $m$, we obtain $m|k+1$. Consider a $\mathcal{U}$-chain
\[
(y_0,y_1,\dots,y_n)=(x_0,x_1,\dots,x_k,y)
\]
of $f$ and note that $n=k+1$. By the definition of $\sim_{f,\mathcal{U}}$, we obtain $x\sim_{f,\mathcal{U}}y$, proving the claim.

Property (P4) is stated in \cite[Lemma 2.3]{BMR} for the metrizable case. It is, in fact, a consequence of a Schur's theorem implying that for any positive integers $m,k_1,k_2,\dots,k_n$, if
\[
\gcd(k_1,k_2,\dots,k_n)=m,
\]
then every sufficiently large multiple $M$ of $m$ can be expressed as a linear combination
\[
M=a_1k_1+a_2k_2+\cdots+a_nk_n
\]
where $a_1,a_2,\dots,a_n$ are non-negative integers. 
\end{rem}

Fix $x\in X$ and let $D_i$, $i\ge0$, denote the equivalence class of $\sim_{f,\mathcal{U}}$ including $f^i(x)$. Then, $D_m=D_0$, and $X=\bigsqcup_{i=0}^{m-1}D_i$ gives the partition of $X$ into the equivalence classes of $\sim_{f,\mathcal{U}}$. Note that every $D_i$, $0\le i\le m-1$, is a clopen subset of $X$ and satisfies $f(D_i)=D_{i+1}$. We call
\[
\mathcal{D}_f^\mathcal{U}=\{D_i\colon 0\le i\le m-1\}
\]
the {\em $\mathcal{U}$-cyclic decomposition} of $X$.

\begin{defi}
\normalfont
We define a relation $\sim_f$ in $X^2$ by: for all $x,y\in X$, $x\sim_f y$ if and only if $x\sim_{f,\mathcal{U}}y$ for every $\mathcal{U}\in\mathscr{U}(X)$. It is a closed $(f\times f)$-invariant equivalence relation in $X^2$.
\end{defi}

\begin{rem}
\normalfont
\begin{itemize}
\item[(1)] We say that $(x,y)\in X^2$ is a {\em chain proximal} pair for $f$ if for every $\mathcal{U}\in\mathscr{U}(X)$, there is a pair of $\mathcal{U}$-chains
\[
((x_i)_{i=0}^k,(y_i)_{i=0}^k)
\]
of $f$ such that $(x_0,y_0)=(x,y)$ and $x_k=y_k$.  By the $\mathcal{U}$-cyclic decomposition, we easily see that any chain proximal pair $(x,y)\in X^2$ for $f$ satisfies $x\sim_{f,\mathcal{U}}y$ for all $\mathcal{U}\in\mathscr{U}(X)$, that is, $x\sim_f y$. Conversely, by property (P4) of $\sim_{f,\mathcal{U}}$, for any $(x,y)\in X^2$, $x\sim_f y$ implies that $(x,y)$ is a chain proximal pair for $f$. In other words, for any $(x,y)\in X^2$, it holds that $x\sim_f y$ if only if $(x,y)$ is a chain proximal pair for $f$. This fact is stated in \cite{RW} for the metrizable case.
\item[(2)] We say that a continuous map $f\colon X\to X$ is {\em chain mixing} if for any $x,y\in X$ and $\mathcal{U}\in\mathscr{U}(X)$, there exists $N>0$ such that for each $k\ge N$, there is a $\mathcal{U}$-chain $(x_i)_{i=0}^k$ of $f$ with $x_0=x$ and $x_k=y$. If a continuous map $f\colon X\to X$ is chain mixing, then $f$ is chain transitive, and $m=m(\mathcal{U})=1$ for every $\mathcal{U}(X)$, so we have $x\sim_f y$ for all $x,y\in X$. Conversely, if a chain transitive map $f\colon X\to X$ satisfies $x\sim_f y$ for all $x,y\in X$, then for every $\mathcal{U}\in\mathcal{U}(X)$, by the $\mathcal{U}$-cyclic decomposition, we have $m=m(\mathcal{U})=1$, so by property (P4) of $\sim_{f,\mathcal{U}}$, $f$ is chain mixing. In other words, a continuous map $f\colon X\to X$ is chain mixing if only if $f$ is chain transitive and satisfies $x\sim_f y$ for all $x,y\in X$. 
\end{itemize}
\end{rem}

\begin{rem}
\normalfont
The corresponding metric definitions are given in Section 2.2.
\end{rem}

\begin{rem}
\normalfont
A continuous map $f\colon X\to X$ is said to be {\em transitive} (resp.\:{\em mixing}) if for any non-empty open subsets $U,V$ of $X$, it holds that $f^j(U)\cap V\ne\emptyset$ for some $j>0$ (resp.\:for all $j\ge i$ for some $i>0$). If $f$ is transitive (resp.\:mixing), then $f$ is chain transitive (resp.\:chain mixing), and the converse holds when $f$ has the shadowing property. 
\end{rem}

Let $X$ be a compact Hausdorff space and let $f\colon X\to X$ be a chain transitive continuous map. We denote by $\mathcal{D}(f)$ the set of equivalence classes of $\sim_f$. Since $\sim_f$ is a closed relation in $X^2$, every $D\in\mathcal{D}(f)$ is a closed subset of $X$. For any $\mathcal{U},\mathcal{V}\in\mathscr{U}(X)$, let $\mathcal{D}^{\mathcal{V},\mathcal{U}}(f)$ denote the set of $D\in\mathcal{D}(f)$ such that every $\mathcal{U}$-pseudo orbit $(x_i)_{i\ge0}$ of $f$ with $x_0\in D$ is $\mathcal{V}$-shadowed by some $x\in D$. We set
\[
\mathcal{D}_{\rm sh}(f)=\bigcap_{\mathcal{V}\in\mathscr{U}(X)}\bigcup_{\mathcal{U}\in\mathscr{U}(X)}\mathcal{D}^{\mathcal{V},\mathcal{U}}(f).
\]

The first result of this paper is the following.

\begin{thm}
For any chain transitive continuous map $f\colon X\to X$, the following properties are equivalent:
\begin{itemize}
\item[(1)] $\mathcal{D}_{\rm sh}(f)\ne\emptyset$,
\item[(2)] $\mathcal{D}(f)=\mathcal{D}_{\rm sh}(f)$,
\item[(3)] for any $\mathcal{V}\in\mathscr{U}(X)$, there exists $\mathcal{U}\in\mathscr{U}(X)$ such that $\mathcal{D}(f)=\mathcal{D}^{\mathcal{V},\mathcal{U}}(f)$.
\end{itemize}
\end{thm}

In view of this theorem, we shall define the special form of shadowing for chain transitive continuous self-maps of compact Hausdorff spaces.

\begin{defi}
\normalfont
We say that a chain transitive continuous map $f\colon X\to X$ has the {\em $\mathcal{D}$-shadowing property} (abbrev.\:DSP) if one of the equivalent three properties in Theorem 1.1 is satisfied. 
\end{defi}

\begin{rem}
\normalfont
By property (3) in Theorem 1.1, we see that if a chain transitive continuous map $f\colon X\to X$ has DSP, then $f$ satisfies the shadowing property.
\end{rem}

Let $X$ be a compact metric space endowed with a metric $d$. For any $\mathcal{W}\in\mathscr{U}(X)$, $\delta>0$ is called a {\em Lebesgue number} of $\mathcal{W}$ if every subset $S$ of $X$ with
\[
{\rm diam}(S)=\sup_{x,y\in S}d(x,y)\le\delta
\]
satisfies $S\subset W$ for some $W\in\mathcal{W}$. We define a number ${\rm mesh}(\mathcal{W})$ to be
\[
{\rm mesh}(\mathcal{W})=\sup_{W\in\mathcal{W}}{\rm diam}(W).
\]
For any $\delta,\epsilon>0$, let $\mathcal{D}^{\epsilon,\delta}(f)$ denote the set of $D\in\mathcal{D}(f)$ such that every $\delta$-pseudo orbit $(x_i)_{i\ge0}$ of $f$ with $x_0\in D$ is $\epsilon$-shadowed by some $x\in D$. For any $\mathcal{U},\mathcal{V}\in\mathscr{U}(X)$ and $\delta,\epsilon>0$, we see that
\begin{itemize}
\item if $\delta$ is a Lebesgue number of $\mathcal{U}$, and if ${\rm mesh}(\mathcal{V})\le\epsilon$, then $\mathcal{D}^{\epsilon,\delta}(f)\subset\mathcal{D}^{\mathcal{V},\mathcal{U}}(f)$,
\item if ${\rm mesh}(\mathcal{U})\le\delta$, and if $\epsilon$ is a Lebesgue number of $\mathcal{V}$, then $\mathcal{D}^{\mathcal{V},\mathcal{U}}(f)\subset\mathcal{D}^{\epsilon,\delta}(f)$;
\end{itemize}
therefore, by Theorem 1.1, we obtain the following corollary. 

\begin{cor}
Let $X$ be a compact metric space and let  $f\colon X\to X$ be a chain transitive continuous map. Then, it holds
\[
\mathcal{D}_{\rm sh}(f)=\bigcap_{\epsilon>0}\bigcup_{\delta>0}\mathcal{D}^{\epsilon,\delta}(f),
\]
and the following properties are equivalent:
\begin{itemize}
\item[(1)] $f$ has DSP,
\item[(2)] $\mathcal{D}_{\rm sh}(f)\ne\emptyset$,
\item[(3)] $\mathcal{D}(f)=\mathcal{D}_{\rm sh}(f)$,
\item[(4)] for any $\epsilon>0$, there exists $\delta>0$ such that $\mathcal{D}(f)=\mathcal{D}^{\epsilon,\delta}(f)$.
\end{itemize}
\end{cor}
We present three results on DSP. Let $X$ be a compact Hausdorff space and let $K(X)$ denote the set of non-empty closed subsets of $X$. Let $\mathscr{F}(X)$ be the set of finite sets of open subsets of $X$. For any $\mathcal{F}\in\mathscr{F}(X)$, we define a subset $[\mathcal{F}]$ of $K(X)$ by
\[
[\mathcal{F}]=\{A\in K(X)\colon A\subset\bigcup_{U\in\mathcal{F}}U, A\cap U\ne \emptyset,\:\forall U\in\mathcal{F}\}.
\]
Note that $[\{X\}]=K(X)$. For any $\mathcal{F},\mathcal{G}\in\mathscr{F}(X)$, letting
\[
\mathcal{H}=\{U\cap\bigcup_{V\in\mathcal{G}}V\colon U\in\mathcal{F}\}\cup\{V\cap\bigcup_{U\in\mathcal{F}}U\colon V\in\mathcal{G}\},
\]
we have $\mathcal{H}\in\mathscr{F}(X)$ and $[\mathcal{F}]\cap[\mathcal{G}]=[\mathcal{H}]$. This implies that
\[
\{[\mathcal{F}]\colon\mathcal{F}\in\mathscr{F}(X)\}
\]
forms a basis for a topology on $K(X)$. The topology is called the {\em Vietoris topology}.

\begin{rem}
\normalfont
When $X$ is a compact metric space, endowed with a metric $d$, the Vietoris topology coincides with the topology induced from the {\em Hausdorff distance} $d_H\colon K(X)\times K(X)\to[0,\infty)$; for any $A,B\in K(X)$,
\[
d_H(A,B)=\inf\{\epsilon>0\colon A\subset B_\epsilon(B), B\subset B_\epsilon(A)\},
\]
where $B_\epsilon(\cdot)$ denotes the $\epsilon$-neighborhood.
\end{rem}

One more definition is needed.

\begin{defi}
\normalfont
Let $f\colon X\to X$ be a chain transitive continuous map. For $\mathcal{U}\in\mathscr{U}(X)$, a $\mathcal{U}$-pseudo orbit $\xi=(x_i)_{i\ge0}$ of $f$ is called a {\em $\mathcal{U}$-pseudo orbit of $f$ along $\mathcal{D}(f)$} if $f(x_i)\sim_f x_{i+1}$ for all $i\ge0$. We say that $f$ has the {\em shadowing along $\mathcal{D}(f)$} if for any $\mathcal{V}\in\mathscr{U}(X)$, there is $\mathcal{U}\in\mathscr{U}(X)$ such that every $\mathcal{U}$-pseudo orbit $\xi=(x_i)_{i\ge0}$ of $f$ along $\mathcal{D}(f)$ is $\mathcal{V}$-shadowed by some $x\in X$ with $x\sim_f x_0$.
\end{defi}

Given any chain transitive continuous map $f\colon X\to X$ and $x\in X$, define $\mathcal{D}_f(x)\in\mathcal{D}(f)$ by $x\in\mathcal{D}_f(x)$. Then, the following theorem gives a characterization of DSP.

\begin{thm}
For any chain transitive continuous map $f\colon X\to X$, $f$ has DSP if and only if the following conditions are satisfied:
\begin{itemize}
\item[(1)] $f$ has the shadowing along $\mathcal{D}(f)$,
\item[(2)] $\mathcal{D}_f(\cdot)\colon X\to K(X)$ is continuous, here $K(X)$ has the Vietoris topology. 
\end{itemize}
\end{thm}

Let $X,Y$ be compact Hausdorff spaces and let $f\colon X\to X,g\colon Y\to Y$ be continuous maps. A map $\pi\colon Y\to X$ is said to be a {\em factor map} if $\pi$ is surjective and satisfies $f\circ\pi=\pi\circ g$. A factor map $\pi\colon Y\to X$ is also denoted as
\[
\pi\colon(Y,g)\to(X,f).
\]
The next theorem gives sufficient conditions for $f$ to have DSP in terms of the factor map.

\begin{thm}
Let $X$ and $Y$ be compact Hausdorff spaces. For any chain transitive continuous maps $f\colon X\to X$ and $g\colon Y\to Y$, suppose that the following conditions are satisfied:
\begin{itemize}
\item[(1)] there is a factor map $\pi\colon(Y,g)\to(X,f)$,
\item[(2)] $g$ has DSP,
\item[(3)] $f$ has the shadowing property.
\end{itemize}
Then, $f$ satisfies DSP.
\end{thm}

For any compact Hausdorff space $X$, let $\dim{X}$ denote the topological dimension of $X$. We know that $X$ satisfies $\dim{X}=0$ if and only if $X$ is totally disconnected. In Section 4, by using a recent result of Good and Meddaugh \cite{GM}, we prove the following theorem. 

\begin{thm}
Given any chain transitive continuous map $f\colon X\to X$, if $\dim{X}=0$, and if $f$ has the shadowing property, then $f$ satisfies DSP.
\end{thm}

By Theorems 1.3 and 1.4, we obtain the following corollary.

\begin{cor}
Let $X$ be a compact Hausdorff space and let $f\colon X\to X$ be a chain transitive continuous map with the shadowing property. Suppose that there are a compact Hausdorff space $Y$ and a continuous map $g\colon Y\to Y$ such that
\begin{itemize}
\item[(1)] there is a factor map $\pi\colon(Y,g)\to(X,f)$,
\item[(2)] $\dim{Y}=0$,
\item[(3)] $g$ is chain transitive and has the shadowing property.
\end{itemize}
Then, $f$ satisfies DSP.
\end{cor}

\begin{q}
\normalfont
It is natural to ask whether for any compact Hausdorff space $X$ and any chain transitive continuous map $f\colon X\to X$, if $f$ has the shadowing property, then $f$ has also DSP. In particular, can we show it by finding $g\colon Y\to Y$ and a factor map
\[
\pi\colon(Y,g)\to(X,f)
\]
such that conditions (2) and (3) in Corollary 1.2 are satisfied?   
\end{q}

Given any continuous map $f\colon X\to X$, let
\[
T(f)=\{x\in X\colon X=\omega(x,f)\},
\]
where $\omega(x,f)$ is the $\omega$-limit set of $x$ for $f$. When $X$ is a compact metric space, it holds that $f$ is transitive if and only if $T(f)\ne\emptyset$ if and only if $T(f)$ is a dense $G_\delta$-subset of $X$. The next statement is essentially proved in Section 5 of \cite{Ka4}.

\begin{thm}
Let $X$ be a compact metric space and let $f\colon X\to X$ be a chain transitive continuous map. If $f$ satisfies DSP and $h_{\rm top}(f)>0$, then there exists a sequence of positive numbers $(\delta_n)_{n\ge2}$ such that every $D\in\mathcal{D}(f)$ contains a dense Mycielski subset $S$ which is included in $T(f)$ and is distributionally $n$-$\delta_n$-scrambled for all $n\ge2$.
\end{thm}

\begin{rem}
\normalfont
If the answer to Question 1.1 is positive, then by Theorem 1.5, for any continuous self-map of a compact metric space, the transitivity, the shadowing, and positive topological entropy imply a strong DC1, which means a positive answer to the question by Li, Li, and Tu in \cite{LLT} mentioned above.
\end{rem}

Next, for application of DSP, we consider a statistical variant of $\omega$-chaos. The notion of $\omega$-chaos was introduced by S.H.\:Li \cite{SHLi} as a variation of Li--Yorke chaos. In \cite{SHLi}, it is proved for interval maps that $\omega$-chaos is equivalent to positive topological entropy. This equivalence does not hold in general. In \cite{D,DO,K,L2,Pik}, the relationships between $\omega$, Li--Yorke, and distributional chaos are examined in detail. In \cite{Bo,L1,SSt}, for transitive interval maps, the size of $\omega$-scrambled sets is discussed with respect to the Lebesgue measure. In \cite{HR,LO,WWH}, specification properties are used to study $\omega$-chaos. In \cite{O}, it is shown that any non-minimal transitive map with the two-sided limit shadowing exhibits $\omega$-chaos in a strong sense. Moreover, in \cite{DLO}, a stronger version of $\omega$-chaos is considered. 

Based on a statistical modification of the definition of $\omega$-limit set, we shall introduce a notion of {\em $\overline{\omega}$}-chaos. In the rest of this section, $X$ denotes a compact metric space endowed with a metric $d$. For any $p\in X$ and $\epsilon>0$, we denote by $B_\epsilon(p)$ the open ball of radius $\epsilon$ centered at $X$:
\[
B_\epsilon(p)=\{q\in X\colon d(p,q)<\epsilon\}.
\]
For any subset $A$ of $\mathbb{N}_0=\{0\}\cup\mathbb{N}=\{0,1,2,\dots\}$, we denote by $\overline{d}(A)$ (resp.\:$\underline{d}(A)$) the {\em upper density} (resp.\:{\em lower density}) of $A$:
\[
\overline{d}(A)=\limsup_{n\to\infty}\frac{1}{n}|A\cap\{0,1,\dots,n-1\}|
\]
and
\[
\underline{d}(A)=\liminf_{n\to\infty}\frac{1}{n}|A\cap\{0,1,\dots,n-1\}|,
\]
here $|\cdot|$ is the cardinality of the set. If $\overline{d}(A)=\underline{d}(A)$, we set $d(A)=\overline{d}(A)=\underline{d}(A)$.

\begin{defi}
\normalfont
For any continuous map $f\colon X\to X$ and $x\in X$, we say that $y\in X$ is an {\em $\overline{\omega}$-limit point} of $x$ if
\[
\overline{d}(\{i\ge0\colon f^i(x)\in B_\epsilon(y)\})>0
\]
for all $\epsilon>0$. We denote by $\overline{\omega}(x,f)$ the set of $\overline{\omega}$-limit points of $x$.
\end{defi}

By this definition, we see that $\overline{\omega}(x,f)$ is a closed $f$-invariant subset of $X$.

\begin{rem}
\normalfont
The $\overline{\omega}$-limit set $\overline{\omega}(x,f)$ was introduced by Hilmy \cite{H} as the {\em minimal center of attraction of $x\in X$} (see also \cite{Sig}). It is treated in \cite{A} as $\omega_{\#}f(x)$ (see page 163 of \cite{A}). The same idea is also used in the definition of statistical limit set by Il'yashenko (see page 148 of \cite{AAIS}). It also appears in the recent paper \cite{CT} where the notation $\omega_{\overline{d}}(x)$ is used.
\end{rem}

For any continuous map $f\colon X\to X$ and $x\in X$, the $\omega$-limit set $\omega(x,f)$ is defined by the set of $y\in X$ such that $\lim_{j\to\infty}f^{i_j}(x)=y$ for some sequence $0\le i_1<i_2<\cdots$. On the other hand, for any $x,y\in X$, it holds that $y\notin\overline{\omega}(x,f)$ if and only if
\[
\lim_{n\to\infty}\frac{1}{n}|\{0\le i\le n-1\colon f^i(x)\in B_\epsilon(y)\}|=0
\]
for some $\epsilon>0$. Therefore, $\overline{\omega}(x,f)$ can be obtained by discarding $y$ from $\omega(x,f)$ such that the upper density of the approach time to some neighborhood of $y$ is $0$. As shown in Section 5, $\overline{\omega}(x,f)$ is always {\em non-empty} and shares basic properties with $\omega(x,f)$ (see Corollary 5.1 and Lemma 5.2 in Section 5).

Here, we recall the definition of $\omega$-chaos.

\begin{defi}
\normalfont
Let $f\colon X\to X$ be a  continuous map. We say that a subset $S$ of $X$ is {\em $\omega$-scrambled} if for any $x,y\in S$ with $x\ne y$, the following conditions are satisfied:
\begin{itemize}
\item[(1)] $\omega(x,f)\setminus\omega(y,f)$ is an uncountable set,
\item[(2)] $\omega(x,f)\cap\omega(y,f)\ne\emptyset$,
\item[(3)] $\omega(x,f)\setminus Per(f)\ne\emptyset$, here $Per(f)$ is the set of periodic points for $f$:
\[
Per(f)=\bigcup_{i>0}\{x\in X\colon f^i(x)=x\}.
\]
\end{itemize}
We say that $f$ exhibits {\em $\omega$-chaos} (or $f$ is {\em $\omega$-chaotic}) if there is an uncountable $\omega$-scrambled subset of $X$.
\end{defi}

As a modification of $\omega$-chaos, we define $\overline{\omega}$-chaos.

\begin{defi}
\normalfont
Let $f\colon X\to X$ be a  continuous map. We say that a subset $S$ of $X$ is {\em $\overline{\omega}$-scrambled} if for any $x,y\in S$ with $x\ne y$, the following conditions are satisfied:
\begin{itemize}
\item[(1)] $\overline{\omega}(x,f)\setminus\overline{\omega}(y,f)$ is an uncountable set,
\item[(2)] $\overline{\omega}(x,f)\cap\overline{\omega}(y,f)\ne\emptyset$,
\item[(3)] $\overline{\omega}(x,f)\setminus Per(f)\ne\emptyset$.
\end{itemize}
We say that $f$ exhibits {\em $\overline{\omega}$-chaos} (or $f$ is {\em $\overline{\omega}$-chaotic}) if there is an uncountable $\overline{\omega}$-scrambled subset of $X$.
\end{defi}

The next theorem gives sufficient conditions for $\overline{\omega}$-chaos. A  subset $S$ of $X$ is said to be {\em $\mathfrak{c}$-dense} if for any non-empty open subset $U$ of $X$, the cardinality $|U\cap S|$ of $U\cap S$ satisfies $|U\cap S|=|\mathbb{R}|$.

\begin{thm}
Let $f\colon X\to X$ be a chain transitive continuous map. If $f$ satisfies DSP and $h_{\rm top}(f)>0$, then $X$ contains an $\overline{\omega}$-scrambled subset $S$ such that $D\cap S$ is $\mathfrak{c}$-dense in $D$ for all $D\in\mathcal{D}(f)$, which implies that $S$ is $\mathfrak{c}$-dense in $X$, thus in particular, $f$ exhibits $\overline{\omega}$-chaos.
\end{thm}

For any complete metric space $Y$, a subset $S$ of $Y$ is said to be a {\em residual} subset of $Y$ if $S$ contains a countable intersection of open and dense subsets of $Y$.

\begin{rem}
\normalfont
Given any continuous map $f\colon X\to X$, let
\[
\overline{T}(f)=\{x\in X\colon X=\overline{\omega}(x,f)\}.
\]
Then, it is obvious that any $\overline{\omega}$-scrambled subset $S$ of $X$ satisfies $S\cap\overline{T}(f)=\emptyset$. If a continuous map $f\colon X\to X$ is chain transitive and has DSP, then by Lemma 6.5 of Section 6,
$D\cap\overline{T}(f)$ is residual in $D$ for all $D\in\mathcal{D}(f)$, thus in particular, $\overline{T}(f)\ne\emptyset$, and so by Lemma 5.3 of Section 5, $\overline{T}(f)$ is residual in $X$. As a consequence, for any $\overline{\omega}$-scrambled subset $S$ of $X$, $D\cap S$ must be meager in $D$ for all $D\in\mathcal{D}(f)$, and $S$ must be meager in $X$.
\end{rem}

The proof of Theorem 1.6 in Section 6 is through the following lemma which may be of independent interest.

\begin{lem}
Let $f\colon X\to X$ be a chain transitive continuous map with DSP. Suppose that a sequence $(x_i)_{i\ge0}$ of points in $X$ satisfies $f(x_i)\sim_f x_{i+1}$ for all $i\ge0$, and
\[
\lim_{n\to\infty}\frac{1}{n}\sum_{i=0}^{n-1}d(f(x_i),x_{i+1})=0.
\]
Then, for any $\epsilon>0$, there is $x\in X$ such that $d(x,x_0)<\epsilon$, $x\sim_f x_0$, and 
\[
\lim_{n\to\infty}\frac{1}{n}\sum_{i=0}^{n-1}d(f^i(x),x_i)=0.
\]
\end{lem}

\begin{rem}
\normalfont
We say that a continuous map $f\colon X\to X$ has the {\em asymptotic average shadowing property} (abbrev.\:AASP) if for any sequence $(x_i)_{i\ge0}$ of points in $X$ such that
\[
\lim_{n\to\infty}\frac{1}{n}\sum_{i=0}^{n-1}d(f(x_i),x_{i+1})=0,
\]
we have
\[
\lim_{n\to\infty}\frac{1}{n}\sum_{i=0}^{n-1}d(f^i(x),x_i)=0
\]
for some $x\in X$. According to Theorem 3.8 of \cite{KKO}, we know that for any continuous map $f\colon X\to X$,
\begin{itemize}
\item[(1)] if $f$ is surjective and has AASP, then $f$ is chain mixing, 
\item[(2)] if $f$ is chain mixing and has the shadowing property, then $f$ satisfies AASP. 
\end{itemize}
Recall that a continuous map $f\colon X\to X$ is chain mixing if and only if $f$ is chain transitive and satisfies $x\sim_f y$ for all $x,y\in X$. By this, we see that the shadowing given in Lemma 1.1 generalizes AASP (under the assumption of surjectivity), and Lemma 1.1 generalizes (2).
\end{rem}

By Theorems 1.4 and 1.6, we obtain the following theorem.

\begin{thm}
Let $f\colon X\to X$ be a chain transitive continuous map. If $\dim{X}=0$, $f$ has the shadowing property, and if $h_{\rm top}(f)>0$, then $f$ exhibits $\overline{\omega}$-chaos.
\end{thm}

Given any continuous map $f\colon X\to X$, the chain recurrent set $CR(f)$ is partitioned into the disjoint union of chain components, i.e., maximal chain transitive subsets, for $f$. If $\dim{CR(f)}=0$, $f$ has the shadowing property, and $h_{\rm top}(f)>0$, then by Lemma 4.1 of \cite{Ka4}, there is a chain component $C$ for $f$ such that $\dim{C}=0$, $f|_C$ has the shadowing property, and $h_{\rm top}(f|_{C})>0$. Thus, by applying Theorem 1.7 to $f|_C$, we obtain the following corollary.

\begin{cor}
For any continuous map $f\colon X\to X$, if $\dim{CR(f)}=0$, $f$ has the shadowing property, and if $h_{\rm top}(f)>0$, then $f$ exhibits $\overline{\omega}$-chaos.
\end{cor}

We present one more corollary. Let $M$ be a closed differentiable manifold and let $\mathcal{C}(M)$ (resp.\:$\mathcal{H}(M)$) denote the set of continuous self-maps (resp.\:homeomorphisms) of $M$, endowed with the $C^0$-topology. Because generic $f\in\mathcal{C}(M)$ (resp.\:$f\in\mathcal{H}(M)$, if $\dim{M}>1$) satisfies the assumptions of Corollary 1.3 (see Section 1 of \cite{Ka4} for details), we obtain the following.

\begin{cor}
Generic $f\in\mathcal{C}(M)$ (resp.\:$f\in\mathcal{H}(M)$, if $\dim{M}>1$) exhibits $\overline{\omega}$-chaos.
\end{cor}

\begin{rem}
\normalfont
In \cite{FKOT}, an example is given of a subshift $(X,f)$ such that
\begin{itemize}
\item[(1)] $f$ is $\omega$-chaotic, 
\item[(2)] every $x\in X$ is a generic point for some ergodic $f$-invariant measure on $X$,
\item[(3)] any ergodic $f$-invariant measure on $X$ is supported by a single periodic orbit for $f$.
\end{itemize}
For this $f\colon X\to X$, by (2), (3), and Lemma 5.1 in Section 5, we obtain $\overline{\omega}(x,f)\subset Per(f)$ for all $x\in X$, so $f$ does not exhibit $\overline{\omega}$-chaos. From (1), it follows that $\omega$-chaos is not a sufficient condition for $\overline{\omega}$-chaos. Conversely, it is natural to expect that $\overline{\omega}$-chaos is not a sufficient condition $\omega$-chaos. For a transitive continuous map $f\colon X\to X$, $(X,f)$ is said to be a {\em ToM system} if $f$ is not minimal, and every $x\in X$ is either a transitive point for $f$ or a minimal point for $f$ \cite{Ko}. We easily see that such $f$ cannot be $\omega$-chaotic. Then, is there an $\overline{\omega}$-chaotic map $f\colon X\to X$ such that $(X,f)$ is ToM? If exist, $\overline{\omega}$-chaos is not a sufficient condition for $\omega$-chaos. In Appendix A, we prove that an interval map $f\colon[0,1]\to[0,1]$ is $\overline{\omega}$-chaotic if and only if $h_{\rm top}(f)>0$.
\end{rem}

For another application of DSP, we consider the distribution of {\em irregular points} (also called as {\em points with historical behaviour} \cite{R}). The irregular points are defined to be points for which Birkhoff averages do not converge. Among many studies on the irregular points, in \cite{LW}, it is proved that for any continuous map $f\colon X\to X$ with the specification property, the $\phi$-irregular set is either empty or residual in $X$ for each continuous function $\phi\colon X\to\mathbb{R}$. It is proved in \cite{DOT} that  if a continuous map satisfies the shadowing, the set of irregular points has full entropy. Also, in \cite{FKOT}, under the assumption of shadowing, $\phi$-irregular and $\phi$-level sets are studied in detail.

Let $\mathcal{M}(X)$ denote the set of Borel probability measures on $X$. Note that $\mathcal{M}(X)$ is a compact metrizable space when equipped with the $\text{weak}^\ast$-topology. For any $p\in X$, we denote by $\delta_p$ the Dirac measure at $p$: for any Borel subset $A$ of $X$,
\begin{equation*}
\delta_p(A)=
\begin{cases}
0&\text{if $p\not\in A$}\\
1&\text{if $p\in A$}
\end{cases}
.
\end{equation*}
We denote by $C(X)$ the set of real-valued continuous functions on $X$. Given any continuous map $f\colon X\to X$, let
\[
I(f,\phi)=\{x\in X\colon\lim_{n\to\infty}\frac{1}{n}\sum_{i=0}^{n-1}\phi(f^i(x))\text{ does not exist}\}
\]
for all $\phi\in C(X)$, and let
\[
I(f)=\bigcup_{\phi\in C(X)}I(f,\phi).
\]
Then, it holds that
\[
I(f)=\{x\in X\colon\lim_{n\to\infty}\frac{1}{n}\sum_{i=0}^{n-1}\delta_{f^i(x)}\text{ does not exist}\},
\]
and each $x\in I(f)$ is said to be an {\em irregular point} for $f$.

The following two theorems describe topological distribution of irregular points.

\begin{thm}
Given any chain transitive continuous map $f\colon X\to X$, if $f$ satisfies DSP and $h_{\rm top}(f)>0$, then there exists $\phi\in C(X)$ such that for every $D\in\mathcal{D}(f)$, $D\cap I(f,\phi)$ is a residual subset of $D$. 
\end{thm}

\begin{thm}
Given any chain transitive continuous map $f\colon X\to X$, if $f$ satisfies DSP, then for any $\phi\in C(X)$, $I(f,\phi)\ne\emptyset$ implies that for every $D\in\mathcal{D}(f)$, $D\cap I(f,\phi)$ is a residual subset of $D$. 
\end{thm}

For any continuous map $f\colon X\to X$, let
\[
\overline{T}(f)=\{x\in X\colon X=\overline{\omega}(x,f)\}
\]
as in Remark 1.10. By Theorems 1.8, 1.9, and Lemma 6.5 in Section 6, we obtain the following result which improves Theorem 1.5.

\begin{thm}
Let $f\colon X\to X$ be a chain transitive continuous map. If $f$ satisfies DSP and $h_{\rm top}(f)>0$, then there exists $\phi\in C(X)$ such that $I(f,\phi)\ne\emptyset$, and there is a sequence of positive numbers $(\delta_n)_{n\ge2}$ such that for any $\psi\in C(X)$ with $I(f,\psi)\ne\emptyset$, every $D\in\mathcal{D}(f)$ contains a dense Mycielski subset $S$ which is included in $\overline{T}(f)\cap I(f,\psi)$ and is distributionally $n$-$\delta_n$-scrambled for all $n\ge2$.
\end{thm}

This paper consists of seven sections. The basic notations, definitions and facts are briefly collected in Section 2. In Section 3, we prove Theorems 1.1, 1.2, and 1.3. In Section 4, we prove Theorem 1.4. In Section 5, we recall some facts about $\overline{\omega}$-limit sets and prove some basic properties of them. Theorem 1.6 is proved in Section 6. In Section 7, we prove Theorems 1.8, 1.9, and 1.10. In Appendix A, we show that for interval maps, $\overline{\omega}$-chaos is equivalent to positive topological entropy. In Appendix B, we discuss properties of $\overline{\omega}$-limit sets other than those discussed in Section 5.

\section{Preliminaries}

In this section, we collect some basic definitions, notations, facts, and prove some lemmas which will be used in the sequel.

\subsection{Chain components}

Let $X$ be a compact Hausdorff space and let $f\colon X\to X$ be a continuous map. For any $x,y\in X$ and $\mathcal{U}\in\mathscr{U}(X)$, the notation $x\rightarrow_{f,\mathcal{U}}y$ means that there is a $\mathcal{U}$-chain $(x_i)_{i=0}^k$ of $f$ with $x_0=x$ and $x_k=y$. We write $x\rightarrow_f y$ if $x\rightarrow_{f,\mathcal{U}}y$ for all $\mathcal{U}\in\mathscr{U}(X)$. We say that $x\in X$ is a {\em chain recurrent point} for $f$ if $x\rightarrow_f x$. Let $CR(f)$ denote the set of chain recurrent points for $f$. We define a relation $\leftrightarrow_f$ in $CR(f)^2$ by: for any $x,y\in CR(f)$, $x\leftrightarrow_f y$ if and only if $x\rightarrow_f y$ and $y\rightarrow_f x$. It follows that $\leftrightarrow_f$ is a closed $(f\times f)$-invariant equivalence relation in $CR(f)^2$. An equivalence class $C$ of $\leftrightarrow_f$ is called a {\em chain component} for $f$. We denote by $\mathcal{C}(f)$ the set of chain components for $f$. Then, the following properties hold:
\begin{itemize}
\item[(1)] $CR(f)=\bigsqcup_{C\in\mathcal{C}(f)}C$,
\item[(2)] Every $C\in\mathcal{C}(f)$ is a closed $f$-invariant subset of $CR(f)$,
\item[(3)] $f|_C\colon C\to C$ is chain transitive for all $C\in\mathcal{C}(f)$.
\end{itemize}
Note that $f$ is chain transitive if and only if $f$ satisfies $X=CR(f)$ and $\mathcal{C}(f)=\{X\}$.

\subsection{Metric definitions}

Throughout this subsection, $X$ denotes a compact metric space endowed with a metric $d$.

\subsubsection{{\it Metric definitions of chains, cycles, pseudo-orbits and the shadowing property}}

Given a continuous map $f\colon X\to X$ and $\delta>0$, a finite sequence $(x_i)_{i=0}^{k}$ of points in $X$, where $k>0$ is a positive integer, is called a {\em $\delta$-chain} of $f$ if $d(f(x_i),x_{i+1})\le\delta$ for every $0\le i\le k-1$. A $\delta$-chain $(x_i)_{i=0}^{k}$ of $f$ is said to be a {\em $\delta$-cycle} of $f$ if $x_0=x_k$. Let $\xi=(x_i)_{i\ge0}$ be a sequence of points in $X$. For $\delta>0$, $\xi$ is called a {\em $\delta$-pseudo orbit} of $f$ if $d(f(x_i),x_{i+1})\le\delta$ for all $i\ge0$. For $\epsilon>0$, $\xi$ is said to be {\em $\epsilon$-shadowed} by $x\in X$ if $d(f^i(x),x_i)\leq \epsilon$ for all $i\ge 0$. We say here that $f$ has the {\em metric shadowing property} if for any $\epsilon>0$, there is $\delta>0$ such that every $\delta$-pseudo orbit of $f$ is $\epsilon$-shadowed by some point of $X$. For completeness, we prove the following lemma.

\begin{lem}
For any continuous map $f\colon X\to X$, $f$ has the shadowing property (in the sense of Definition 1.1) if only if $f$ has the metric shadowing property.
\end{lem}

\begin{proof}
Assume that $f$ has the shadowing property (in the sense of Definition 1.1), and for any $\epsilon>0$, take $\mathcal{V}\in\mathscr{U}(X)$ with ${\rm mesh}(\mathcal{V})\le\epsilon$. Then, there is $\mathcal{U}\in\mathscr{U}(X)$ such that every $\mathcal{U}$-pseudo orbit of $f$ is $\mathcal{V}$-shadowed by some point of $X$. For such $\mathcal{U}$, by taking a Lebesgue number $\delta>0$ of $\mathcal{U}$, we see that every $\delta$-pseudo orbit $\xi$ of $f$ is $\epsilon$-shadowed by some $x\in X$. Since $\epsilon>0$ is arbitrary,  we conclude that $f$ has the metric shadowing property.

Conversely, assume that $f$ has the metric shadowing property. For any $\mathcal{V}\in\mathscr{U}(X)$, take a Lebesgue number $\epsilon>0$ of $\mathcal{V}$ and $\delta>0$ such that every $\delta$-pseudo orbit of $f$ is $\epsilon$-shadowed by some point of $X$. Then, by taking $\mathcal{U}\in\mathscr{U}(X)$ with ${\rm mesh}(\mathcal{U})\le\delta$, we see that every $\mathcal{U}$-pseudo orbit $\xi$ of $f$ is $\mathcal{V}$-shadowed by some $x\in X$. Since $\mathcal{V}\in\mathscr{U}(X)$ is arbitrary, we conclude that $f$ has the shadowing property (in the sense of Definition 1.1), completing the proof.
\end{proof}

\subsubsection{{\it Metric definition of the relation $\sim_f$}}
Given any continuous map $f\colon X\to X$, it holds that, for any $\mathcal{U}\in\mathscr{U}(X)$ and $\delta>0$,
\begin{itemize}
\item if ${\rm mesh}(\mathcal{U})\le\delta$, then every $\mathcal{U}$-chain of $f$ is a $\delta$-chain of $f$,
\item if $\delta$ is a Lebesgue number of $\mathcal{U}$, then every $\delta$-chain of $f$ is a $\mathcal{U}$-chain of $f$.
\end{itemize}
For any $x,y\in X$ and $\delta>0$, we write $x\rightarrow_{f,\delta}y$ if there is a $\delta$-chain $(x_i)_{i=0}^k$ of $f$ with $x_0=x$ and $x_k=y$. Then, it is obvious that $x\rightarrow_f y$ (in the sense of Section 2.1) if only if $x\rightarrow_{f,\delta}y$ for all $\delta>0$. In particular, $f$ is chain transitive (in the sense of Definition 1.2) if and only if $x\rightarrow_{f,\delta}y$ for all $x,y\in X$ and $\delta>0$. 

Let $f\colon X\to X$ be a chain transitive continuous map. Given any $\delta>0$, the {\em length} of a $\delta$-cycle $(x_i)_{i=0}^k$ of $f$ is defined to be $k$. Let $m=m(\delta)>0$ be the greatest common divisor of the lengths of all $\delta$-cycles of $f$. A relation $\sim_{f,\delta}$ in $X^2$ is defined by: for all $x,y\in X$, $x\sim_{f,\delta}y$ if and only if there is a $\delta$-chain $(x_i)_{i=0}^k$ of $f$ with $x_0=x$, $x_k=y$, and $m|k$. Similarly as $\sim_{f,\mathcal{U}}$, $\mathcal{U}\in\mathscr{U}(X)$, the following properties hold:
\begin{itemize}
\item[(P1)] $\sim_{f,\delta}$ is an open and closed $(f\times f)$-invariant equivalence relation in $X^2$,
\item[(P2)] any $x,y\in X$ with $d(x,y)\le\delta$ satisfies $x\sim_{f,\delta}y$, so for any $\delta$-chain $(x_i)_{i=0}^k$ of $f$, we have $f(x_i)\sim_{f,\delta}x_{i+1}$ for each $0\le i\le k-1$, implying $x_i\sim_{f,\delta}f^i(x_0)$ for every $0\le i\le k$,
\item[(P3)] for any $x\in X$ and $n\ge0$, $x\sim_{f,\delta}f^{mn}(x)$,
\item[(P4)] there exists $N>0$ such that for any $x,y\in X$ with $x\sim_{f,\delta}y$ and $n\ge N$, there is a $\mathcal{U}$-chain $(x_i)_{i=0}^k$ of $f$ with $x_0=x$, $x_k=y$, and $k=mn$.
\end{itemize}

\begin{rem}
\normalfont
Since it holds that, for any $\mathcal{U}\in\mathscr{U}(X)$ and $\delta>0$,
\begin{itemize}
\item if ${\rm mesh}(\mathcal{U})\le\delta$, then every $\mathcal{U}$-cycle of $f$ is a $\delta$-cycle of $f$, so $m(\delta)|m(\mathcal{U})$, and for all $x,y\in X$, $x\sim_{f,\mathcal{U}}y$ implies $x\sim_{f,\delta}y$,
\item if $\delta$ is a Lebesgue number of $\mathcal{U}$, then every $\delta$-cycle of $f$ is a $\mathcal{U}$-cycle of $f$,  so $m(\mathcal{U})|m(\delta)$, and for all $x,y\in X$, $x\sim_{f,\delta}y$ implies $x\sim_{f,\mathcal{U}}y$,   
\end{itemize}
we see that for all $x,y\in X$, $x\sim_f y$ (in the sense of Definition 1.3) if only if $x\sim_{f,\delta}y$ for every $\delta>0$.
\end{rem}

\begin{rem}
\normalfont
\begin{itemize}
\item[(1)] For any $(x,y)\in X^2$, $(x,y)$ is a chain proximal pair for $f$ (in the sense of Remark 1.3 (1)) if and only if for every $\delta>0$, there is a pair of $\delta$-chains
\[
((x_i)_{i=0}^k,(y_i)_{i=0}^k)
\]
of $f$ such that $(x_0,y_0)=(x,y)$ and $x_k=y_k$.
\item[(2)] A continuous map $f\colon X\to X$ is chain mixing (in the sense of Remark 1.3 (2)) if and only if for any $x,y\in X$ and $\delta>0$, there exists $N>0$ such that for each $k\ge N$, there is a $\delta$-chain $(x_i)_{i=0}^k$ of $f$ with $x_0=x$ and $x_k=y$.
\end{itemize}
\end{rem}

\subsection{Net}

A non-empty set $I$ equipped with a preorder $\le$ is said to be a {\em directed set} when for any $a,b\in I$, there is $c\in I$ such that $a\le c$ and $b\le c$. Let $Y$ be a topological space. For any directed set $I$, a map $x_\bullet\colon I\to Y$ is called a {\em net} in $Y$ and denoted as $x_\bullet=(x_a)_{a\in I}$. For a net $x_\bullet=(x_a)_{a\in I}$ in $Y$ and $x\in Y$, we say that $x_\bullet$ {\em converges} to $x$ if for any open neighborhood $U$ of $x$, there is $a\in I$ such that $x_b\in U$ for all $b\in I$ with $a\le b$, and we denote it by $x_\bullet\to x$. For any two nets $x_\bullet=(x_a)_{a\in I}$ and $y_\bullet=(y_b)_{b\in J}$ in $Y$, $y_\bullet$ is said to be a {\em subnet} of $x_\bullet$ if there is a map $h\colon J\to I$ such that (i) $y_b=x_{h(b)}$ for each $b\in J$, (ii) $h(b)\le h(c)$ for all $b,c\in J$ with $b\le c$, and (iii) for any $a\in I$, $a\le h(b)$ holds for some $b\in J$. We know that $Y$ is compact if only if every net in $Y$ has a convergent subnet in $Y$ (see, e.g. \cite{Wil}).

\subsection{Equivariance, factor and the topological conjugacy}

Let $X$ and $Y$ be compact Hausdorff spaces. Given any two continuous maps $f\colon X\to X$ and $g\colon Y\to Y$, a continuous map $\pi\colon X\to Y$ is said to be {\em equivariant} if $\pi\circ f=g\circ\pi$, and such $\pi$ is also denoted as
\[
\pi\colon(X,f)\to(Y,g).
\]
An equivariant map $\pi\colon(X,f)\to(Y,g)$ is called a {\em factor map} (resp.\:{\em topological conjugacy}) if it is surjective (resp.\:a homeomorphism). Two systems $(X,f)$ and $(Y,g)$ are said to be {\em topologically conjugate} if there is a topological conjugacy $h\colon(X,f)\to(Y,g)$.

\subsection{Inverse limit}

Let $I$ be a directed set. For any continuous self-maps $f_\lambda\colon X_\lambda\to X_\lambda$ of compact Hausdorff spaces $X_\lambda$, $\lambda\in I$, we say that a family of equivariant maps
\[
\pi=(\pi_\lambda^\mu\colon(X_\mu,f_\mu)\to(X_\lambda,f_\lambda))_{\lambda\le\mu\in I}
\]
is an {\em inverse system} of equivariant maps if the following conditions are satisfied:
\begin{itemize}
\item[(1)] $\pi_\lambda^\lambda=id_{X_\lambda}$ for every $\lambda\in I$,
\item[(2)] $\pi_\lambda^\nu=\pi_\lambda^\mu\circ\pi_\mu^\nu$ for all $\lambda\le\mu\le\nu\in I$.
\end{itemize}

Let
\[
\pi=(\pi_\lambda^\mu\colon(X_\mu,f_\mu)\to(X_\lambda,f_\lambda))_{\lambda\le\mu\in I}
\]
be an inverse system of equivariant maps. We define the {\em inverse limit space} $X=\varprojlim X_\lambda$ by 
\[
X=\{x=(x_\lambda)_{\lambda\in I}\in\prod_{\lambda\in I}X_\lambda\colon\pi_\lambda^\mu(x_\mu)=x_\lambda,\forall\lambda\le\mu\in I\},
\]
which is a compact Hausdorff space as a closed subspace of
\[
\prod_{\lambda\in I}X_\lambda
\]
with the product topology. Then, a continuous map $f\colon X\to X$ can be well-defined by $f(x)=(f_\lambda(x_\lambda))_{\lambda\in I}$ for all $x=(x_\lambda)_{\lambda\in I}\in X$. By this, we define the {\em inverse limit system}
\[
(X,f)=\varprojlim(X_\lambda,f_\lambda).
\]
For each $\lambda\in I$, we define a map $\pi_\lambda\colon X\to X_\lambda$ by $\pi_\lambda(x)=x_\lambda$ for all $x=(x_\lambda)_{\lambda\in I}\in X$, which gives an equivariant map
\[
\pi_\lambda\colon(X,f)\to(X_\lambda,f_\lambda).
\]
By compactness, we can show that
\[
\pi_\lambda(X)=\bigcap_{\lambda\le\mu}\pi_\lambda^\mu(X_\mu)
\]
for every $\lambda\in I$. Following \cite{GM}, we say that $\pi$ satisfies the {\em Mittag-Leffler condition} (abbrev.\:MLC) if for any $\lambda\in I$, there is $\lambda\le\mu$ such that $\pi_\lambda^\mu(X_\mu)=\pi_\lambda^\nu(X_\nu)$ for all $\mu\le\nu$, or equivalently, $\pi_\lambda(X)=\pi_\lambda^\mu(X_\mu)$.

\begin{lem}
Let
\[
\pi=(\pi_\lambda^\mu\colon(X_\mu,f_\mu)\to(X_\lambda,f_\lambda))_{\lambda\le\mu\in I}
\]
be an inverse system of equivariant maps such that $f_\lambda\colon X_\lambda\to X_\lambda$ is chain transitive for all $\lambda\in I$. Let $(X,f)=\varprojlim(X_\lambda,f_\lambda)$. Then, $f\colon X\to X$ is chain transitive.
\end{lem}

\begin{proof}
For any $\lambda\le\mu\in I$, since $\pi_\lambda=\pi_\lambda^\mu\circ\pi_\mu$,  $\pi_\lambda^\mu|_{\pi_\mu(X)}\colon\pi_\mu(X)\to\pi_\lambda(X)$ gives a factor map
\[
\pi_\lambda^\mu|_{\pi_\mu(X)}\colon(\pi_\mu(X),f_{\mu}|_{\pi_\mu(X)})\to(\pi_\lambda(X),f_{\lambda}|_{\pi_\lambda(X)}).
\]
We consider an inverse system of factor maps
\[
\hat{\pi}=(\pi_\lambda^\mu|_{\pi_\mu(X)}\colon(\pi_\mu(X),f_{\mu}|_{\pi_\mu(X)})\to(\pi_\lambda(X),f_{\lambda}|_{\pi_\lambda(X)}))_{\lambda\le\mu\in I}
\]
and let
\[
(Y,g)=\varprojlim(\pi_\lambda(X),f_{\lambda}|_{\pi_\lambda(X)}).
\]
Since
\[
Y=\{x=(x_\lambda)_{\lambda\in I}\in\prod_{\lambda\in I}\pi_\lambda(X_\lambda)\colon\pi_\lambda^\mu(x_\mu)=x_\lambda,\forall\lambda\le\mu\in I\}=X,
\]
it is obvious that $(Y,g)$ is topologically conjugate to $(X,f)$. On the other hand, for every $\lambda\in I$, since $\pi_\lambda\colon(X,f)\to(X_\lambda,f_\lambda)$ is an equivariant map, by the chain transitivity of $f_\lambda$, $f_\lambda|_{\pi_\lambda(X)}\colon\pi_\lambda(X)\to\pi_\lambda(X)$ is chain transitive. Since $\hat{\pi}$ is an inverse system of factor maps, it follows that $g$ is chain transitive and so is $f$.
\end{proof}

\subsection{Subshifts}

\subsubsection{{\it Subshifts of finite type}}

Let $S$ be a finite set with the discrete topology. The {\em shift map} $\sigma\colon S^\mathbb{N}\to S^\mathbb{N}$ is defined by $\sigma(x)=(x_{n+1})_{n\ge1}$ for all $x=(x_n)_{n\ge1}\in S^\mathbb{N}$. Note that $\sigma$ is continuous with respect to the product topology of $S^\mathbb{N}$. A closed $\sigma$-invariant subset $X$ of $S^\mathbb{N}$ (and also the subsystem $(X,\sigma|_X)$ of $(S^\mathbb{N},\sigma)$) is called a {\em subshift}. A subshift $X$ of $S^\mathbb{N}$ (and also $(X,\sigma|_X)$ of $(S^\mathbb{N},\sigma)$) is called a {\em subshift of finite type} (abbrev.\:SFT) if there are $N\ge1$ and $F\subset S^{N+1}$ such that for any $x=(x_n)_{n\ge1}\in S^\mathbb{N}$, $x\in X$ if and only if $(x_i,x_{i+1},\dots,x_{i+N})\in F$ for all $i\ge1$. The shift map $\sigma\colon S^\mathbb{N}\to S^\mathbb{N}$ is positively expansive and has the shadowing property. We know that a subshift $X$ of $S^\mathbb{N}$ is of finite type if and only if $\sigma|_X\colon X\to X$ has the shadowing property \cite{AH}.

\subsubsection{{\it Some properties of SFTs}}

Let $(X,\sigma|_X)$ be an SFT and put $f=\sigma|_X$. By Theorem 3.4.4 of \cite{AH}, $\mathcal{C}(f)$ coincides with the finite set of {\em basic sets}, and so $CR(f)$ is decomposed into the finite disjoint union of chain components. Consider the case where $f$ is transitive (or, $(X,f)$ is a transitive SFT). Then, we have $X=CR(f)$ and $\mathcal{C}(f)=\{X\}$. Again by Theorem 3.4.4 of \cite{AH}, $X$ admits a decomposition
\[
X=\bigsqcup_{i=0}^{m-1}f^i(D),
\]
where $m>0$ is a positive integer, such that $f^i(D)$, $0\le i\le m-1$, are clopen $f^m$-invariant subsets of $X$, and
\[
f^m|_{f^i(D)}\colon f^i(D)\to f^i(D)
\]
is mixing for every $0\le i\le m-1$. In this case, we easily see that $\mathcal{D}(f)=\{f^i(D)\colon0\le i\le m-1\}$.

\section{Proofs of Theorems 1.1, 1.2 and 1.3}

In this section, we prove Theorems 1.1, 1.2, and 1.3. A finite {\em open cover} of $X$ is a finite set of open subsets of $X$ whose union is $X$. We denote by $\mathscr{U}(X)$ the set of finite open covers of $X$. We need two lemmas both of which are elementary.

\begin{lem}
For any $\mathcal{U}\in\mathscr{U}(X)$, there is $\mathcal{V}\in\mathscr{U}(X)$ such that for all $m\ge1$ and $V_i\in\mathcal{V}$, $1\le i\le m$, $\bigcap_{i=1}^m V_i\ne\emptyset$ implies $\bigcup_{i=1}^m V_i\subset U$ for some $U\in\mathcal{U}$.
\end{lem}

\begin{proof}
Take $\mathcal{W}\in\mathscr{U}(X)$ such that $\overline{W}\subset U_W$ for all $W\in\mathcal{W}$ for some $U_W\in\mathcal{U}$. For each $W\in\mathcal{W}$, define $\mathcal{V}_W\in\mathscr{U}(X)$ by $\mathcal{V}_W=\{U_W,X\setminus\overline{W}\}$. Let
\[
\mathcal{V}=\bigvee_{W\in\mathcal{W}}\mathcal{V}_W=\{\bigcap_{W\in\mathcal{W}}A_W\colon A_W\in\mathcal{V}_W,\:\forall W\in\mathcal{W}\}
\]
and note that $\mathcal{V}\in\mathscr{U}(X)$. For any $m\ge1$ and $V_i=\bigcap_{W\in\mathcal{W}}A_W^{(i)}\in\mathcal{V}$, $1\le i\le m$, here $A_W^{(i)}\in\mathcal{V}_W$ for each $W\in\mathcal{W}$ and each $1\le i\le m$, assume that $\bigcap_{i=1}^m V_i\ne\emptyset$. If $A_W^{(i_W)}=X\setminus\overline{W}$ for all $W\in\mathcal{W}$ for some $1\le i_W\le m$, then
\[
\bigcap_{i=1}^m V_i=\bigcap_{W\in\mathcal{W}}\bigcap_{i=1}^m A_W^{(i)}\subset\bigcap_{W\in\mathcal{W}}A_W^{(i_W)}=\bigcap_{W\in\mathcal{W}}[X\setminus\overline{W}]=\emptyset,
\]
which is a contradiction. It follows that $A_{W_0}^{(i)}=U_{W_0}$ for some $W_0\in\mathcal{W}$ for all $1\le i\le m$, and so
\[
\bigcup_{i=1}^m V_i=\bigcup_{i=1}^m\bigcap_{W\in\mathcal{W}}A_W^{(i)}\subset\bigcup_{i=1}^m A_{W_0}^{(i)}=\bigcup_{i=1}^m U_{W_0}=U_{W_0}\in\mathcal{U}.
\]
This shows that $\mathcal{V}$ has the desired property, and thus the lemma has been proved.
\end{proof}

For any $\mathcal{U},\mathcal{V}\in\mathscr{U}(X)$, the notation $\mathcal{U}\prec\mathcal{V}$ means that for every $V\in\mathcal{V}$, there is $U\in\mathcal{U}$ such that $V\subset U$. For any $\mathcal{U}\in\mathscr{U}(X)$, define $\mathcal{U}^{(n)}\in\mathscr{U}(X)$, $n\ge2$, by
\[
\mathcal{U}^{(n)}=\{\bigcup_{i=1}^n U_i\colon U_i\in\mathcal{U},\:1\le\forall i\le n,\:U_i\cap U_{i+1}\ne\emptyset,1\le\forall i\le n-1\}.
\] 

\begin{lem}
For any $\mathcal{U}\in\mathscr{U}(X)$ and $n\ge2$, there is $\mathcal{V}_n\in\mathscr{U}(X)$ such that $\mathcal{U}\prec\mathcal{V}^{(n)}_n$.
\end{lem}

\begin{proof}
The proof is by induction on $n$. First, by Lemma 3.1, there is $\mathcal{V}_2\in\mathscr{U}(X)$ such that $\mathcal{U}\prec\mathcal{V}^{(2)}_2$. Assume that $\mathcal{U}\prec\mathcal{V}^{(n)}_n$ for some $n\ge2$ and some $\mathcal{V}_n\in\mathscr{U}(X)$. Then, again by Lemma 3.1, there is $\mathcal{V}_{n+1}\in\mathscr{U}(X)$ such that $\mathcal{V}_n\prec\mathcal{V}^{(2)}_{n+1}$. We shall show that $\mathcal{U}\prec\mathcal{V}^{(n+1)}_{n+1}$. For any $V_i\in\mathcal{V}_{n+1}$, $1\le i\le n+1$, suppose that $V_i\cap V_{i+1}\ne\emptyset$ for all $1\le i\le n$. Since $\mathcal{V}_n\prec\mathcal{V}^{(2)}_{n+1}$, we have $V_i\cup V_{i+1}\subset U_i$ for all $1\le i\le n$ for some $U_i\in\mathcal{V}_n$.
Then, because
\[
V_{i+1}\subset[V_i\cup V_{i+1}]\cap[V_{i+1}\cup V_{i+2}]\subset U_i\cap U_{i+1},
\]
it holds that $U_i\cap U_{i+1}\ne\emptyset$ for every $1\le i\le n-1$, implying $\bigcup_{i=1}^n U_i\in\mathcal{V}^{(n)}_n$. By $\mathcal{U}\prec\mathcal{V}^{(n)}_n$, we obtain $\bigcup_{i=1}^n U_i\subset U$ for some $U\in\mathcal{U}$. It follows that
\[
\bigcup_{i=1}^{n+1} V_i\subset\bigcup_{i=1}^n U_i\subset U.
\]
This shows that  $\mathcal{U}\prec\mathcal{V}^{(n+1)}_{n+1}$, and thus the induction is complete.
\end{proof}

We shall prove Theorem 1.1.

\begin{proof}[Proof of Theorem 1.1]
The implications $(3)\Rightarrow(2)\Rightarrow(1)$ are clear from the definitions. We shall prove $(1)\Rightarrow(3)$. Let  $\mathcal{V}\in\mathscr{U}(X)$. By using Lemmas 3.1 and 3.2, take $\mathcal{R},\mathcal{W},\mathcal{T}\in\mathscr{U}(X)$ with the following properties:
\begin{itemize}
\item $\overline{R}\subset V_R$ for all $R\in\mathcal{R}$ for some $V_R\in\mathcal{V}$,
\item  for all $m\ge1$ and $W_i\in\mathcal{W}$, $1\le i\le m$, $\bigcap_{i=1}^m W_i\ne\emptyset$ implies $\bigcup_{i=1}^m W_i\subset R$ for some $R\in\mathcal{R}$,
\item $\mathcal{W}\prec\mathcal{T}^{(2)}$.
\end{itemize}
Fix $D_0\in\mathcal{D}_{\rm sh}(f)$ and $p\in D_0$. By the definition of $\mathcal{D}_{\rm sh}(f)$, we have $D_0\in\mathcal{D}^{\mathcal{T},\mathcal{S}}(f)$ for some $\mathcal{S}\in\mathscr{U}(X)$. Take $\mathcal{U}\in\mathscr{U}(X)$ such that $\overline{U}\subset S_U$ for all $U\in\mathcal{U}$ for some $S_U\in\mathcal{S}$. Given any $D\in\mathcal{D}(f)$, we shall prove $D\in\mathcal{D}^{\mathcal{V},\mathcal{U}}(f)$. Let $\mathcal{N}(D)$ denote the set of open neighborhoods of $D$ and let $\xi=(x_i)_{i\ge0}$ be a $\mathcal{U}$-pseudo orbit of $f$ with $x_0\in D$. For any $A\in\mathcal{N}(D)$, we can take an open neighborhood $B$ of $x_0$ with the following properties:
\begin{itemize}
\item $(y,x_1,x_2,\dots)$ is an $\mathcal{S}$-pseudo orbit of $f$ for any $y\in B$,
\item $\mathcal{D}_f(y)\subset A$ for all $y\in B$,
\item $B\subset T_B$ for some $T_B\in\mathcal{T}$.
\end{itemize}
Choose an open neighborhood $C$ of $x_0$ with $\overline{C}\subset B$ and let
\[
\mathcal{Q}=\{B,X\setminus\overline{C}\}\in\mathscr{U}(X).
\]
Since $D_0\in\mathcal{D}_{\rm sh}(f)$, we have $D_0\in\mathcal{D}^{\mathcal{Q},\mathcal{P}}(f)$ for some $\mathcal{P}\in\mathscr{U}(X)$. The chain transitivity of $f$ gives a $\mathcal{P}$-chain $(y_i)_{i=0}^k$ of $f$ with $y_0=p$ and $y_k=x_0$. By $D_0\in\mathcal{D}^{\mathcal{Q},\mathcal{P}}(f)$ and $y_0=p\in D_0$, we obtain $q\in D_0$ such that $\{f^i(q),y_i\}\subset Q_i$ for all $0\le i\le k$ for some $Q_i\in\mathcal{Q}$. In particular, this with $y_k=x_0\in\overline{C}$ implies $f^k(q)\in B$. Consider the following sequence:
\[
\xi'=(z_i)_{i\ge0}=(q,f(q),\dots,f^k(q),x_1,x_2,\dots).
\]
Since $f^k(q)\in B$, by the choice of $B$, we see that $\xi'$ is an $\mathcal{S}$-pseudo orbit of $f$ with $z_0=q\in D_0$. Then, because $D_0\in\mathcal{D}^{\mathcal{T},\mathcal{S}}(f)$, $\xi'$ is $\mathcal{T}$-shadowed by some $z\in X$ with $z\sim_f q$. Since $\sim_f$ is $(f\times f)$-invariant, $z\sim_f q$ implies $f^k(z)\sim_f f^k(q)$, that is, $f^k(z)\in\mathcal{D}_f(f^k(q))$. Again by $f^k(q)\in B$ and the choice of $B$, we obtain $f^k(z)\in A$. On the other hand, we have $\{f^k(z),f^k(q)\}\subset T_0$ and $\{f^k(q),x_0\}\subset B\subset T_B$ for some $T_0,T_B\in\mathcal{T}$. From $\mathcal{W}\prec\mathcal{T}^{(2)}$, it follows that $\{f^k(z),x_0\}\subset W_0$ for some $W_0\in\mathcal{W}$. Moreover, since $\{f^{k+i}(z),x_i\}\subset T_i$ for all $i\ge1$ for some $T_i\in\mathcal{T}$, again by $\mathcal{W}\prec\mathcal{T}^{(2)}$, we obtain $\{f^{k+i}(z),x_i\}\subset W_i$ for all $i\ge1$ for some $W_i\in\mathcal{W}$. In other words, $\xi$ is $\mathcal{W}$-shadowed by $f^k(z)\in A$. Note that $\mathcal{N}(D)$ is a directed set with respect to the order $\le$ defined by $A_0\le A_1$ if and only if $A_1\subset A_0$ for all $A_0,A_1\in\mathcal{N}(D)$. As is already proved,  for any $A\in\mathcal{N}(D)$, $\xi$ is $\mathcal{W}$-shadowed by some $x_A\in A$. By compactness of $X$, we have a subnet $y_\bullet=(y_{A'})_{A'\in I}$ of $x_\bullet=(x_A)_{A\in\mathcal{N}(D)}$ such that $y_\bullet\to x$ for some $x\in X$. For such $x$, it holds that $x\in D$. Moreover, for each $i\ge0$, letting
\[
W'_i=\bigcup\{W\in\mathcal{W}\colon x_i\in W\},
\]  
by the choice of $\mathcal{W}$ and $\mathcal{R}$, we obtain
\[
x_i\in W'_i\subset R_i\subset\overline{R_i}\subset V_i
\]
for some $R_i\in\mathcal{R}$ and some $V_i\in\mathcal{V}$. For any $A'\in I$, because $\xi$ is $\mathcal{W}$-shadowed by $y_{A'}$, we have $f^i(y_{A'})\in W'_i\subset\overline{R_i}$ for each $i\ge0$. By continuity of $f$, we obtain $f^i(x)\in\overline{R_i}\subset V_i$ for all $i\ge0$. This means that $\xi$ is $\mathcal{V}$-shadowed by $x\in D$. Since $\xi$ is arbitrary, we conclude that $D\in\mathcal{D}^{\mathcal{V},\mathcal{U}}(f)$, completing the proof.
\end{proof}

Next, we prove Theorem 1.2.

\begin{proof}[Proof of Theorem 1.2]
Assume that $f$ satisfies the conditions (1) and (2). We shall prove that $f$ has DSP. First, we  prove that for every $\mathcal{T}\in\mathscr{U}(X)$, there is $\mathcal{U}\in\mathscr{U}(X)$ such that the following property holds:
\begin{itemize}
\item[(P)] for any $x\in X$ and $y\in\mathcal{D}_f^{\mathcal{U}}(x)$, there are $z\in\mathcal{D}_f(x)$ and $T\in\mathcal{T}$ such that $\{y,z\}\subset T$.
\end{itemize}
Let $\mathcal{T}\in\mathscr{U}(X)$. For any $x\in X$, let
\[
\mathcal{T}_x=\{T\in\mathcal{T}\colon\mathcal{D}_f(x)\cap T\ne\emptyset\}
\]
and note that $\mathcal{D}_f(x)\in[\mathcal{T}_x]$. By the assumption (2), we can take an open neighborhood $U_x$ of $\mathcal{D}_f(x)$ such that $\mathcal{D}_f(y)\in[\mathcal{T}_x]$ for all $y\in U_x$. Then, we have $\mathcal{D}_f^{\mathcal{U}_x}(x)\subset U_x$ for some $\mathcal{U}_x\in\mathscr{U}(X)$. By compactness of $X$, there are $n\ge1$ and $x_i\in X$, $1\le i\le n$, such that $X=\bigcup_{i=1}^n\mathcal{D}_f^{\mathcal{U}_{x_i}}(x_i)$. Define $\mathcal{U}\in\mathscr{U}(X)$ by
\[
\mathcal{U}=\bigvee_{i=1}^n\mathcal{U}_{x_i}=\{\bigcap_{i=1}^n A_i\colon A_i\in\mathcal{U}_{x_i},\:1\le\forall i\le n\}
\]
and note that $\mathcal{U}_{x_i}\prec\mathcal{U}$ for all $1\le i\le n$. For any $x\in X$, take $1\le i\le n$ such that $x\in\mathcal{D}_f^{\mathcal{U}_{x_i}}(x_i)$. Then, for every $y\in\mathcal{D}_f^\mathcal{U}(x)$, we have
\[
y\in\mathcal{D}_f^\mathcal{U}(x)\subset\mathcal{D}_f^{\mathcal{U}_{x_i}}(x)=\mathcal{D}_f^{\mathcal{U}_{x_i}}(x_i)\subset U_{x_i}, 
\]
so $\mathcal{D}_f(y)\in[\mathcal{T}_{x_i}]$. Because $x\in D_f^\mathcal{U}(x)$, in particular, we have $\mathcal{D}_f(x)\in[\mathcal{T}_{x_i}]$. By $y\in\mathcal{D}_f(y)$ and $\mathcal{D}_f(y)\in[\mathcal{T}_{x_i}]$, we obtain $y\in T$ for some $T\in\mathcal{T}_{x_i}$. For such $T$, by $\mathcal{D}_f(x)\in[\mathcal{T}_{x_i}]$, we obtain some $z\in\mathcal{D}_f(x)$ with $z\in T$ and so $\{y,z\}\subset T$. Since $x\in X$ and $y\in\mathcal{D}_f^\mathcal{U}(x)$ are arbitrary, this proves the claim.

We continue the proof. Given any $\mathcal{V}\in\mathscr{U}(X)$, by Lemma 3.2, we have $\mathcal{V}\prec\mathcal{R}^{(2)}$ for some $\mathcal{R}\in\mathscr{U}(X)$. For such $\mathcal{R}$, due to the assumption (1), there is $\mathcal{S}\in\mathscr{U}(X)$ such that every $\mathcal{S}$-pseudo orbit $\xi'=(y_i)_{i\ge0}$ of $f$ along $\mathcal{D}(f)$ is $\mathcal{R}$-shadowed by some $x\in X$ with $x\sim_f y_0$. For such $\mathcal{S}$, by Lemma 3.2, we have $\mathcal{S}\prec\mathcal{W}^{(3)}$ for some  $\mathcal{W}\in\mathscr{U}(X)$. Fix $\mathcal{T}\in\mathscr{U}(X)$ such that $\mathcal{R}\prec\mathcal{T}$, $\mathcal{W}\prec\mathcal{T}$, and $f(T)\subset W_{T}$ for all $T\in\mathcal{T}$ for some $W_T\in\mathcal{W}$. Take $\mathcal{U}\in\mathscr{U}(X)$ with the above property P and also $\mathcal{W}\prec\mathcal{U}$.  We shall prove $\mathcal{D}(f)=\mathcal{D}^{\mathcal{V},\mathcal{U}}(f)$. Let $D\in\mathcal{D}(f)$ and let $\xi=(x_i)_{i\ge0}$ be a $\mathcal{U}$-pseudo orbit of $f$ with $x_0\in D$. Since $\xi$ is a $\mathcal{U}$-pseudo orbit of $f$, we have $x_i\in\mathcal{D}_f^{\mathcal{U}}(f^i(x_0))$ for all $i\ge0$. For each $i\ge0$, by using the property P, we obtain $y_i\in\mathcal{D}_f(f^i(x_0))$ such that $\{x_i,y_i\}\subset T_i$ for some $T_i\in\mathcal{T}$. Then, by the choice of $\mathcal{T}$, it holds that $\{x_{i+1},y_{i+1}\}\subset W_i$ and $\{f(x_i),f(y_i)\}\subset W'_i$ for all $i\ge0$ for some $W_i,W'_i\in\mathcal{W}$. On the other hand, since $\xi$ is a $\mathcal{U}$-pseudo orbit of $f$, by $\mathcal{W}\prec\mathcal{U}$, we have $\{f(x_i),x_{i+1}\}\subset U_i\subset W''_i$ for all $i\ge0$ for some $U_i\in\mathcal{U},W''_i\in\mathcal{W}$. From $\mathcal{S}\prec\mathcal{W}^{(3)}$, it follows that $\{f(y_i),y_{i+1}\}\subset S_i$ for all $i\ge0$ for some $S_i\in\mathcal{S}$. This with $y_i\in\mathcal{D}_f(f^i(x_0))$, $i\ge0$, implies that $\xi'=(y_i)_{i\ge0}$ is an $\mathcal{S}$-pseudo orbit of $f$ along $\mathcal{D}(f)$ and so is $\mathcal{R}$-shadowed by some $x\in X$ with $x\sim_f y_0$. For such $x$, by $x_0\sim_f y_0$, we obtain $x\sim_f x_0$. Since $\xi'$ is $\mathcal{R}$-shadowed by $x$, we have $\{f^i(x),y_i\}\subset R_i$ for all $i\ge0$ for some $R_i\in\mathcal{R}$. Note that for each $i\ge0$, because $\mathcal{R}\prec\mathcal{T}$ and $\{x_i,y_i\}\subset T_i$, we have $\{x_i,y_i\}\subset R'_i$ for some $R'_i\in\mathcal{R}$. By $\mathcal{V}\prec\mathcal{R}^{(2)}$, we conclude that $\{f^i(x),x_i\}\subset V_i$ for all $i\ge0$ for some $V_i\in\mathcal{V}$, that is, $\xi$ is $\mathcal{V}$-shadowed by $x$ with $x\sim_f x_0$. Since $\xi$ is arbitrary, this implies $D\in\mathcal{D}^{\mathcal{V},\mathcal{U}}(f)$ and so $\mathcal{D}(f)=\mathcal{D}^{\mathcal{V},\mathcal{U}}(f)$. Since $\mathcal{V}\in\mathscr{U}(X)$ is arbitrary, $f$ satisfies DSP.

Conversely, assume that $f$ has DSP. Then, it is clear from the definition that $f$ has the shadowing along $\mathcal{D}(f)$. Let us prove that $\mathcal{D}_f(\cdot)\colon X\to K(X)$ is continuous. Let $x\in X$ and $\mathcal{D}_f(x)\in[\mathcal{F}]$ for some $\mathcal{F}\in\mathscr{F}(X)$. Then, since $\mathcal{D}_f(x)\subset\bigcup_{U\in\mathcal{F}}U$, we have $\mathcal{D}_f^\mathcal{W}(x)\subset\bigcup_{U\in\mathcal{F}}U$ for some $\mathcal{W}\in\mathscr{U}(X)$. This implies
\[
\mathcal{D}_f(y)\subset\mathcal{D}_f^\mathcal{W}(y)=\mathcal{D}_f^\mathcal{W}(x)\subset\bigcup_{U\in\mathcal{F}}U
\]
for all $y\in\mathcal{D}_f^\mathcal{W}(x)$. We shall show that for any $U\in\mathcal{F}$, there is $\mathcal{U}_U\in\mathscr{U}(X)$ such that $\mathcal{D}_f(y)\cap U\ne\emptyset$ for all $y\in\mathcal{D}_f^{\mathcal{U}_U}(x)$. This implies $\mathcal{D}_f(y)\in[\mathcal{F}]$ for any
\[
y\in\mathcal{D}_f^\mathcal{W}(x)\cap\bigcap_{U\in\mathcal{F}}\mathcal{D}_f^{\mathcal{U}_U}(x),
\]
and so, since $\mathcal{F}\in\mathscr{F}(X)$ is arbitrary, $\mathcal{D}_f(\cdot)\colon X\to K(X)$ is continuous at $x\in X$. Assume the contrary, that is, there is $U_0\in\mathcal{F}$ such that $\mathcal{D}_f(y)\cap U_0=\emptyset$ for all $\mathcal{U}\in\mathscr{U}(X)$ for some $y\in\mathcal{D}_f^\mathcal{U}(x)$. By $U_0\in\mathcal{F}$ and $\mathcal{D}_f(x)\in[\mathcal{F}]$, we have $\mathcal{D}_f(x)\cap U_0\ne\emptyset$. Fix $p\in\mathcal{D}_f(x)\cap U_0$ and take open neighborhoods $B$, $C$ of $p$ such that $\overline{B}\subset C\subset\overline{C}\subset U_0$. Define $\mathcal{V}\in\mathscr{U}(X)$ by $\mathcal{V}=\{C,X\setminus\overline{B}\}$. Since $f$ has DSP, we have $\mathcal{D}(f)=\mathcal{D}^{\mathcal{V},\mathcal{U}}(f)$ for some $\mathcal{U}\in\mathscr{U}(X)$. For such $\mathcal{U}$, by the assumption, we obtain $y\in\mathcal{D}_f^\mathcal{U}(x)$ with $\mathcal{D}_f(y)\cap U_0=\emptyset$, implying $\mathcal{D}_f(y)\subset X\setminus\overline{C}$. Since $y\in\mathcal{D}_f^\mathcal{U}(x)$ and $p\in\mathcal{D}_f(x)\subset\mathcal{D}_f^\mathcal{U}(x)$, by property (P4) of $\sim_{f,\mathcal{U}}$, there are some $k>0$ and $\mathcal{U}$-chains $\alpha=(x_i)_{i=0}^k$, $\beta=(y_i)_{i=0}^k$ of $f$ with $x_0=y$ and $x_k=y_0=y_k=p$. Consider the sequence
\[
\xi=(z_i)_{i\ge0}=\alpha\beta\beta\cdots
\]
which is a $\mathcal{U}$-pseudo orbit of $f$ with $z_0=y$. From $\mathcal{D}(f)=\mathcal{D}^{\mathcal{V},\mathcal{U}}(f)$, it follows that $\xi$ is $\mathcal{V}$-shadowed by some $z\in\mathcal{D}_f(y)$. Then, for any $n\ge1$, since $z_{kn}=p$ and $p\in\overline{B}$, we have $f^{kn}(z)\in C$. On the other hand, since $X\setminus\overline{C}$ is an open neighborhood of $\mathcal{D}_f(y)$, there are some $\mathcal{U}_y\in\mathscr{U}(X)$ and $m>0$ such that, letting $V=\mathcal{D}_f^{\mathcal{U}_y}(y)$, we have $\mathcal{D}_f(y)\subset V\subset X\setminus\overline{C}$ and $f^m(V)=V$. By $z\in\mathcal{D}_f(y)$, we obtain $f^{km}(z)\in f^{km}(V)=V\subset X\setminus\overline{C}$, which implies $C\cap[X\setminus\overline{C}]\ne\emptyset$, a contradiction. Thus, the claim has been proved. Since $x\in X$ is arbitrary, we conclude that $\mathcal{D}_f(\cdot)\colon X\to K(X)$ is continuous, completing the proof.
\end{proof}

Finally, we prove Theorem 1.3.

\begin{proof}[Proof of Theorem 1.3]
We shall prove that for any $D\in\mathcal{D}(f)$ and $\mathcal{V}\in\mathscr{U}(X)$, there is $\mathcal{U}\in\mathscr{U}(X)$ such that $D\in\mathcal{D}^{\mathcal{V},\mathcal{U}}(f)$. This implies $\mathcal{D}(f)=\mathcal{D}_{\rm sh}(f)$ and so that $f$ has DSP. Let $D\in\mathcal{D}(f)$ and $\mathcal{V}\in\mathscr{U}(X)$. By Lemma 3.2, we have $\mathcal{V}\prec\mathcal{W}^{(2)}$ for some $\mathcal{W}\in\mathscr{U}(X)$. Take $\mathcal{T}\in\mathscr{U}(Y)$ such that $\pi(T)\subset W_T$ for all $T\in\mathcal{T}$ for some $W_T\in\mathcal{W}$. Since $g$ has DSP, there is $\mathcal{S}\in\mathscr{U}(Y)$ such that $\mathcal{D}(g)=\mathcal{D}^{\mathcal{T},\mathcal{S}}(g)$. Take an open neighborhood $V$ of $\pi^{-1}(D)$ such that the following property hold:
\begin{itemize}
\item[(P)] for any $p\in V$, there is $q\in\pi^{-1}(D)$ such that $(q,g(p),g^2(p),\dots)$ is an $\mathcal{S}$-pseudo orbit of $g$, and $\{\pi(q),\pi(p)\}\subset W$ for some $W\in\mathcal{W}$.
\end{itemize}
Fix an open neighborhood $U$ of $D$ such that $\pi^{-1}(U)\subset V$. Since $f$ has the shadowing property, there is $\mathcal{U}\in\mathscr{U}(X)$ such that every $\mathcal{U}$-pseudo orbit $\xi=(x_i)_{i\ge0}$ of $f$ with $x_0\in D$ is $\mathcal{W}$-shadowed by some $x\in U$. Let us prove that such $\xi$ is also $\mathcal{V}$-shadowed by some $z\in D$. Since $\pi^{-1}(U)\subset V$, we have $\pi(p)=x$ for some $p\in V$. For such $p$, by the above property P, we have $q\in\pi^{-1}(D)$ such that
\[
\xi'=(y_i)_{i\ge0}=(q,g(p),g^2(p),\dots)
\]
is an $\mathcal{S}$-pseudo orbit of $g$, and $\{\pi(q),\pi(p)\}\subset W$ for some $W\in\mathcal{W}$. Then, because $\mathcal{D}(g)=\mathcal{D}^{\mathcal{T},\mathcal{S}}(g)$, $\xi'$ is $\mathcal{T}$-shadowed by some $r\in\mathcal{D}_g(q)$. Let $z=\pi(r)$ and note that
\[
z=\pi(r)\in\mathcal{D}_f(\pi(q))=D.
\]
For every $i\ge1$, since $\xi'$ is $\mathcal{T}$-shadowed by $r$, we have $\{g^i(r),g^i(p)\}\subset T_i$ for some $T_i\in\mathcal{T}$, and so, by the choice of $\mathcal{T}$, it holds that
\[
\{f^i(z),f^i(x)\}=\{f^i(\pi(r)),f^i(\pi(p))\}=\{\pi(g^i(r)),\pi(g^i(p))\}\subset W_i
\] 
for some $W_i\in\mathcal{W}$. Since $\xi$ is $\mathcal{W}$-shadowed by $x$, for any $i\ge1$, we have $\{f^i(x),x_i\}\subset W'_i$ for some $W'_i\in\mathcal{W}$. From $\mathcal{V}\prec\mathcal{W}^{(2)}$, it follows that $\{f^i(z),x_i\}\subset V_i$ for all $i\ge1$ for some $V_i\in\mathcal{V}$. Again, since $\xi'$ is $\mathcal{T}$-shadowed by $r$, we have $\{r,q\}\subset T_0$ for some $T_0\in\mathcal{T}$, and so, by the choice of $\mathcal{T}$, it holds that  
\[
\{z,\pi(q)\}=\{\pi(r),\pi(q)\}\subset W_0
\] 
for some $W_0\subset\mathcal{W}$. On the other hand, we have
\[
\{\pi(q),x\}=\{\pi(q),\pi(p)\}\subset W
\]
for some $W\in\mathcal{W}$. By $\mathcal{V}\prec\mathcal{W}^{(2)}$, we obtain $\{z,x\}\subset V_0$ for some $V_0\in\mathcal{V}$. In other words, $\xi$ is $\mathcal{V}$-shadowed by $z\in D$. Since $\xi$ is arbitrary, we conclude that $D\in\mathcal{D}^{\mathcal{V},\mathcal{U}}(f)$, proving the theorem.  
\end{proof}

\section{Proof of Theorem 1.4}

In this section, we prove Theorem 1.4. Terms related to inverse limit are defined in Section 2.5. For the proof, we need several lemmas.

\begin{lem}
Let $\pi=(\pi_\lambda^\mu\colon(X_\mu,f_\mu)\to(X_\lambda,f_\lambda))_{\lambda\le\mu\in I}$ be an inverse system of equivariant maps with the MLC, and let $(X,f)=\varprojlim(X_\lambda,f_\lambda)$. If $f$ is chain transitive, then there are $C_\lambda\in\mathcal{C}(f_\lambda)$, $\lambda\in I$, such that the following properties hold:
\begin{itemize}
\item[(1)] $\pi_\lambda^\mu(C_\mu)\subset C_{\lambda}$ for all $\lambda\le\mu\in I$,
\item[(2)] the inverse system of equivariant maps
\[
\pi'=(\pi_\lambda^\mu|_{C_\mu}\colon(C_\mu,f_\mu|_{C_\mu})\to(C_\lambda,f_\lambda|_{C_\lambda}))_{\lambda\le\mu\in I}
\]
satisfies the MLC,
\item[(3)] $(X,f)=\varprojlim(C_\lambda,f_\lambda|_{C_\lambda})$.
\end{itemize}
\end{lem}

\begin{proof}
Fix $x=(x_\lambda)_{\lambda\in I}\in X$ and note that for any $y=(y_\lambda)_{\lambda\in I}\in X$, since $f$ is chain transitive, we have $x\rightarrow_{f} y$. Then, for every $\lambda\in I$, since
\[
\pi_\lambda\colon(X,f)\to(X_\lambda,f_\lambda)
\]
is an equivariant map, it holds that $\pi_\lambda(x)\rightarrow_{f_\lambda}\pi_\lambda(y)$, that is, $x_\lambda\rightarrow_{f_\lambda}y_\lambda$. This particularly implies $x_\lambda\rightarrow_{f_\lambda}x_\lambda$ for all $\lambda\in I$. Take $C_\lambda\in\mathcal{C}(f_\lambda)$, $\lambda\in I$, such that $x_\lambda\in C_\lambda$ for all $\lambda\in I$. For any $\lambda\le\mu\in I$, by $x_\mu\in C_\mu$ and $\pi_\lambda^\mu(x_\mu)=x_\lambda\in C_\lambda$, we obtain $\pi_\lambda^\mu(C_\mu)\subset C_\lambda$. Consider the inverse system of equivarinat maps:
\[
\pi'=(\pi_\lambda^\mu|_{C_\mu}\colon(C_\mu,f_\mu|_{C_\mu})\to(C_\lambda,f_\lambda|_{C_\lambda}))_{\lambda\le\mu\in I}.
\]
For any $y=(y_\lambda)_{\lambda\in I}\in X$ and $\lambda\in I$, because $x_\lambda\rightarrow_{f_\lambda}y_\lambda$,
we have $y_\lambda\in C_\lambda$. This implies
\[
X=\{y=(y_\lambda)_{\lambda\in I}\in\prod_{\lambda\in I}C_\lambda\colon\pi_\lambda^\mu(y_\mu)=y_\lambda,\:\forall\lambda\le\mu\in I\}=\varprojlim C_\lambda
\]
and so
\[
(X,f)=\varprojlim(C_\lambda,f_\lambda|_{C_\lambda}).
\]
We shall prove that $\pi'$ satisfies the MLC. For any $\lambda\in I$, since $\pi$ satisfies the MLC, there is $\lambda\le\mu$ such that $\pi_\lambda^\mu(X_\mu)=\pi_\lambda^\nu(X_\nu)$ for all $\mu\le\nu$. Then, it holds that $\pi_\lambda(X)=\pi_\lambda^\mu(X_\mu)$. For any $y_\lambda\in\pi_\lambda^\mu(C_\mu)$, by
\[
\pi_\lambda^\mu(C_\mu)\subset\pi_\lambda^\mu(X_\mu)=\pi_\lambda(X),
\]
we obtain $z=(z_\lambda)_{\lambda\in I}\in X$ such that $z_\lambda=\pi_\lambda(z)=y_\lambda$. Then, for any $\mu\le\nu$, as is already proved, we have $z_\nu\in C_\nu$ and so
\[
y_\lambda=z_\lambda=\pi_\lambda^\nu(z_\nu)\in\pi_\lambda^\nu(C_\nu).
\]
Since $y_\lambda\in\pi_\lambda^\mu(C_\mu)$ is arbitrary, we obtain $\pi_\lambda^\mu(C_\mu)\subset\pi_\lambda^\nu(C_\nu)$, implying $\pi_\lambda^\mu(C_\mu)=\pi_\lambda^\nu(C_\nu)$. Since $\lambda\in I$ is arbitrary, this shows that $\pi'$ satisfies the MLC, and thus the lemma has been proved.
\end{proof}

Let $\pi=(\pi_\lambda^\mu\colon(X_\mu,f_\mu)\to(X_\lambda,f_\lambda))_{\lambda\le\mu\in I}$ be an inverse system of equivariant maps such that $f_\lambda\colon X_\lambda\to X_\lambda$ is chain transitive for all $\lambda\in I$. Let $(X,f)=\varprojlim(X_\lambda,f_\lambda)$ and note that by Lemma 2.2, $f\colon X\to X$ is chain transitive. For any $\lambda\le\mu\in I$ and $D_\mu\in\mathcal{D}(f_\mu)$, since $\pi_\lambda^\mu\colon(X_\mu,f_\mu)\to(X_\lambda,f_\lambda)$ is an equivariant map, we have $\pi_\lambda^\mu(D_\mu)\subset D_\lambda$ for some $D_\lambda\in\mathcal{D}(f_\lambda)$. Let
\[
\mathcal{D}_\pi=\{D_\ast=(D_\lambda)_{\lambda\in I}\in\prod_{\lambda\in I}\mathcal{D}(f_\lambda)\colon \pi_\lambda^\mu(D_\mu)\subset D_\lambda,\:\forall\lambda\le\mu\in I\}.
\]
Also, for any $D_\ast=(D_\lambda)_{\lambda\in I}\in\mathcal{D}_\pi$, let
\[
[D_\ast]=\{x=(x_\lambda)_{\lambda\in I}\in X\colon x_\lambda\in D_\lambda,\forall\lambda\in I\}.
\]

\begin{lem}
If $\pi$ satisfies the MLC, then it holds that $\mathcal{D}(f)=\{[D_\ast]\colon D_\ast\in\mathcal{D}_\pi\}$.
\end{lem}

\begin{proof}
First, we shall prove that for any $D_\ast=(D_\lambda)_{\lambda\in I}\in\mathcal{D}_\pi$ and $x=(x_\lambda)_{\lambda\in I},y=(y_\lambda)_{\lambda\in I}\in[D_\ast]$, $(x,y)$ is a chain proximal pair for $f$ and so satisfies $x\sim_f y$. Let $\mathcal{U}\in\mathscr{U}(X)$. Then, there are $\lambda\in I$ and $\mathcal{U}_\lambda\in\mathscr{U}(X_\lambda)$ such that $\pi_\lambda^{-1}(V)\subset U_V$ for all $V\in\mathcal{U}_\lambda$ for some $U_V\in\mathcal{U}$. For such $\lambda$, since $\pi$ satisfies the MLC, there is $\lambda\le\mu$ such that $\pi_\lambda^\mu(X_\mu)=\pi_\lambda^\nu(X_\nu)$ for all $\mu\le\nu$, implying $\pi_\lambda(X)=\pi_\lambda^\mu(X_\mu)$. For such $\mu$, take $\mathcal{U}_\mu\in\mathscr{U}(X_\mu)$ such that $\pi_\lambda^\mu(W)\subset V_W$ for all $W\in\mathcal{U}_\mu$ for some $V_W\in\mathcal{U}_\lambda$. Since $x_\mu,y_\mu\in D_\mu$, $(x_\mu,y_\mu)$ is a chain proximal pair $f_\mu$, so there is a pair
\[
((x_i)_{i=0}^k,(y_i)_{i=0}^k)
\]
of $\mathcal{U}_\mu$-chains of $f_\mu$ such that $(x_0,y_0)=(x_\mu,y_\mu)$ and $x_k=y_k$. Let
\[
\Pi=((z_i)_{i=0}^k,(w_i)_{i=0}^k)=((\pi_\lambda^\mu(x_i))_{i=0}^k,(\pi_\lambda^\mu(y_i))_{i=0}^k).
\]
Since $\pi_\lambda^\mu\colon(X_\mu,f_\mu)\to(X_\lambda,f_\lambda)$ is an equivariant map, we have
\[
\{f_\lambda(z_i),z_{i+1}\}=\{f_\lambda(\pi_\lambda^\mu(x_i)),\pi_\lambda^\mu(x_{i+1})\}=\{\pi_\lambda^\mu(f_\mu(x_i)),\pi_\lambda^\mu(x_{i+1})\}
\]
and
\[
\{f_\lambda(w_i),w_{i+1}\}=\{f_\lambda(\pi_\lambda^\mu(y_i)),\pi_\lambda^\mu(y_{i+1})\}=\{\pi_\lambda^\mu(f_\mu(y_i)),\pi_\lambda^\mu(y_{i+1})\}
\]
for each $0\le i\le k-1$, so by the choice of $\mathcal{U}_\mu$, $\Pi$ is a pair of $\mathcal{U}_\lambda$-chains of $f_\lambda$ with $(z_0,w_0)=(x_\lambda,y_\lambda)$ and $z_k=w_k$. By $\pi_\lambda(X)=\pi_\lambda^\mu(X_\mu)$, we obtain $(u_i,v_i)\in X^2$, $1\le i\le k-1$, and $p\in X$  such that
\[
(\pi_\lambda(u_i),\pi_\lambda(v_i))=(z_i,w_i)
\]
and $\pi_\lambda(p)=z_k=w_k$. Let $(u_0,v_0)=(x,y)$ and note that
\[
(\pi_\lambda(u_0),\pi_\lambda(v_0))=(\pi_\lambda(x),\pi_\lambda(y))=(x_\lambda,y_\lambda)=(z_0,w_0).
\]
Also, let $u_k=v_k=p$ and note that
\[
\pi_\lambda(u_k)=\pi_\lambda(v_k)=\pi_\lambda(p)=z_k=w_k.
\]
Since $\pi_\lambda\colon(X,f)\to(X_\lambda,f_\lambda)$ is an equivariant map, we have
\[
\{\pi_\lambda(f(u_i)),\pi_\lambda(u_{i+1})\}=\{f_\lambda(\pi_\lambda(u_i)),\pi_\lambda(u_{i+1})\}=\{f_\lambda(z_i),z_{i+1}\}
\]
and
\[
\{\pi_\lambda(f(v_i)),\pi_\lambda(v_{i+1})\}=\{f_\lambda(\pi_\lambda(v_i)),\pi_\lambda(v_{i+1})\}=\{f_\lambda(w_i),w_{i+1}\}
\]
for each $0\le i\le k-1$, so by the choice of $\lambda$ and $\mathcal{U}_\lambda$, we see that
\[
((u_i)_{i=0}^k,(v_i)_{i=0}^k)
\]
is a pair of $\mathcal{U}$-chains of $f$ with $(u_0,v_0)=(x,y)$ and $u_k=v_k$. Since $\mathcal{U}\in\mathscr{U}(X)$ is arbitrary, this shows that $(x,y)$ is a proximal pair for $f$ and so satisfies $x\sim_f y$.

For any $D_\ast=(D_\lambda)_{\lambda\in I}\in\mathcal{D}_\pi$, by what is shown above, we have $[D_\ast]\subset D$ for some $D\in\mathcal{D}(f)$. Fix $x=(x_\lambda)_{\lambda\in I}\in[D_\ast]$. Then, for every $y=(y_\lambda)_{\lambda\in I}\in D$, since $x,y\in D$, we have $x\sim_f y$. Since $\pi_\lambda\colon(X,f)\to(X_\lambda,f_\lambda)$, $\lambda\in I$, are equivariant maps, it holds that $\pi_\lambda(x)\sim_{f_\lambda}\pi_\lambda(y)$, that is, $x_\lambda\sim_{f_\lambda}y_\lambda$ for all $\lambda\in I$; therefore,  $y\in[D_\ast]$. Since $y\in D$ is arbitrary, we obtain $D\subset[D_\ast]$ and so $[D_\ast]=D$. This implies
\[
\{[D_\ast]\colon D_\ast\in\mathcal{D}_\pi\}\subset\mathcal{D}(f).
\]
Conversely, for any $D\in\mathcal{D}(f)$, fix $x=(x_\lambda)_{\lambda\in I}\in D$ and take $D_\lambda\in\mathcal{D}(f_\lambda)$, $\lambda\in I$, such that $x_\lambda\in D_\lambda$ for all $\lambda\in I$. By $x_\mu\in D_\mu$ and $\pi_\lambda^\mu(x_\mu)=x_\lambda\in D_\lambda$, we obtain $\pi_\lambda^\mu(D_\mu)\subset D_\lambda$ for all $\lambda\le\mu\in I$, that is, $D_\ast=(D_\lambda)_{\lambda\in I}\in\mathcal{D}_\pi$. Note that $x\in[D_\ast]$.  By $x\in D$ and what is shown above, we obtain $[D_\ast]\subset D$. Then, similarly as above, we see that $D=[D_\ast]$. Since $D\in\mathcal{D}(f)$ is arbitrary, it follows that
\[
\mathcal{D}(f)\subset\{[D_\ast]\colon D_\ast\in\mathcal{D}_\pi\},
\]
and thus the lemma has been proved.
\end{proof}

Let $\pi=(\pi_\lambda^\mu\colon(X_\mu,f_\mu)\to(X_\lambda,f_\lambda))_{\lambda\le\mu\in I}$ be an inverse system of equivariant maps such that $(X_\lambda,f_\lambda)$, $\lambda\in I$, are transitive subshifts of finite type for all $\lambda\in I$. For every $\lambda\in I$, note that $\mathcal{D}(f_\lambda)$ is a finite set, and let $m_\lambda=|\mathcal{D}(f_\lambda)|$. Then, we have $m_\lambda|m_\mu$ for all $\lambda\le\mu\in I$, and for any $\lambda\in I$ and $E\in\mathcal{D}(f_\lambda)$, $(f_\lambda^{m_\lambda})|_E\colon E\to E$ is mixing.

\begin{lem}
If $\pi$ satisfies the MLC, then for every $D_\ast=(D_\lambda)_{\lambda\in I}\in\mathcal{D}_\pi$,
\[
\tilde{\pi}=((\pi_\lambda^\mu)|_{D_\mu}\colon D_\mu\to D_\lambda)_{\lambda\le\mu\in I}
\]
satisfies the MLC.
\end{lem}
 
\begin{proof}
Let $D_\ast=(D_\lambda)_{\lambda\in I}\in\mathcal{D}_\pi$. For any $\lambda\in I$, since $\pi$ satisfies the MLC, there is $\lambda\le\mu$ such that $\pi_\lambda^\mu(X_\mu)=\pi_\lambda^\nu(X_\nu)$ for all $\mu\le\nu$. For such $\mu$, we shall prove that
\[
\pi_\lambda^\mu(D_\mu)=\pi^\nu_\lambda(D_\nu)
\]
for all $\mu\le\nu$. Take $E_\mu\in\mathcal{D}(f_\mu)$ such that $\pi_\lambda^\mu(E_\mu)$ is maximal among
\[
\{\pi_\lambda^\mu(F_\mu)\colon F_\mu\in\mathcal{D}(f_\mu)\}
\]
with respect to the inclusion relation, meaning that for every $F_\mu\in\mathcal{D}(f_\mu)$, $\pi_\lambda^\mu(E_\mu)\subset\pi_\lambda^\mu(F_\mu)$ implies $\pi_\lambda^\mu(E_\mu)=\pi_\lambda^\mu(F_\mu)$. Note that $\pi_\lambda^\mu(E_\mu)$ is $f_\lambda^{m_\mu}$-invariant, and
\[
(f_\lambda^{m_\mu})|_{\pi_\lambda^\mu(E_\mu)}\colon\pi_\lambda^\mu(E_\mu)\to\pi_\lambda^\mu(E_\mu)
\]
is mixing. Since $m_\mu|m_\nu$, $\pi_\lambda^\mu(E_\mu)$ is $f_\lambda^{m_\nu}$-invariant, and
\[
(f_\lambda^{m_\nu})|_{\pi_\lambda^\mu(E_\mu)}\colon\pi_\lambda^\mu(E_\mu)\to\pi_\lambda^\mu(E_\mu)
\]
is transitive. Take a transitive point $p\in\pi_\lambda^\mu(E_\mu)$, that is, $\pi_\lambda^\mu(E_\mu)=\omega(p,f_\lambda^{m_\nu})$, the $\omega$-limit set. By the choice of $\mu$, it holds that
\[
p\in\pi_\lambda^\mu(E_\mu)\subset\pi_\lambda^\mu(X_\mu)=\pi_\lambda^\nu(X_\nu),
\]
so we have $p=\pi_\lambda^\nu(q)$ for some $q\in X_\nu$. Take $F_\nu\in\mathcal{D}(f_\nu)$ and $F_\mu\in\mathcal{D}(f_\mu)$ such that $q\in F_\nu$ and $\pi_\mu^\nu(F_\nu)\subset F_\mu$. Then, since  $p=\pi_\lambda^\nu(q)\in\pi_\lambda^\nu(F_\nu)$, and $\pi_\lambda^\nu(F_\nu)$ is $f_\lambda^{m_\nu}$-invariant,
\[
\pi_\lambda^\mu(E_\mu)=\omega(p,f_\lambda^{m_\nu})\subset\pi_\lambda^\nu(F_\nu)=\pi_\lambda^\mu(\pi_\mu^\nu(F_\nu))\subset\pi_\lambda^\mu(F_\mu).
\]
By the maximality of $\pi_\lambda^\mu(E_\mu)$, we obtain
\[
\pi_\lambda^\mu(E_\mu)=\pi_\lambda^\nu(F_\nu)=\pi_\lambda^\mu(F_\mu).
\]
Take $k\ge0$ such that $D_\nu=f_\nu^k(F_\nu)$. Then, we have $f_\mu^k(F_\mu)\in\mathcal{D}(f_\mu)$ and
\[
\pi_\mu^\nu(D_\nu)=\pi_\mu^\nu(f_\nu^k(F_\nu))=f_\mu^k(\pi_\mu^\nu(F_\nu))\subset f^k_\mu(F_\mu).
\]
Because $\pi_\mu^\nu(D_\nu)\subset D_\mu$, this implies $D_\mu=f_\mu^k(F_\mu)$. From $\pi_\lambda^\mu(F_\mu)=\pi_\lambda^\nu(F_\nu)$ as shown above, it follows that
\begin{align*}
\pi_\lambda^\mu(D_\mu)&=\pi_\lambda^\mu(f_\mu^k(F_\mu))=f_\lambda^k(\pi_\lambda^\mu(F_\mu))\\
&=f_\lambda^k(\pi_\lambda^\nu(F_\nu))=\pi_\lambda^\nu(f_\nu^k(F_\nu))=\pi_\lambda^\nu(D_\nu),
\end{align*}
proving the claim. Since $\lambda\in I$ is arbitrary, we conclude that $\pi'$ satisfies the MLC, completing the proof.
\end{proof}

Let $\pi=(\pi_\lambda^\mu\colon(X_\mu,f_\mu)\to(X_\lambda,f_\lambda))_{\lambda\le\mu\in I}$ be an inverse system of equivariant maps such that $(X_\lambda,f_\lambda)$, $\lambda\in I$, are transitive subshifts of finite type for all $\lambda\in I$; and $\pi$ satisfies the MLC. By Lemmas 4.2 and 4.3, we have
\[
\mathcal{D}(f)=\{[D_\ast]\colon D_\ast\in\mathcal{D}_\pi\}, 
\]
and for every $D_\ast=(D_\lambda)_{\lambda\in I}\in\mathcal{D}_\pi$,
\[
\tilde{\pi}=((\pi_\lambda^\mu)|_{D_\mu}\colon D_\mu\to D_\lambda)_{\lambda\le\mu\in I}
\]
satisfies the MLC.

\begin{lem}
Let $(X,f)=\varprojlim(X_\lambda,f_\lambda)$. Then, $f\colon X\to X$ has DSP.
\end{lem}

\begin{proof}
We shall prove that for any $D\in\mathcal{D}(f)$ and $\mathcal{V}\in\mathscr{U}(X)$, there is $\mathcal{U}\in\mathscr{U}(X)$ such that $D\in\mathcal{D}^{\mathcal{V},\mathcal{U}}(f)$. This implies $\mathcal{D}(f)=\mathcal{D}_{\rm sh}(f)$ and so that $f$ has DSP. Let $D\in\mathcal{D}(f)$ and $\mathcal{V}\in\mathscr{U}(X)$. Then, we have $D_\ast\in\mathcal{D}_\pi$ such that $D=[D_\ast]$; and
\[
\tilde{\pi}=((\pi_\lambda^\mu)|_{D_\mu}\colon D_\mu\to D_\lambda)_{\lambda\le\mu\in I}
\]
satisfies the MLC. Take $\lambda\in I$ and $\mathcal{V}_\lambda\in\mathscr{U}(X_\lambda)$ such that $\pi_\lambda^{-1}(S)\subset V_S$ for all $S\in\mathcal{V}_\lambda$ for some $V_S\in\mathcal{V}$. For such $\lambda$, since $\tilde{\pi}$ satisfies the MLC, there is $\lambda\le\mu$ such that $\pi_\lambda^\mu(D_\mu)=\pi_\lambda^\nu(D_\nu)$ for all $\mu\le\nu$, implying $\pi_\lambda([D_\ast])=\pi_\lambda^\mu(D_\mu)$. For such $\mu$, take $\mathcal{V}_\mu\in\mathscr{U}(X_\mu)$ such that $\pi_\lambda^\mu(T)\subset S_T$ for all $T\in\mathcal{V}_\mu$ for some $S_T\in\mathcal{V}_\lambda$. Then, since $(X_\mu,f_\mu)$ is a transitive subshift of finite type, $f_\mu\colon X_\mu\to X_\mu$ has DSP, so there is $\mathcal{U}_\mu\in\mathscr{U}(X_\mu)$ such that $\mathcal{D}(f_\mu)=\mathcal{D}^{\mathcal{V}_\mu,\mathcal{U}_\mu}(f_\mu)$, implying that every $\mathcal{U}_\mu$-pseudo orbit $\xi'=(y_i)_{i\ge0}$ of $f_\mu$ is $\mathcal{V}_\mu$-shadowed by some $y_\mu\in X_\mu$ with $y_\mu\sim_{f_\mu}y_0$. For such $\mathcal{U}_\mu$, take $\mathcal{U}\in\mathscr{U}(X)$ such that $\pi_\mu(U)\subset W_U$ for all $U\in\mathcal{U}$ for some $W_U\in\mathcal{U}_\mu$. We shall prove $D\in\mathcal{D}^{\mathcal{V},\mathcal{U}}(f)$.

Given any $\mathcal{U}$-pseudo orbit $\xi=(x_i)_{i\ge0}$ of $f$ with $x_0\in D=[D_\ast]$, let
\[
\xi'=(y_i)_{i\ge0}=(\pi_\mu(x_i))_{i\ge0}.
\]
Since $\pi_\mu\colon(X,f)\to(X_\mu,f_\mu)$ is an equivariant map, we have
\[
\{f_\mu(y_i),y_{i+1}\}=\{f_\mu(\pi_\mu(x_i)),\pi_\mu(x_{i+1})\}=\{\pi_\mu(f(x_i)),\pi_\mu(x_{i+1})\}
\]
for every $i\ge0$, so by the choice of $\mathcal{U}$, $\xi'$ is a $\mathcal{U}_\mu$-pseudo orbit of $f_\mu$, which is $\mathcal{V}_\mu$-shadowed by some $y_\mu\in X_\mu$ with $y_\mu\sim_{f_\mu} y_0$. By $y_0=\pi_\mu(x_0)$ and $x_0\in[D_\ast]$, we obtain $y_\mu\in D_\mu$. Letting $y_\lambda=\pi_\lambda^\mu(y_\mu)$, we have
\[
y_\lambda\in\pi_\lambda^\mu(D_\mu)=\pi_\lambda([D_\ast]),
\]
so $y_\lambda=\pi_\lambda(x)$ for some $x\in[D_\ast]$. Let
\[
\xi''=(z_i)_{i\ge0}=(\pi_\lambda^\mu(y_i))_{i\ge0}=(\pi_\lambda(x_i))_{i\ge0}
\]
and note that, since $\pi_\lambda^\mu\colon(X_\mu,f_\mu)\to(X_\lambda,f_\lambda)$ is an equivariant map,
\[
\{f_\lambda^i(y_\lambda),z_i\}=\{f_\lambda^i(\pi_\lambda^\mu(y_\mu)),\pi_\lambda^\mu(y_i)\}=\{\pi_\lambda^\mu(f_\mu^i(y_\mu)),\pi_\lambda^\mu(y_i)\}
\]
for all $i\ge0$, so by the choice of $\mathcal{V}_\mu$, $\xi''$ is $\mathcal{V}_\lambda$-shadowed by $y_\lambda$. Finally, since $\pi_\lambda\colon(X,f)\to(X_\lambda,f_\lambda)$ is an equivariant map, it holds that
\[
\{\pi_\lambda(f^i(x)),\pi_\lambda(x_i)\}=\{f_\lambda^i(\pi_\lambda(x)),\pi_\lambda(x_i)\}=\{f_\lambda^i(y_\lambda),z_i\},
\]
so by the choice of $\lambda$ and $\mathcal{V}_\lambda$, $\xi$ is $\mathcal{V}$-shadowed by $x\in[D_\ast]=D$. Since $\xi$ is arbitrary, we conclude that $D\in\mathcal{D}^{\mathcal{V},\mathcal{U}}(f)$, proving the lemma.
\end{proof}

The following lemma can be found in \cite{GM}.

\begin{lem}{\cite[Corollary 17]{GM}}
For any continuous map $f\colon X\to X$, if $\dim{X}=0$, and if $f$ has the shadowing property, then there is an
inverse system of equivariant maps
\[
\pi=(\pi_\lambda^\mu\colon(X_\mu,f_\mu)\to(X_\lambda,f_\lambda))_{\lambda\le\mu\in I}
\]
such that the following properties hold:
\begin{itemize}
\item[(1)] $(X_\lambda,f_\lambda)$, $\lambda\in I$, are subshifts of finite type for all $\lambda\in I$,
\item[(2)] $\pi$ satisfies the MLC,
\item[(3)] $(X,f)$ is topologically conjugate to $\varprojlim(X_\lambda,f_\lambda)$.
\end{itemize}
\end{lem}

Finally, we apply these lemmas to prove Theorem 1.4.

\begin{proof}[Proof of Theorem 1.4]
By Lemmas 4.1 and 4.5, $(X,f)$ is topologically conjugate to
\[
(\hat{X},\hat{f})=\varprojlim(X_\lambda,f_\lambda),
\]
where $\pi=(\pi_\lambda^\mu\colon(X_\mu,f_\mu)\to(X_\lambda,f_\lambda))_{\lambda\le\mu\in I}$ is an inverse system of equivariant maps such that $(X_\lambda,f_\lambda)$, $\lambda\in I$, are transitive subshifts of finite type for all $\lambda\in I$; and $\pi$ satisfies the MLC. Then, by Lemma 4.4, $\hat{f}\colon\hat{X}\to\hat{X}$ has DSP, so $f$ also satisfies DSP.
\end{proof}

\section{$\overline{\omega}$-limit sets}

In this section, we recall some facts about $\overline{\omega}$-limit sets and prove some basic properties of them. Throughout this section, $X$ denotes a compact metric space endowed with a metric $d$. The definition of $\overline{\omega}$-limit sets is given in Section 1. Let $\mathcal{M}(X)$ denote the set of Borel probability measures on $X$. Note that $\mathcal{M}(X)$ is a compact metrizable space when equipped with the $\text{weak}^\ast$-topology. For every $\mu\in\mathcal{M}(X)$, let ${\rm supp}(\mu)$ denote the support of $\mu$: for any $p\in X$, $p\in{\rm supp}(\mu)$ if and only if $\mu(U)>0$ for all open neighborhood $U$ of $p$.  Also, for any $p\in X$, we denote by $\delta_p$ the Dirac measure at $p$.

The following lemma can be found in \cite{Sig}.

\begin{lem}{\cite[Proposition 1]{Sig}}
Let $f\colon X\to X$ be a continuous map. For any $x\in X$, let $S_f(x)$ be the set of $\mu\in\mathcal{M}(X)$ such that \[
\lim_{j\to\infty}\frac{1}{n_j}\sum_{i=0}^{n_j-1}\delta_{f^i(x)}=\mu
\]
for some sequence $1\le n_1<n_2<\cdots$. Then, it holds that
\[
\overline{\omega}(x,f)=\overline{\bigcup_{\mu\in S_f(x)}{\rm supp}(\mu)}.
\]
\end{lem}

\begin{cor}
For any $x\in X$, $\overline{\omega}(x,f)$ is a non-empty closed $f$-invariant subset of $X$. 
\end{cor}

For any continuous map $f\colon X\to X$, a non-empty closed $f$-invariant subset $M$ of $X$ is said to be a {\em minimal set} for $f$ if closed $f|_M$-invariant subsets of $M$ are only $\emptyset$ and $M$. This condition is equivalent to $M=\omega(x,f)$ for all $x\in M$. We say that $f$ is {\em minimal} if $X$ is a minimal set for $f$. 

\begin{cor}
Let $f\colon X\to X$ be a continuous map. If $f$ is minimal, then $X=\overline{\omega}(x,f)$ for all $x\in X$. 
\end{cor}

\begin{cor}
Let $f\colon X\to X$ be a continuous map.  For any $x\in X$ and $\mu\in\mathcal{M}(X)$, if
\[
\lim_{n\to\infty}\frac{1}{n}\sum_{i=0}^{n-1}\delta_{f^i(x)}=\mu,
\]
then $\overline{\omega}(x,f)={\rm supp}(\mu)$.
\end{cor}

\begin{rem}
\normalfont
If $f$ is minimal, then by Corollary 5.2, $f$ is clearly not $\overline{\omega}$-chaotic. If $f$ is {\em uniquely ergodic}, that is, there is a unique $f$-invariant Borel probability measure $\mu$ on $X$, then by Corollary 5.3, it holds that $\overline{\omega}(x,f)={\rm supp}(\mu)$ for every $x\in X$, so $f$ does not exhibit $\overline{\omega}$-chaos. Since there exists $f$ such that $h_{\rm top}(f)>0$ and is minimal or uniquely ergodic, $h_{\rm top}(f)>0$ does not necessarily imply $\overline{\omega}$-chaos. Assume that $f$ satisfies the shadowing property and $h_{\rm top}(f)=0$. Then, by \cite{M}, $f|_{CR(f)}\colon CR(f)\to CR(f)$ is an equicontinuous homeomorphism with $\dim{CR(f)}=0$ and so is chain continuous, therefore the chain recurrent set for $f$ is decomposed into the chain components for $f$:
\[
CR(f)=\bigsqcup_{M\in\mathcal{C}(f)}M,
\]
here, for each $M\in\mathcal{C}(f)$, $M$ is a periodic orbit, or $(M,f|_M)$ is topologically conjugate to an odometer (see also Section 6 of \cite{Ka1}). It follows that for any $x\in X$, there is $M\in\mathcal{C}(f)$ such that  $\overline{\omega}(x,f)\subset M$ implying $\overline{\omega}(x,f)=M$, a minimal set for $f$. Thus, $f$ does not exhibit $\overline{\omega}$-chaos.
\end{rem}

As shown in the following lemma, the $\overline{\omega}$-limit sets share basic properties of the standard $\omega$-limit sets.

\begin{lem}
Let $f\colon X\to X$ be a continuous map. For any $x\in X$, the following properties hold:
\begin{itemize}
\item[(1)] $\overline{\omega}(f^i(x),f)=\overline{\omega}(x,f)$ for all $i\ge0$,
\item[(2)] $f^i(\overline{\omega}(x,f^m))=\overline{\omega}(f^i(x),f^m)$ for any $i\ge0$ and $m\ge1$,
\item[(3)]
\[
\overline{\omega}(x,f)=\bigcup_{i=0}^{m-1}\overline{\omega}(f^i(x),f^m)
\]
for all $m\ge1$.
\end{itemize}
Let
\[
\overline{\omega}(f^m)=\bigcup_{x\in X}\overline{\omega}(x,f^m)
\]
for each $m\ge1$. Then, the following properties also hold:
\begin{itemize}
\item[(4)] $f(\overline{\omega}(f))=\overline{\omega}(f)$,
\item[(5)] $\overline{\omega}(f)=\overline{\omega}(f^m)$ for all $m\ge1$.
\end{itemize}
\end{lem}

\begin{proof}
Properties (1) and (3) are clear from the definitions. We prove property $(2)$. For any $i\ge0$, define $f_\ast^i\colon\mathcal{M}(X)\to\mathcal{M}(X)$ by $f_\ast^i(\mu)(A)=\mu(f^{-i}(A))$ for every $\mu\in\mathcal{M}(X)$ and every Borel subset $A$ of $X$. Then, we can show that $f_\ast^i(S_{f^m}(x))=S_{f^m}(f^i(x))$ for all $x\in X$ and $m\ge1$. On the other hand, we have $f^i({\rm supp}(\mu))={\rm supp}(f_\ast^i(\mu))$ for all $\mu\in\mathcal{M}(X)$. Thus, by Lemma 5.1, we obtain
\begin{align*}
f^i(\overline{\omega}(x,f^m))&=f^i(\overline{\bigcup_{\mu\in S_{f^m}(x)}{\rm supp}(\mu)})\\
&=\overline{\bigcup_{\mu\in S_{f^m}(x)}f^i({\rm supp}(\mu))}=\overline{\bigcup_{\mu\in S_{f^m}(x)}{\rm supp}(f_\ast^i(\mu))}\\
&=\overline{\bigcup_{\nu\in S_{f^m}(f^i(x))}{\rm supp}(\nu)}=\overline{\omega}(f^i(x),f^m),
\end{align*}
proving property (2). Property (4) follows from (1) and (2). Property (5) is a consequence of (3).
\end{proof}  

\begin{lem}
For any continuous map $f\colon X\to X$, if there is $p\in X$ such that $X=\overline{\omega}(p,f)$, then
\[
\overline{T}(f)=\{x\in X\colon X=\overline{\omega}(x,f)\}
\]
is a residual subset of $X$.
\end{lem}

\begin{proof}
Let $\{U_j\colon j\ge1\}$ be a countable basis for the topology on $X$. Since $X=\overline{\omega}(p,f)$, letting
\[
m_j=\limsup_{n\to\infty}\frac{1}{n}|\{0\le i\le n-1\colon f^i(p)\in U_j\}|,
\]
$j\ge1$, we have $m_j>0$ for each $j\ge1$. Let
\[
A=\bigcap_{j=1}^\infty\{x\in X\colon\limsup_{n\to\infty}\frac{1}{n}|\{0\le i\le n-1\colon f^i(x)\in U_j\}|\ge m_j\}
\]
and note that $\{f^k(p)\colon k\ge0\}\subset A\subset\overline{T}(f)$. Since $X=\overline{\omega}(p,f)$, we have $X=\omega(p,f)$, so $\{f^k(p)\colon k\ge0\}$ is dense in $X$. By
\[
A=\bigcap_{j=1}^\infty\bigcap_{l=1}^\infty\bigcap_{N=1}^\infty\bigcup_{n\ge N}\{x\in X\colon\frac{1}{n}|\{0\le i\le n-1\colon f^i(x)\in U_j\}|> m_j-\frac{1}{l}\},
\]
we see that $A$ is a dense $G_\delta$-subset of $X$; therefore, $\overline{T}(f)$ is a residual subset of $X$, completing the proof.
\end{proof}

For any sequence $\xi=(x_i)_{i\ge0}$ of points in $X$, we define its $\overline{\omega}$-limit set by
\[
\overline{\omega}(\xi)=\{y\in X\colon\overline{d}(\{i\ge0\colon x_i\in B_\epsilon(y)\})>0,\:\forall\epsilon>0\}.
\]
In this notation, we have $\overline{\omega}(x,f)=\overline{\omega}((f^i(x))_{i\ge0})$ for any continuous map $f\colon X\to X$ and $x\in X$.

\begin{lem}
For any sequences $\xi=(x_i)_{i\ge0}$, $\xi'=(y_i)_{i\ge0}$ of points in $X$, if
\[
\lim_{n\to\infty}\frac{1}{n}\sum_{i=0}^{n-1}d(x_i,y_i)=0,
\]
then $\overline{\omega}(\xi)=\overline{\omega}(\xi')$.
\end{lem}

\begin{proof}
We shall prove $\overline{\omega}(\xi)\subset\overline{\omega}(\xi')$. For any $y\in\overline{\omega}(\xi)$ and $\epsilon>0$, take $a,b>0$ such that 
\[
\overline{d}(\{i\ge0\colon x_i\in B_{\epsilon/2}(y)\})>b
\]
and $b-2a>0$. Then, due to the assumption, there is a sequence $1\le n_1<n_2<\cdots$ such that
\[
\frac{1}{n_j}|\{0\le i\le n_j-1\colon x_i\in B_{\epsilon/2}(y)\}|>b
\]
and
\[
\frac{1}{n_j}\sum_{i=0}^{n_j-1}d(x_i,y_i)<a\epsilon
\]
for all $j\ge1$. The second condition implies that
\[
\frac{1}{n_j}|\{0\le i\le n_j-1\colon d(x_i,y_i)>\epsilon/2\}|<2a
\]
for every $j\ge1$. Note that
\[
\{0\le i\le n_j-1\colon y_i\in B_\epsilon(y)\}
\]
contains
\[
\{0\le i\le n_j-1\colon x_i\in B_{\epsilon/2}(y)\}\setminus\{0\le i\le n_j-1\colon d(x_i,y_i)>\epsilon/2\},
\]
so we have
\[
\frac{1}{n_j}|\{0\le i\le n_j-1\colon y_i\in B_\epsilon(y)\}|>b-2a>0
\]
for all $j\ge1$. This implies
\[
\overline{d}(\{i\ge0\colon y_i\in B_\epsilon(y)\})>0.
\]
Since $\epsilon>0$ is arbitrary, we obtain $y\in\overline{\omega}(\xi')$ and so $\overline{\omega}(\xi)\subset\overline{\omega}(\xi')$, proving the claim. Similarly, we can show that $\overline{\omega}(\xi')\subset\overline{\omega}(\xi)$, thus the lemma has been proved.
\end{proof}

\begin{cor}
Let $f\colon X\to X$ be a continuous map.  For any $x,y\in X$, if
\[
\lim_{n\to\infty}\frac{1}{n}\sum_{i=0}^{n-1}d(f^i(x),f^i(y))=0,
\]
then $\overline{\omega}(x,f)=\overline{\omega}(y,f)$.
\end{cor}

Let $f\colon X\to X$ be a continuous map. For any $(x,y)\in X^2$ and $t>0$, we define
\[
F_{xy}(t)=\liminf_{n\to\infty}\frac{1}{n}|\{0\le i\le n-1\colon d(f^i(x),f^i(y))<t\}|
\]
and
\[
F_{xy}^\ast(t)=\limsup_{n\to\infty}\frac{1}{n}|\{0\le i\le n-1\colon d(f^i(x),f^i(y))<t\}|.
\]
A pair $(x,y)\in X^2$ is said to be {\em DC2-scrambled} if $F_{xy}(\delta)<1$ for some $\delta>0$, and $F_{xy}^\ast(t)=1$ for all $t>0$. It is known that these properties are equivalent to
\[
\limsup_{n\to\infty}\frac{1}{n}\sum_{i=0}^{n-1}d(f^i(x),f^i(y))>0\quad\text{and}\quad\liminf_{n\to\infty}\frac{1}{n}\sum_{i=0}^{n-1}d(f^i(x),f^i(y))=0,
\]
respectively.

The following two results describe relationships between DC2 and $\overline{\omega}$-chaos.

\begin{prop}
Let $f\colon X\to X$ be a continuous map. For any $x,y\in X$, if $\overline{\omega}(x,f)\setminus\overline{\omega}(y,f)\ne\emptyset$, then $F_{xy}(\delta)<1$ for some $\delta>0$.
\end{prop}

\begin{proof}
Let $p\in\overline{\omega}(x,f)\setminus\overline{\omega}(y,f)$. Then, since $p\not\in\overline{\omega}(y,f)$, there is $\delta>0$ such that
\[
\lim_{n\to\infty}\frac{1}{n}|\{0\le i\le n-1\colon f^i(y)\in B_{2\delta}(p)\}|=0.
\]
For such $\delta>0$, because $p\in\overline{\omega}(x,f)$, we have
\[
\overline{d}(\{i\ge0\colon f^i(x)\in B_\delta(p)\})=\limsup_{n\to\infty}\frac{1}{n}|\{0\le i\le n-1\colon f^i(x)\in B_\delta(p)\}|>0.
\]
Since
\[
\{0\le i\le n-1\colon d(f^i(x),f^i(y))\ge\delta\}
\]
contains
\[
\{0\le i\le n-1\colon f^i(x)\in B_\delta(p)\}\setminus\{0\le i\le n-1\colon f^i(y)\in B_{2\delta}(p)\}
\]
for each $n\ge1$, the above conditions clearly imply
\[
\limsup_{n\to\infty}\frac{1}{n}|\{0\le i\le n-1\colon d(f^i(x),f^i(y))\ge\delta\}|>0,
\]
which is equivalent to $F_{xy}(\delta)<1$. Thus, the proposition has been proved.
\end{proof}

\begin{prop}
Let $f\colon X\to X$ be a continuous map. For any $x,y\in X$, if $F^\ast_{xy}(t)=1$ for all $t>0$, then $\overline{\omega}(x,f)\cap\overline{\omega}(y,f)\ne\emptyset$.
\end{prop}

\begin{proof}
The assumption that $F^\ast_{xy}(t)=1$ for all $t>0$ is equivalent to
\[
\liminf_{n\to\infty}\frac{1}{n}\sum_{i=0}^{n-1}d(f^i(x),f^i(y))=0.
\]
Then, there are a sequence $1\le n_1<n_2<\cdots$ and $\mu\in\mathcal{M}(X)$ such that
\[
\lim_{j\to\infty}\frac{1}{n_j}\sum_{i=0}^{n_j-1}d(f^i(x),f^i(y))=0
\]
and
\[
\lim_{j\to\infty}\frac{1}{n_j}\sum_{i=0}^{n_j-1}\delta_{f^i(x)}=\mu.
\]
These conditions imply
\[
\lim_{j\to\infty}\frac{1}{n_j}\sum_{i=0}^{n_j-1}\delta_{f^i(y)}=\mu.
\]
By Lemma 5.1, we obtain ${\rm supp}(\mu)\subset\overline{\omega}(x,f)\cap\overline{\omega}(y,f)$ and so $\overline{\omega}(x,f)\cap\overline{\omega}(y,f)\ne\emptyset$, completing the proof.
\end{proof}

\section{Proof of Theorem 1.6}

In this section, we prove Theorem 1.6. The proof is by Lemma 1.1 stated in Section 1, and in order to obtain $\overline{\omega}$-chaos, we use with modifications a standard method of obtaining $\omega$-chaos, introduced in \cite{SHLi}.

The first lemma is well-known but we give a proof of it for completeness.

\begin{lem}
Let $f\colon X\to X$ be a continuous map. If $f\colon X\to X$ is minimal, then for any $\epsilon>0$, there is $k\ge1$ such that
\[
\{i\ge0\colon f^i(x)\in B_\epsilon(y)\}\cap\{j,j+1,\dots,j+k-1\}\ne\emptyset
\]
for all $x,y\in X$ and $j\ge0$.
\end{lem}

\begin{proof}
Assume the contrary. Then, there is $\epsilon>0$ such that for every $k\ge1$, we have
\[
\{i\ge0\colon f^i(x_k)\in B_\epsilon(y_k)\}\cap\{j_k,j_k+1,\dots,j_k+l_k-1\}=\emptyset
\]
for some $x_k,y_k\in X$, $j_k\ge0$, and $l_k\ge k$. By passing to subsequences if necessary, we may assume that $\lim_{k\to\infty}y_k=y$ and $\lim_{k\to\infty}f^{j_k}(x_k)=z$ for some $y,z\in X$. If $f^{l}(z)\in B_\epsilon(y)$ for some $l\ge0$, then we obtain $f^{j_k+l}(x_k)\in B_\epsilon(y_k)$ and $0\le l\le l_k-1$ for some $k\ge1$, which is a contradiction. It follows that $f^{l}(z)\not\in B_\epsilon(y)$ for all $l\ge0$, however this contradicts that $f$ is minimal. Thus, the lemma has been proved.
\end{proof}

Note that for any continuous map $f\colon X\to X$ and minimal sets $M,M'$ for $f$, we have $M=M'$ if and only if $M\cap M'\ne\emptyset$. The next lemma shows that the shadowing with positive topological entropy provides many minimal sets. 

\begin{lem}
Let $f\colon X\to X$ be a continuous map. If $f$ satisfies the shadowing property and $h_{\rm top}(f)>0$, then there is a set $\mathcal{F}$ of infinite minimal sets for $f$ such that $|\mathcal{F}|=|\mathbb{R}|$.
\end{lem}

\begin{proof}
Assume that $f$ satisfies the shadowing property and $h_{\rm top}(f)>0$. Then, we know that there are $m>0$ and  a closed $f^m$-invariant subset $Y$ of $X$ such that there is a factor map
\[
\pi\colon(Y,f^m)\to(\{0,1\}^\mathbb{N},\sigma),
\]
here $\sigma\colon\{0,1\}^\mathbb{N}\to\{0,1\}^\mathbb{N}$ is the shift map (see, e.g. \cite{LiO,MO}). Take a set $\mathcal{G}$ of infinite minimal sets for $\sigma$ such that $|\mathcal{G}|=|\mathbb{R}|$, e.g. $\mathcal{G}=\{M_\alpha\colon\alpha\in[0,\log{2})\}$, where $M_\alpha$, $\alpha\in[0,\log{2})$, are Toeplitz subshifts with $h_{\rm top}(\sigma|_{M_\alpha})=\alpha$ for all $\alpha\in[0,\log{2})$. For any $N\in\mathcal{G}$, since $\pi^{-1}(N)$ is a closed $f^m$-invariant subset of $X$, we can take a minimal set $M_N$ for $f^m$ with $M_N\subset\pi^{-1}(N)$. Then, for each $N\in\mathcal{G}$, since $\pi(M_N)$ is a closed $\sigma$-invariant subset of $N$, by the minimality of $N$ for $\sigma$, we have $\pi(M_N)=N$; therefore, $M_N$ is especially an infinite set. Let $\mathcal{H}=\{M_N\colon N\in\mathcal{G}\}$. For any $N,N'\in\mathcal{G}$ with $N\ne N'$, we have $M_N\cap M_{N'}=\emptyset$, so $|\mathcal{H}|=|\mathcal{G}|=|\mathbb{R}|$. For every $M\in\mathcal{H}$, it holds that $f^m(M)=M$,  $f^i(M)$, $0\le i\le m-1$, are minimal sets for $f^m$, and
\[
\langle M\rangle=\bigcup_{i=0}^{m-1}f^i(M)
\]
is an infinite minimal set for $f$. We shall define an equivalence relation $\equiv$ in $\mathcal{H}^2$ by for any $M,M'\in\mathcal{H}$, $M\equiv M'$ if and only if $M'=f^k(M)$ for some $0\le k\le m-1$. Since every equivalence class with respect to $\equiv$ has at most $m$ elements, we have $|\mathcal{H}\slash\equiv|=|\mathbb{R}|$. For each $s\in\mathcal{H}\slash\equiv$, take $M^s\in\mathcal{H}$ such that $M^s\in s$. For any $s,s'\in\mathcal{H}\slash\equiv$, if $\langle M^s\rangle\cap\langle M^{s'}\rangle\ne\emptyset$, then we have $f^i(M^s)=f^j(M^{s'})$ for some $0\le i,j\le m-1$, and so $M^{s'}=f^k(M^s)$ for some $0\le k\le m-1$, which implies $M^s\equiv M^{s'}$; therefore $s=s'$. It follows that $\langle M^s\rangle\ne\langle M^{s'}\rangle$ for all $s,s'\in\mathcal{H}\slash\equiv$ with $s\ne s'$, thus
\[
\mathcal{F}=\{\langle M^s\rangle\colon s\in\mathcal{H}\slash\equiv\}
\]
gives a set of infinite minimal sets for $f$ such that $|\mathcal{F}|=|\mathcal{H}\slash\equiv|=|\mathbb{R}|$.
\end{proof}

The proof of the next lemma is left to the reader (see Theorem 1.20 of \cite{W}).

\begin{lem}
For any bounded sequence $(d_i)_{i\ge0}$ of non-negative numbers, it holds that
\[
\lim_{n\to\infty}\frac{1}{n}\sum_{i=0}^{n-1}d_i=0
\]
if and only if there is a sequence $0\le i_0<i_1<\cdots$ such that $d(\{i_j\colon j\ge0\})=1$ and $\lim_{j\to\infty}d_{i_j}=0$. 
\end{lem}

By using Lemma 6.3, we prove Lemma 1.1.

\begin{proof}[Proof of Lemma 1.1]
First, we prove that, under the assumptions on $(x_i)_{i\ge0}$, there is a sequence $(y_i)_{i\ge0}$ of points in $X$ such that $y_0=x_0$, $f(y_i)\sim_f y_{i+1}$ for all $i\ge0$,
\[
\lim_{i\to\infty}d(f(y_i),y_{i+1})=0,
\]
and
\[
\lim_{n\to\infty}\frac{1}{n}\sum_{i=0}^{n-1}d(x_i,y_i)=0.
\]
Fix a sequence $\epsilon_j>0$, $j\ge1$, such that $\lim_{j\to\infty}\epsilon_j=0$. Since $f$ has DSP, we can take $\delta_j>0$, $j\ge1$, such that $\mathcal{D}(f)=\mathcal{D}^{\epsilon_j,\delta_j}(f)$ for all $j\ge1$. For each $j\ge1$, take $m_j, N_j>0$ as in the properties of $\sim_{f,\delta_j}$. For any $a,b\in\mathbb{N}_0$ with $a\le b$, let $\langle a,b\rangle$ denote the interval
\[
\{i\in\mathbb{N}_0\colon a\le i\le b\}.
\]
Since $(x_i)_{i\ge0}$ satisfies
\[
\lim_{n\to\infty}\frac{1}{n}\sum_{i=0}^{n-1}d(f(x_i),x_{i+1})=0,
\]
by Lemma 6.3, there is a sequence of intervals $\langle a_k,b_k\rangle$, $k\ge1$, such that $b_k+1<a_{k+1}$ for all $k\ge1$, and letting
\[
\bigsqcup_{k=1}^\infty\langle a_k,b_k\rangle=\{i_l\colon l\ge0\},
\]
here $i_0<i_1<\cdots$, we have $d(\{i_l\colon l\ge0\})=1$ and $\lim_{l\to\infty}d(f(x_{i_l}),x_{i_l+1})=0$. After a suitable change of $\langle a_k,b_k\rangle$, $k\ge1$, we can also assume that there is a sequence $1\le k_1<k_2<\cdots$ such that for any $j\ge1$ and $k_j\le k\le k_{j+1}-1$,
\[
a_{k+1}=b_k+m_jn_{j,k}
\]
for some $n_{j,k}\ge N_j$. Fix $j\ge1$. Then, for any $k_j\le k\le k_{j+1}-1$, since $(x_i)_{i\ge0}$ satisfies $f(x_i)\sim_f x_{i+1}$ for all $i\ge0$, we have
\[
f^{m_jn_{j,k}}(x_{b_k})=f^{a_{k+1}-b_k}(x_{b_k})\sim_f x_{a_{k+1}},
\]
so $f^{m_jn_{j,k}}(x_{b_k})\sim_{f,\delta_j} x_{a_{k+1}}$. By $x_{b_k}\sim_{f,\delta_j}f^{m_jn_{j,k}}(x_{b_k})$, we obtain $x_{b_k}\sim_{f,\delta_j} x_{a_{k+1}}$. Then, the choice of $m_j,N_j$ gives a $\delta_j$-chain $(z_i)_{i=0}^{m_jn_{j,k}}$ of $f$ such that $z_0=x_{b_k}$ and $z_{m_jn_{j,k}}=x_{a_{k+1}}$. By $\mathcal{D}(f)=\mathcal{D}^{\epsilon_j,\delta_j}(f)$, we obtain $d(f^i(z_{j,k}),z_i)\le\epsilon_j$ for all $0\le i\le m_jn_{j,k}$ for some $z_{j,k}\in X$ with $z_{j,k}\sim_f x_{b_k}$. Let
\[
\alpha_{k}=(x_{a_k},x_{a_k+1},\dots,x_{b_k})
\]
and
\[
\beta_k=(x_{b_k},f(z_{j,k}),f^2(z_{j,k}),\dots,f^{m_jn_{j,k}-1}(z_{j,k}),x_{a_{k+1}})
\] 
for each $k_j\le k\le k_{j+1}-1$. Let $\gamma_0=(x_0,x_1,\dots,x_{a_{k_1}})$,
\[
\gamma_j=\alpha_{k_j}\beta_{k_j}\alpha_{k_j+1}\beta_{k_j+1}\cdots\alpha_{k_{j+1}-1}\beta_{k_{j+1}-1}
\]
for every $j\ge1$, and consider a sequence
\[
(y_i)_{i\ge0}=\gamma_0\gamma_1\gamma_2\cdots.
\]
Then, $y_0=x_0$. By the choice of $\beta_k$, $k\ge1$, we have $f(y_i)\sim_f y_{i+1}$ for all $i\ge0$. Recall that
\[
\bigsqcup_{k=1}^\infty\langle a_k,b_k\rangle=\{i_l\colon l\ge0\}
\]
satisfies $d(\{i_l\colon l\ge0\})=1$ and $\lim_{l\to\infty}d(f(x_{i_l}),x_{i_l+1})=0$. Since $\lim_{j\to\infty}\epsilon_j=0$, the choice $\alpha_k,\beta_k$, $k\ge1$, implies
\[
\lim_{i\to\infty}d(f(y_i),y_{i+1})=0.
\]
Also, by Lemma 6.3, the choice of $\alpha_k$, $k\ge1$, implies
\[
\lim_{n\to\infty}\frac{1}{n}\sum_{i=0}^{n-1}d(x_i,y_i)=0,
\]
thus $(y_i)_{i\ge0}$ satisfies the desired properties.

For a sequence $(y_i)_{i\ge0}$ as above, we prove that for any $\epsilon>0$, there is $x\in X$ such that $x\sim_f y_0$, $d(x,y_0)<\epsilon$, and
\[
\lim_{n\to\infty}\frac{1}{n}\sum_{i=0}^{n-1}d(f^i(x),y_i)=0.
\]
This implies $x\sim_f x_0$, $d(x,x_0)<\epsilon$, and
\[
\lim_{n\to\infty}\frac{1}{n}\sum_{i=0}^{n-1}d(f^i(x),x_i)=0,
\]
thus completing the proof of the lemma.

Take a sequence $\eta_j>0$, $j\ge0$, such that $\sum_{j=0}^\infty\eta_j<\epsilon$. Since $f$ has DSP, we can take $\zeta_j>0$, $j\ge0$, such that $\mathcal{D}(f)=\mathcal{D}^{\eta_j,\zeta_j}(f)$ for all $j\ge0$. For each $j\ge0$, fix $m_j,N_j>0$ as in the properties of $\sim_{f,\zeta_j}$. Then, we take a sequence of integers $n_j\ge0$, $j\ge0$, such that  the following properties hold:
\begin{itemize}
\item $d(f(y_i),y_{i+1})\le\zeta_j$ for all $j\ge0$ and $i\ge n_j+m_jN_j$,
\item $n_j+m_jN_j\le n_{j+1}$ for any $j\ge0$.
\end{itemize}
By taking $N_0$ large enough, we may assume that $n_0=0$ and $n_1=m_0N_0$. Since $f(y_i)\sim_f y_{i+1}$ for all $i\ge0$, we have $f^{n_1}(y_0)\sim_f y_{n_1}$, implying $f^{n_1}(y_0)\sim_{f,\zeta_0}y_{n_1}$. By $y_0\sim_{f,\zeta_0}f^{n_1}(y_0)$, we obtain $y_0\sim_{f,\zeta_0}y_{n_1}$. Then, the choice of $m_0,N_0$ gives a $\zeta_0$-chain $\kappa_0=(v_{0,i})_{i=0}^{n_1}$ of $f$ with $v_{0,0}=y_0$ and $v_{0,n_1}=y_{n_1}$. By the choice of $n_1$, 
\[
\xi_0=\kappa_0(y_{n_1},y_{n_1+1},\dots)
\]
is a $\zeta_0$-pseudo orbit of $f$ and so $\eta_0$-shadowed by some $w_0\in X$ with $w_0\sim_f y_0$. In the following, we inductively construct a sequence $w_k\in X$, $k\ge1$, such that $w_k\sim_f y_0$ for all $k\ge1$; and for each $k\ge1$, a $\zeta_k$-pseudo orbit
\[
\xi_k=(w_{k-1},f(w_{k-1}),\dots,f^{n_k}(w_{k-1}))\kappa_k(y_{n_k+m_kN_k},y_{n_k+m_kN_k+1},\dots)
\]
of $f$ is $\eta_k$-shadowed by $w_k$, where $\kappa_k=(v_{k,i})_{i=0}^{m_kN_k}$ is a $\zeta_k$-chain of $f$ with $v_{k,0}=f^{n_k}(w_{k-1})$ and $v_{k,m_kN_k}=y_{n_k+m_kN_k}$.

First, since $w_0\sim_f y_0$, it holds that $f^{n_1}(w_0)\sim_f f^{n_1}(y_0)$. By $f^{n_1}(y_0)\sim_f y_{n_1}$, we obtain $f^{n_1}(w_0)\sim_f y_{n_1}$. Since $y_{n_1}\sim_{f,\zeta_1}f^{m_1N_1}(y_{n_1})$ and $f^{m_1N_1}(y_{n_1})\sim_f y_{n_1+m_1N_1}$, we have $y_{n_1}\sim_{f,\zeta_1}y_{n_1+m_1N_1}$. It follows that
\[
f^{n_1}(w_0)\sim_{f,\zeta_1}y_{n_1+m_1N_1}.
\]
Then, the choice of $m_1,N_1$ gives a $\zeta_1$-chain $\kappa_1=(v_{1,i})_{i=0}^{m_1N_1}$ of $f$ with $v_{1,0}=f^{n_1}(w_0)$ and $v_{1,m_1N_1}=y_{n_1+m_1N_1}$. By the choice of $n_1$, 
\[
\xi_1=(w_0,f(w_0),\dots,f^{n_1}(w_0))\kappa_1(y_{n_1+m_1N_1},y_{n_1+m_1N_1+1},\dots)
\]
is a $\zeta_1$-pseudo orbit of $f$ and so $\eta_1$-shadowed by some $w_1\in X$ with $w_1\sim_f w_0$.  By $w_0\sim_f y_0$, we obtain $w_1\sim_f y_0$. Assume that $w_k\in X$ is given as above for some $k\ge1$. Since $w_k\sim_f y_0$, it holds that $f^{n_{k+1}}(w_k)\sim_f f^{n_{k+1}}(y_0)$. By $f^{n_{k+1}}(y_0)\sim_f y_{n_{k+1}}$, we obtain $f^{n_{k+1}}(w_k)\sim_f y_{n_{k+1}}$. Since $y_{n_{k+1}}\sim_{f,\zeta_{k+1}}f^{m_{k+1}N_{k+1}}(y_{n_{k+1}})$ and $f^{m_{k+1}N_{k+1}}(y_{n_{k+1}})\sim_f y_{n_{k+1}+m_{k+1}N_{k+1}}$, we have $y_{n_{k+1}}\sim_{f,\zeta_{k+1}}y_{n_{k+1}+m_{k+1}N_{k+1}}$. It follows that
\[
f^{n_{k+1}}(w_k)\sim_{f,\zeta_{k+1}}y_{n_{k+1}+m_{k+1}N_{k+1}}.
\]
Then, the choice of $m_{k+1},N_{k+1}$ gives a $\zeta_{k+1}$-chain $\kappa_{k+1}=(v_{k+1,i})_{i=0}^{m_{k+1}N_{k+1}}$ of $f$ with $v_{k+1,0}=f^{n_{k+1}}(w_k)$ and $v_{k+1,m_{k+1}N_{k+1}}=y_{n_{k+1}+m_{k+1}N_{k+1}}$. By the choice of $n_{k+1}$, 
\[
\xi_{k+1}=(w_k,f(w_k),\dots,f^{n_{k+1}}(w_k))\kappa_{k+1}(y_{n_{k+1}+m_{k+1}N_{k+1}},y_{n_{k+1}+m_{k+1}N_{k+1}+1},\dots)
\]
is a $\zeta_{k+1}$-pseudo orbit of $f$ and so $\eta_{k+1}$-shadowed by some $w_{k+1}\in X$ with $w_{k+1}\sim_f w_k$.  By $w_k\sim_f y_0$, we obtain $w_{k+1}\sim_f y_0$, completing the construction.

By the construction above, we have
\[
d(y_0,w_0)+\sum_{k=1}^\infty d(w_{k-1},w_k)\le\eta_0+\sum_{k=1}^\infty\eta_k<\epsilon,
\]
implying that $(w_k)_{k\ge0}$ is a Cauchy sequence, and so $\lim_{k\to\infty}w_k=x$ for some $x\in X$ with $d(x,y_0)<\epsilon$. For such $x$, because $w_k\sim_f y_0$ for every $k\ge1$, we obtain $x\sim_f y_0$. Again by the construction above, we see that for any $j\ge0$ and $k\ge j$,
\[
(y_{n_j+m_jN_j},y_{n_j+m_jN_j+1},\dots,y_{n_{j+1}})
\]
is $\eta_j$-shadowed by $f^{n_j+m_jN_j}(w_j)$, and
\[
(f^{n_j+m_jN_j}(w_k),f^{n_j+m_jN_j+1}(w_k),\dots,f^{n_{j+1}}(w_k))
\]
is $\eta_{k+1}$-shadowed by $f^{n_j+m_jN_j}(w_{k+1})$. It follows that for any $j\ge0$ and $k\ge j$,
\[
(y_{n_j+m_jN_j},y_{n_j+m_jN_j+1},\dots,y_{n_{j+1}})
\]
is $\sum_{l=j}^\infty\eta_l$-shadowed by $f^{n_j+m_jN_j}(w_k)$. By taking the limit as $k\to\infty$, we obtain that
\[
(y_{n_j+m_jN_j},y_{n_j+m_jN_j+1},\dots,y_{n_{j+1}})
\]
is $\sum_{l=j}^\infty\eta_l$-shadowed by $f^{n_j+m_jN_j}(x)$ for all $j\ge0$. If $(n_j)_{j\ge0}$ is a rapidly increasing sequence, then letting
\[
A=\bigsqcup_{j=0}^\infty\langle n_j+m_jN_j,n_{j+1}\rangle,
\]
we have $d(A)=1$. Since $\lim_{j\to\infty}\sum_{l=j}^\infty\eta_l=0$, by Lemma 6.3, we conclude that 
\[
\lim_{n\to\infty}\frac{1}{n}\sum_{i=0}^{n-1}d(f^i(x),y_i)=0,
\]
completing the proof of the lemma.
\end{proof}

For the proof of Theorem 1.6, we need the following lemma. 

\begin{lem}
Let $f\colon X\to X$ be a chain transitive continuous map and let $M,M'$ be minimal sets for $f$.
Then, for any $x\in X$, there is a sequence $\xi=(x_i)_{i\ge0}$ of points in $X$ such that $x_0=x$, $f(x_i)\sim_f x_{i+1}$ for every $i\ge0$,
\[
\lim_{n\to\infty}\frac{1}{n}\sum_{i=0}^{n-1}d(f(x_i),x_{i+1})=0,
\]
and $\overline{\omega}(\xi)=M\cup M'$.
\end{lem}

\begin{proof}
Let $(Y,g)$ be the factor of $(X,f)$ with respect to $\sim_f$ and let
\[
\pi\colon(X,f)\to(Y,g)
\]
be the factor map. We know that $Y$ is a periodic orbit, or $(Y,g)$ is topologically conjugate to an odometer, therefore $g\colon Y\to Y$ is minimal in both cases  (see the proof of \cite[Theorem 6]{RW}). For any closed $f$-invariant subset $S$ of $X$, $\pi(S)$ is a closed $g$-invariant subset of $Y$, thus $\pi(S)=Y$, since $g$ is minimal. This implies that for every $p\in X$, there is $q\in S$ such that $\pi(p)=\pi(q)$, that is, $p\sim_f q$. We take $y_1\in M$ with $f(x)\sim_f y_1$ and fix a sequence $1\le N_1<N_2<\cdots$. Then, take $z_1\in M'$ with $f^{N_1}(y_1)\sim_f z_1$, $y_2\in M$ with $f^{N_2}(z_1)\sim_f y_2$, $z_2\in M'$ with $f^{N_3}(y_2)\sim_f z_2$, $y_3\in M$ with $f^{N_4}(z_2)\sim_f y_3$, and so on. Continuing this process, we obtain $y_k\in M,z_k\in M'$, $k\ge1$, such that $f^{N_{2k-1}}(y_k)\sim_f z_k$ and $f^{N_{2k}}(z_k)\sim_f y_{k+1}$  for all $k\ge1$. Let
\[
\alpha_k=(y_k,f(y_k),\dots,f^{N_{2k-1}-1}(y_k))
\]
and
\[
\beta_k=(z_k,f(z_k),\dots,f^{N_{2k}-1}(z_k))
\]
for each $k\ge1$.  Consider a sequence
\[
\xi=(x_i)_{i\ge0}=(x)\alpha_1\beta_1\alpha_2\beta_2\cdots
\]
and note that $x_0=x$. By the choice of $y_k\in M,z_k\in M'$, $k\ge1$, it is clear that $f(x_i)\sim_f x_{i+1}$ for every $i\ge0$. By the choice of $\alpha_k,\beta_k$, $k\ge1$, if $(N_j)_{j\ge1}$ is a rapidly increasing sequence, then
\[
\lim_{n\to\infty}\frac{1}{n}\sum_{i=0}^{n-1}d(f(x_i),x_{i+1})=0.
\]
By the choice of $\alpha_k,\beta_k$, $k\ge1$, it is obvious that $\overline{\omega}(\xi)\subset M\cup M'$. Moreover, if $(N_j)_{j\ge1}$ is a rapidly increasing sequence, then since $M,M'$ are minimal sets for $f$, by Lemma 6.1, we have
\[
\overline{d}(\{i\ge0\colon x_i\in B_\epsilon(y)\})>0
\]
for all $y\in M\cup M'$ and $\epsilon>0$. This implies $ M\cup M'\subset\overline{\omega}(\xi)$, thus the lemma has been proved.
\end{proof}

By Lemmas 1.1, 6.2, and 6.4, we prove Theorem 1.6.

\begin{proof}[Proof of Theorem 1.6]
Due to Lemma 6.2, we have a set $\mathcal{F}$ of infinite minimal sets for $f$ such that $|\mathcal{F}|=|\mathbb{R}|$. Fix $M_0\in\mathcal{F}$ and let $\mathcal{G}=\mathcal{F}\setminus\{M_0\}$. Consider a triple
\[
(p,r,M)\in X\times(0,\infty)\times\mathcal{G}.
\]
Then, by Lemma 6.4, there is a sequence $\xi=(x_i)_{i\ge0}$ of points in $X$ such that $x_0=p$, $f(x_i)\sim_f x_{i+1}$ for every $i\ge0$,
\[
\lim_{n\to\infty}\frac{1}{n}\sum_{i=0}^{n-1}d(f(x_i),x_{i+1})=0,
\]
and $\overline{\omega}(\xi)=M_0\cup M$. For such $\xi$, since $f$ has DSP, Lemma 1.1 gives $x\in X$ with $d(x,x_0)<r$, $x\sim_f x_0$, and 
\[
\lim_{n\to\infty}\frac{1}{n}\sum_{i=0}^{n-1}d(f^i(x),x_i)=0.
\]
For such $x$, by Lemma 5.4, we obtain $\overline{\omega}(x,f)=\overline{\omega}(\xi)=M_0\cup M$. In other words, for any
$(p,r,M)\in X\times(0,\infty)\times\mathcal{G}$, we have $\overline{\omega}(x,f)=M_0\cup M$ for some $x\in X$ with $d(p,x)<r$ and $p\sim_f x$. This enables us to define a map
\[
\Gamma\colon X\times(0,\infty)\times\mathcal{G}\to X
\]
such that $d(p,\Gamma(p,r,M))<r$, $p\sim_f\Gamma(p,r,M)$, and
\[
\overline{\omega}(\Gamma(p,r,M),f)=M_0\cup M
\]
for all $(p,r,M)\in X\times(0,\infty)\times\mathcal{G}$. Since $|\mathcal{G}|=|\mathbb{R}|$ and
\[
|\mathcal{D}(f)|\le|X|\le|\mathbb{R}|,
\]
we can decompose $\mathcal{G}$ into a disjoint union of subsets $\mathcal{G}_D$, $D\in\mathcal{D}(f)$, of $\mathcal{G}$ such that $|\mathcal{G}_D|=|\mathbb{R}|$ for all $D\in\mathcal{D}(f)$. Then, since $|\{0,1\}^\mathbb{N}|=|\mathcal{G}_D|=|\mathbb{R}|$, we can take an injective map
\[
i_D\colon\{0,1\}^\mathbb{N}\to\mathcal{G}_D
\]
for every $D\in\mathcal{D}(f)$. For each $D\in\mathcal{D}(f)$, fix any continuous map
\[
\pi_D\colon\{0,1\}^\mathbb{N}\to D
\]
which is surjective. Then, we define $\phi\colon\{0,1\}^\mathbb{N}\to(0,\infty)$ by
\begin{equation*}
\phi(s)=
\begin{cases}
\exp(\sum_{n=1}^\infty(-1)^{n-1}\frac{s_n}{n})&\text{if $\sum_{n=1}^\infty(-1)^{n-1}\frac{s_n}{n}$ exists}\\
1&\text{if otherwise}
\end{cases}
\end{equation*}
for all $s=(s_n)_{n\ge1}\in\{0,1\}^\mathbb{N}$. By this definition, we see that $\phi(U)=(0,\infty)$ for any non-empty open subset $U$ of $\{0,1\}^\mathbb{N}$. For every $D\in\mathcal{D}(f)$, let $\Phi_D\colon\{0,1\}^\mathbb{N}\to D$ be a map defined by
\[
\Phi_D(s)=\Gamma(\pi_D(s),\phi(s),i_D(s))
\]
for all $s\in\{0,1\}^\mathbb{N}$. Note that
\[
\overline{\omega}(\Phi_D(s),f)=\overline{\omega}(\Gamma(\pi_D(s),\phi(s),i_D(s)),f)=M_0\cup i_D(s)
\]
for all $D\in\mathcal{D}(f)$ and $s\in\{0,1\}^\mathbb{N}$. For any $D\in\mathcal{D}(f)$, since $i_D$ is injective, this especially implies that $\Phi_D$ is also injective. Letting
\[
S=\bigcup_{D\in\mathcal{D}(f)}\Phi_D(\{0,1\}^\mathbb{N}),
\]
we show that $S$ is an $\overline{\omega}$-scrambled subset of $X$ for which $D\cap S$ is $\mathfrak{c}$-dense in $D$ for all $D\in\mathcal{D}(f)$.

Given any $x,y\in S$ with $x\ne y$, we have $x\in\Phi_D(\{0,1\}^\mathbb{N})$ and $y\in\Phi_{D'}(\{0,1\}^\mathbb{N})$ for some $D,D'\in\mathcal{D}(f)$. Then, there are $s,t\in\{0,1\}^\mathbb{N}$ such that $x=\Phi_D(s)$ and $y=\Phi_{D'}(t)$; therefore,
\[
\overline{\omega}(x,f)=\overline{\omega}(\Phi_D(s),f)=M_0\cup i_D(s)
\]
and
\[
\overline{\omega}(y,f)=\overline{\omega}(\Phi_{D'}(t),f)=M_0\cup i_{D'}(t).
\]
If $D\ne D'$, then since $i_D(s)\in\mathcal{G}_D$, $i_{D'}(t)\in\mathcal{G}_{D'}$, and $\mathcal{G}_D\cap\mathcal{G}_{D'}=\emptyset$, we have $i_D(s)\ne i_{D'}(t)$. If $D=D'$, then we have $s\ne t$, and so since $i_D=i_{D'}$ is injective, $i_D(s)\ne i_{D'}(t)$. In both cases, we obtain
\[
\overline{\omega}(x,f)\cap\overline{\omega}(y,f)=M_0,
\]
$\overline{\omega}(x,f)\setminus\overline{\omega}(y,f)=i_D(s)$, and $\overline{\omega}(y,f)\setminus\overline{\omega}(x,f)=i_{D'}(t)$. Since $x,y\in S$ with $x\ne y$ are arbitrary, this implies that $\Phi(\{0,1\}^\mathbb{N})$ is an $\overline{\omega}$-scrambled subset of $X$.

It remains to prove that $D\cap S$ is $\mathfrak{c}$-dense in $D$ for all $D\in\mathcal{D}(f)$. Let $D\in\mathcal{D}(f)$ and note that
\[
D\cap S=\Phi_D(\{0,1\}^\mathbb{N}).
\]
For any $p\in D$ and $\epsilon>0$, take $s\in\{0,1\}^\mathbb{N}$ and an open subset $U$ of $\{0,1\}^\mathbb{N}$ such that $\pi_D(s)=p$ and $\pi_D(U)\subset B_{\epsilon/2}(p)$. Then, since $\phi(U)=(0,\infty)$, for each $r\in(0,\epsilon/2)$, there is $s_r\in U$ such that $\phi(s_r)=r$. For such $s_r$, $r\in(0,\epsilon/2)$, we have
\begin{align*}
d(p,\Phi_D(s_r))&\le d(p,\pi_D(s_r))+d(\pi_D(s_r),\Phi_D(s_r))\\
&=d(p,\pi_D(s_r))+d(\pi_D(s_r),\Gamma(\pi_D(s_r),\phi(s_r),i_D(s_r)))\\
&<\epsilon/2+\phi(s_r)\\
&=\epsilon/2+r<\epsilon,
\end{align*}
thus $\Phi_D(s_r)\in B_\epsilon(p)\cap\Phi_D(\{0,1\}^\mathbb{N})$. For any $r,r'\in(0,\epsilon/2)$ with $r\ne r'$, because $\phi(s_r)=r$ and $\phi(s_{r'})=r'$, we have $s_r\ne s_{r'}$. Since $\Phi_D$ is injective, letting
\[
A=\{\Phi_D(s_r)\colon r\in(0,\epsilon/2)\},
\]
we obtain
\[
A\subset B_\epsilon(p)\cap\Phi_D(\{0,1\}^\mathbb{N})=B_\epsilon(p)\cap D\cap S
\]
and $|A|=|\mathbb{R}|$. Since $p\in D$ and $\epsilon>0$ are arbitrary, we conclude that $D\cap S$ is $\mathfrak{c}$-dense in $D$, proving the theorem.
\end{proof}

We present four corollaries of Lemma 1.1.

\begin{cor}
Let $f\colon X\to X$ be a chain transitive continuous map. If $f$ has DSP, then for any $x,y\in X$ with $y\in\mathcal{D}_f(x)$ and $\epsilon>0$, there is $z\in\mathcal{D}_f(x)$ such that $d(y,z)<\epsilon$ and
\[
\lim_{n\to\infty}\frac{1}{n}\sum_{i=0}^{n-1}d(f^i(x),f^i(z))=0.
\] 
\end{cor}

\begin{proof}
For any $x\in X$ and $y\in\mathcal{D}_f(x)$, let
\[
\xi=(x_i)_{i\ge0}=(y,f(x),f^2(x),\dots).
\]
Since $y\in\mathcal{D}_f(x)$ implies $f(y)\in\mathcal{D}_f(f(x))$, that is, $f(y)\sim_f f(x)$, $\xi$ satisfies $f(x_i)\sim_f x_{i+1}$ for all $i\ge0$; and also
\[
\lim_{n\to\infty}\frac{1}{n}\sum_{i=0}^{n-1}d(f(x_i),x_{i+1})=0.
\]
Then, for any $\epsilon>0$, Lemma 1.1 gives $z\in X$ such that $d(z,x_0)<\epsilon$, $z\sim_f x_0$, and 
\[
\lim_{n\to\infty}\frac{1}{n}\sum_{i=0}^{n-1}d(f^i(z),x_i)=0.
\]
It follows that $d(y,z)=d(x_0,z)<\epsilon$, $z\in\mathcal{D}_f(x_0)=\mathcal{D}_f(y)=\mathcal{D}_f(x)$, and
\[
\lim_{n\to\infty}\frac{1}{n}\sum_{i=0}^{n-1}d(f^i(x),f^i(z))=0,
\]
thus the corollary has been proved.
\end{proof}

\begin{cor}
Let $f\colon X\to X$ be a chain transitive continuous map. Let
\[
U_D(x)=\{y\in D\colon\liminf_{n\to\infty}\frac{1}{n}\sum_{i=0}^{n-1}d(f^i(x),f^i(y))=0\}
\]
for all $D\in\mathcal{D}(f)$ and $x\in D$. If $f$ has DSP, then for any $D\in\mathcal{D}(f)$ and any non-empty countable subset $S$ of $D$,
\[
\bigcap_{x\in S}U_D(x)
\]
is a dense $G_\delta$-subset of $D$.
\end{cor}

\begin{proof}
Given any $D\in\mathcal{D}(f)$ and $x\in D$, we have
\[
U_D(x)=\bigcap_{l=1}^\infty\bigcap_{N=1}^\infty\bigcup_{n\ge N}\{y\in D\colon\frac{1}{n}\sum_{i=0}^{n-1}d(f^i(x),f^i(y))<\frac{1}{l}\},
\]
so $U_D(x)$ is a $G_\delta$-subset of $D$. Since
\[
\{y\in D\colon\lim_{n\to\infty}\frac{1}{n}\sum_{i=0}^{n-1}d(f^i(x),f^i(y))=0\}\subset U_D(x),
\] 
and by Corollary 6.1,
\[
\{y\in D\colon\lim_{n\to\infty}\frac{1}{n}\sum_{i=0}^{n-1}d(f^i(x),f^i(y))=0\}
\]
is dense in $D$, $U_D(x)$ is a dense $G_\delta$-subset of $D$. Thus,
\[
\bigcap_{x\in S}U_D(x)
\]
is a dense $G_\delta$-subset of $D$, proving the corollary.
\end{proof}

\begin{cor}
Let $f\colon X\to X$ be a chain transitive continuous map. If $f$ has DSP, then for any $x,y\in X$ and $\epsilon>0$, there are $j\ge0$ and $z\in X$ such that $d(y,z)<\epsilon$ and
\[
\lim_{n\to\infty}\frac{1}{n}\sum_{i=0}^{n-1}d(f^{i+j}(x),f^i(z))=0.
\] 
\end{cor}

\begin{proof}
Since $f$ has DSP, for any $\epsilon>0$, we have $\mathcal{D}(f)=\mathcal{D}^{\epsilon/2,\delta}(f)$ for some $\delta>0$. For such $\delta>0$, the chain transitivity of $f$ gives a $\delta$-chain $\alpha=(x_i)_{i=0}^k$ of $f$ with $x_0=x$ and $x_k=y$. Then, it holds that $d(f^i(w),x_i)\le\epsilon/2$ for all $0\le i\le k$ for some $w\in\mathcal{D}_f(x_0)=\mathcal{D}_f(x)$, which implies $f^k(w)\in\mathcal{D}_f(f^k(x))$ and $d(f^k(w),y)=d(f^k(w),x_k)\le\epsilon/2$. By Corollary 6.1, we obtain $z\in\mathcal{D}_f(f^k(x))$ such that $d(f^k(w),z)<\epsilon/2$ and 
\[
\lim_{n\to\infty}\frac{1}{n}\sum_{i=0}^{n-1}d(f^{i+k}(x),f^i(z))=0.
\]
Since $d(y,z)\le d(f^k(w),y)+d(f^k(w),z)<\epsilon$, the corollary has been proved.
\end{proof}

\begin{cor}
Let $f\colon X\to X$ be a chain transitive continuous map and let $\mathcal{M}_f^{\rm erg}(X)$ be the set of ergodic $f$-invariant Borel probability measures on $X$. If $f$ has DSP, then for any $\mu\in\mathcal{M}_f^{\rm erg}(X)$,
\[
G_\mu=\{x\in X\colon\lim_{n\to\infty}\frac{1}{n}\sum_{i=0}^{n-1}\delta_{f^i(x)}=\mu\}
\]
is a dense subset of $X$.
\end{cor}

\begin{proof}
For any $\mu\in\mathcal{M}_f^{\rm erg}(X)$, it is well-known that, as a consequence of  Birkhoff's ergodic theorem, $\mu(G_\mu)=1$, thus in particular, $G_\mu\ne\emptyset$. Taking $x\in G_\mu$, we obtain
\[
\lim_{n\to\infty}\frac{1}{n}\sum_{i=0}^{n-1}\delta_{f^i(x)}=\mu
\]
and so
\[
\lim_{n\to\infty}\frac{1}{n}\sum_{i=0}^{n-1}\delta_{f^{i+j}(x)}=\mu
\]
for every $j\ge0$. Then, for any $y\in X$ and $\epsilon>0$, by Corollary 6.3, we have $d(y,z)<\epsilon$ and
\[
\lim_{n\to\infty}\frac{1}{n}\sum_{i=0}^{n-1}d(f^{i+j}(x),f^i(z))=0
\]
for some $j\ge0$ and $z\in X$, implying
\[
\lim_{n\to\infty}\frac{1}{n}\sum_{i=0}^{n-1}\delta_{f^i(z)}=\lim_{n\to\infty}\frac{1}{n}\sum_{i=0}^{n-1}\delta_{f^{i+j}(x)}=\mu,
\]
that is, $z\in G_\mu$. This shows that $G_\mu$ is a dense subset of $X$, completing the proof.
\end{proof}

Finally, for the proof of Theorem 1.10 in Section 7, we prove the following lemma.

\begin{lem}
Let $f\colon X\to X$ be a chain transitive continuous map. If $f$ has DSP, then for every $D\in\mathcal{D}(f)$, $D\cap\overline{T}(f)$ is a residual subset of $D$, where
\[
\overline{T}(f)=\{x\in X\colon X=\overline{\omega}(x,f)\}.
\]
\end{lem}

\begin{proof}
We say that $x\in X$ is a {\em minimal point} for $f$ if $x\in\omega(x,f)$, and $\omega(x,f)$ is a minimal set for $f$. Since $f$ is chain transitive and has DSP (and so the shadowing property), by Corollary 1 of \cite{M}, the set of minimal points for $f$ is dense in $X$, therefore, we can take a countable family $\mathcal{F}$ of disjoint minimal sets for $f$ such that
\[
X=\overline{\bigcup_{M\in\mathcal{F}}M}.
\]
Let $D\in\mathcal{D}(f)$ and for any $M\in\mathcal{F}$, similarly as in the proof of Lemma 6.4, take $x_M\in M$ such that $x_M\in D$. Then, by Corollary 6.2,
\[
\bigcap_{M\in\mathcal{F}}U_D(x_M)
\]
is a dense $G_\delta$-subset of $D$. Given any
\[
y\in\bigcap_{M\in\mathcal{F}}U_D(x_M)
\]
and $M\in\mathcal{F}$, since $y\in U_D(x_M)$, similarly as in the proof of Proposition 5.2, we have
\[
\overline{\omega}(x_M,f)\cap\overline{\omega}(y,f)\ne\emptyset.
\]
By $x_M\in M$ and minimality, we obtain $M\cap\overline{\omega}(y,f)\ne\emptyset$ and so $M\subset\overline{\omega}(y,f)$, which implies
\[
X=\overline{\bigcup_{M\in\mathcal{F}}M}\subset\overline{\omega}(y,f);
\]
therefore, $y\in\overline{T}(f)$. It follows that
\[
\bigcap_{M\in\mathcal{F}}U_D(x_M)\subset D\cap\overline{T}(f),
\]
and thus $D\cap\overline{T}(f)$ is a residual subset of $D$, proving the lemma.
\end{proof}

\section{Proofs of Theorems 1.8, 1.9 and 1.10}

In this section, we prove Theorem 1.8, 1.9, and 1.10. First, we prove Theorem 1.8.

\begin{proof}[Proof of Theorem 1.8.]
For any $\phi\in C(X)$ and $a,b\in\mathbb{R}$, let
\begin{align*}
X^-(f,\phi,a)&=\{x\in X\colon\liminf_{n\to\infty}\frac{1}{n}\sum_{i=0}^{n-1}\phi(f^i(x))\le a\}\\
&=\bigcap_{m=1}^\infty\bigcap_{N=1}^\infty\bigcup_{n\ge N}\{x\in X\colon\frac{1}{n}\sum_{i=0}^{n-1}\phi(f^i(x))<a+\frac{1}{m}\}
\end{align*}
and
\begin{align*}
X^+(f,\phi,b)&=\{x\in X\colon\limsup_{n\to\infty}\frac{1}{n}\sum_{i=0}^{n-1}\phi(f^i(x))\ge b\}\\
&=\bigcap_{m=1}^\infty\bigcap_{N=1}^\infty\bigcup_{n\ge N}\{x\in X\colon\frac{1}{n}\sum_{i=0}^{n-1}\phi(f^i(x))>b-\frac{1}{m}\},
\end{align*}
which are $G_\delta$-subsets of $X$.

Since $f$ has DSP and so the shadowing property, by $h_{\rm top}(f)>0$ and Lemma 6.2, we have two disjoint closed $f$-invariant subsets $K_0,K_1$ of $X$. Take $\phi\in C(X)$ such that $\phi(x)=j$ for all $j\in\{0,1\}$ and $x\in K_j$. We choose $\epsilon>0$ so small  that $|\phi(p)-\phi(q)|<1/3$ for all $p,q\in X$ with $d(p,q)<\epsilon$. Fix any $D\in\mathcal{D}(f)$ and $0<\eta<\epsilon$. Since $f$ has DSP, we have $D\in\mathcal{D}^{\eta,\delta}(f)$ for some $\delta>0$. Then, for every $y\in D$, the chain transitivity of $f$ gives a pair of $\delta$-chains
\[
(\alpha_0,\alpha_1)=((x_i^{(0)})_{i=0}^k,(x_i^{(1)})_{i=0}^k)
\]
of $f$ such that $x_0^{(0)}=x_0^{(1)}=y$, $x_k^{(0)}\in K_0$, and $x_k^{(1)}\in K_1$. For each $j\in\{0,1\}$, since $D\in\mathcal{D}^{\eta,\delta}(f)$ and $x_0^{(j)}=y\in D$, a $\delta$-pseudo orbit
\[
\xi_j=\alpha_j(x_k^{(j)},f(x_k^{(j)}),f^2(x_k^{(j)}),\dots)
\]
of $f$ is $\eta$-shadowed by some $y_j\in D$. Then, by the choice of $\epsilon$ and $\xi_j$, $j\in\{0,1\}$, we obtain $y_0\in X^-(f,\phi,1/3)$ and $y_1\in X^+(f,\phi,2/3)$. Note that $d(y_j,y)=d(y_j,x_0^{(j)})\le\eta$ for each $j\in\{0,1\}$. Since $0<\eta<\epsilon$ is arbitrary, it follows that $D\cap X^-(f,\phi,1/3)$ and $D\cap X^+(f,\phi,2/3)$ are dense $G_\delta$-subsets of $D$; therefore,
\[
D\cap[X^-(f,\phi,1/3)\cap X^+(f,\phi,2/3)]
\]
is a dense $G_\delta$-subset of $D$. By
\[
D\cap[X^-(f,\phi,1/3)\cap X^+(f,\phi,2/3)]\subset D\cap I(f,\phi),
\]
we conclude that $D\cap I(f,\phi)$ is a residual subset of $D$. Since $D\in\mathcal{D}(f)$ is arbitrary, the proof has been completed. 
\end{proof}

\begin{rem}
\normalfont
By a similar argument as in the above proof of Theorem 1.8, we can show that for any chain transitive continuous map $f\colon X\to X$, if $f$ satisfies the shadowing property and $h_{\rm top}(f)>0$, then there exists $\phi\in C(X)$ such that $I(f,\phi)$ is a residual subset of $X$.
\end{rem}

Next, we prove Theorem 1.9.

\begin{proof}[Proof of Theorem 1.9]
Note that
\[
I(f,\phi)=\bigcup_{\epsilon>0}\bigcap_{N=1}^\infty\bigcup_{m,n\ge N}\{x\in X\colon|\frac{1}{m}\sum_{i=0}^{m-1}\phi(f^i(x))-\frac{1}{n}\sum_{i=0}^{n-1}\phi(f^i(x))|>\epsilon\}.
\]
Since $I(f,\phi)\ne\emptyset$, letting
\[
I(f,\phi,\epsilon)=\bigcap_{N=1}^\infty\bigcup_{m,n\ge N}\{x\in X\colon|\frac{1}{m}\sum_{i=0}^{m-1}\phi(f^i(x))-\frac{1}{n}\sum_{i=0}^{n-1}\phi(f^i(x))|>\epsilon\},
\]
$\epsilon>0$, we have $I(f,\phi,3\epsilon)\ne\emptyset$ for some $\epsilon>0$. For such $\epsilon>0$, fix
\[
x\in I(f,\phi,3\epsilon)
\]
and take $\zeta>0$ such that $d(a,b)\le\zeta$ implies
\[
|\phi(a)-\phi(b)|\le\epsilon/3
\]
for all $a,b\in X$. Since $f$ has DSP, we have
\[
\mathcal{D}(f)=\mathcal{D}^{\zeta,\delta}(f)
\]
for some $\delta>0$. Then, for any $D\in\mathcal{D}(f)$ and $y\in D$, the chain transitivity of $f$ gives a $\delta$-chain $\alpha=(x_i)_{i=0}^k$ of $f$ with $x_0=y$ and $x_k=x$. Consider a $\delta$-pseudo orbit
\[
\xi=(x_i)_{i\ge0}=\alpha(x,f(x),f^2(x),\dots)
\]
and note that $\xi$ is $\zeta$-shadowed by some $z\in\mathcal{D}_f(x_0)$, implying $z\in\mathcal{D}_f(y)=D$ and $d(y,z)=d(z,x_0)\le\zeta$.

Let $M=\sup_{x\in X}|\phi(x)|$ and take $N\ge1$ so large that $\frac{k}{N+k}M\le\epsilon/3$. Since $x\in I(f,\phi,3\epsilon)$, we can take $m,n\ge N$ such that
\[
|\frac{1}{m}\sum_{i=0}^{m-1}\phi(f^i(x))-\frac{1}{n}\sum_{i=0}^{n-1}\phi(f^i(x))|>3\epsilon.
\]
Then, because
\begin{align*}
&|\frac{1}{m+k}\sum_{i=0}^{m+k-1}\phi(f^i(z))-\frac{1}{m+k}\sum_{i=k}^{m+k-1}\phi(f^i(z))|=|\frac{1}{m+k}\sum_{i=0}^{k-1}\phi(f^i(z))|\\
&\le\frac{k}{m+k}M\le\frac{k}{N+k}M\le\epsilon/3,
\end{align*}
\begin{align*}
&|\frac{1}{m+k}\sum_{i=k}^{m+k-1}\phi(f^i(z))-\frac{1}{m}\sum_{i=k}^{m+k-1}\phi(f^i(z))|=|-\frac{k}{(m+k)m}\sum_{i=k}^{m+k-1}\phi(f^i(z))|\\
&\le\frac{km}{(m+k)m}M\le\frac{k}{N+k}M\le\epsilon/3,
\end{align*}
and
\begin{align*}
&|\frac{1}{m}\sum_{i=k}^{m+k-1}\phi(f^i(z))-\frac{1}{m}\sum_{i=0}^{m-1}\phi(f^i(x))|=|\frac{1}{m}\sum_{i=k}^{m+k-1}(\phi(f^i(z))-\phi(x_i))|\\
&\le\frac{m}{m}\cdot\epsilon/3=\epsilon/3,
\end{align*}
we have
\[
|\frac{1}{m+k}\sum_{i=0}^{m+k-1}\phi(f^i(z))-\frac{1}{m}\sum_{i=0}^{m-1}\phi(f^i(x))|\le\epsilon.
\]
Similarly, we have
\[
|\frac{1}{n+k}\sum_{i=0}^{n+k-1}\phi(f^i(z))-\frac{1}{n}\sum_{i=0}^{n-1}\phi(f^i(x))|\le\epsilon,
\]
thus by the choice of $m,n$, we obtain
\[
|\frac{1}{m+k}\sum_{i=0}^{m+k-1}\phi(f^i(z))-\frac{1}{n+k}\sum_{i=0}^{n+k-1}\phi(f^i(z))|>\epsilon.
\]
Note that $m+k,n+k\ge N$. Since $N$ can be taken arbitrarily large, it follows that $z\in I(f,\phi,\epsilon)$. Recall that $z\in D$ and $d(y,z)\le\zeta$. Because $\zeta$ can be taken arbitrarily small, we conclude that
\[
y\in\overline{D\cap I(f,\phi,\epsilon)},
\]
and so
\[
D\cap I(f,\phi,\epsilon)
\]
is a dense $G_\delta$-subset of $D$. Since $D\cap I(f,\phi)$ contains $D\cap I(f,\phi,\epsilon)$, $D\cap I(f,\phi)$ is a residual subset of $D$. Since $D\in\mathcal{D}(f)$ is arbitrary, the proof has been completed.
\end{proof}

\begin{rem}
\normalfont
By a similar argument as in the above proof of Theorem 1.9, we can show that for any chain transitive continuous map $f\colon X\to X$, if $f$ satisfies the shadowing property, then for any $\phi\in C(X)$, $I(f,\phi)\ne\emptyset$ implies that $I(f,\phi)$ is a residual subset of $X$.
\end{rem}

Finally, we prove Theorem 1.10. For the proof, we use a simplified version of Mycielski's theorem.

\begin{lem}{\cite[Theorem 1]{My}}
Let $X$ be a perfect complete metric space. If $R_n$ is a residual subset of $X^n$ for each $n\ge2$, then there is a Mycielski set $S$ which is dense in $X$ and satisfies
$(x_1,x_2,\dots,x_n)\in R_n$ for any $n\ge2$ and distinct $x_1,x_2,\dots,x_n\in S$.
\end{lem}

\begin{proof}[Proof of Theorem 1.10]
Since $f$ satisfies DSP and $h_{\rm top}(f)>0$, as shown in the proof of \cite[Lemma 1.3]{Ka4}, there exists a sequence of positive numbers $(\delta_n)_{n\ge2}$ such that
\[
D^n\cap{\rm DC1}_n^{\delta_n}(X,f)
\]
is a residual subset of $D^n$ for all $D\in\mathcal{D}(f)$ and $n\ge2$, where ${\rm DC1}_n^{\delta_n}(X,f)$ is the set of distributionally $n$-$\delta_n$-scrambled $n$-tuples for $f$. By Theorem 1.8, we have $I(f,\phi)\ne\emptyset$ for some $\phi\in C(X)$. Given any $\psi\in C(X)$ with $I(f,\psi)\ne\emptyset$, combining Lemma 6.5 and Theorem 1.9, we see that 
\[
D^n\cap[\overline{T}(f)^n\cap I(f,\psi)^n\cap{\rm DC1}_n^{\delta_n}(X,f)]
\]
is a residual subset of $D^n$ for all $D\in\mathcal{D}(f)$ and $n\ge2$. From Lemma 7.1, it follows that every $D\in\mathcal{D}(f)$ contains a dense Mycielski subset $S$ such that
\[
(x_1,x_2,\dots,x_n)\in\overline{T}(f)^n\cap I(f,\psi)^n\cap{\rm DC1}_n^{\delta_n}(X,f)
\]
for any $n\ge2$ and distinct $x_1,x_2,\dots,x_n\in S$, which completes the proof.
\end{proof}

\appendix
\section{}

The aim of this appendix is to prove the following theorem.

\begin{thm}
Let $I=[0,1]$ and let $f\colon I\to I$ be a continuous map. Then, $f$ is $\overline{\omega}$-chaotic if and only if $h_{\rm top}(f)>0$. 
\end{thm}

For the ``only if" part, it suffices to prove the following lemma which slightly generalizes Proposition 4.6 of \cite{SHLi}.

\begin{lem}
Let $f\colon I\to I$ be a continuous map. If there are $x,y\in I$ and closed $f$-invariant subsets $\Lambda(x),\Lambda(y)$ of $I$ such that
\begin{itemize}
\item[(1)] $\Lambda(x)\subset\omega(x,f)$ and $\Lambda(y)\subset\omega(y,f)$,
\item[(2)] $\Lambda(x)\setminus\Lambda(y)$ is an uncountable set,
\item[(3)] $\Lambda(x)\cap\Lambda(y)\ne\emptyset$,
\end{itemize}
then $h_{\rm top}(f)>0$.
\end{lem}

\begin{proof}
We assume $h_{\rm top}(f)=0$ to obtain a contradiction. Let $\Lambda=\bigcup_{p\in I}\omega(p,f)$ and $\Lambda^2=\bigcup_{p\in\Lambda}\omega(p,f)$. Let $M(f)$ denote the set of minimal points for $f$:
\[
M(f)=\{p\in I\colon\text{$p\in\omega(p,f)$ and $\omega(p,f)$ is a minimal set for $f$}\}.
\]
Following the proof of Proposition 4.6 in \cite{SHLi}, we use the following facts (see \cite{SHLi} for details): given any continuous map $f\colon I\to I$,
\begin{itemize}
\item[(i)]  if $h_{\rm top}(f)=0$, then for every $p\in I$, $\omega(p,f)$ contains a unique minimal set for $f$,
\item[(ii)] if $h_{\rm top}(f)=0$, then $\Lambda^2=M(f)$,
\item[(iii)] $\Lambda\setminus\Lambda^2$ is a countable set.
\end{itemize}

By (1) and (i), $\Lambda(x)\subset\Lambda$ (resp.\:$\Lambda(y)\subset\Lambda$) and $\Lambda(x)$ (resp.\:$\Lambda(y)$) contains a unique minimal set $M$ (resp.\:$M'$) for $f$. Since $M\subset M(f)\subset\Lambda^2$, we have 
\[
M\subset\Lambda(x)\cap\Lambda^2.
\]
For every $z\in\Lambda(x)\cap\Lambda^2$, by (ii), we have $z\in M(f)$, so
\[
z\in\omega(z,f)\subset\Lambda(x)
\]
and $\omega(z,f)$ is a minimal set for $f$, implying $z\in\omega(z,f)=M$ by the uniqueness of $M$. It follows that $\Lambda(x)\cap\Lambda^2\subset M$ and so $\Lambda(x)\cap\Lambda^2=M$. Similarly, we obtain $\Lambda(y)\cap\Lambda^2=M'$. By (3), we can take a minimal set $M''$ for $f$ such that
\[
M''\in\Lambda(x)\cap\Lambda(y).
\] 
By the uniqueness of $M$ and $M'$, we obtain $M''=M=M'$ and so
\[
\Lambda(x)\cap\Lambda^2=M=M'=\Lambda(y)\cap\Lambda^2.
\]
It follows that $\Lambda(x)\setminus\Lambda(y)\subset\Lambda\setminus\Lambda^2$, therefore by (iii), $\Lambda(x)\setminus\Lambda(y)$ is a countable set, but this contradicts (2). Thus, we conclude that $h_{\rm top}(f)>0$, proving the lemma.
\end{proof}

\begin{cor}
For any continuous map $f\colon I\to I$, if there are $x,y\in I$ with $x\ne y$ such that $\{x,y\}$ is $\overline{\omega}$-scrambled, then $h_{\rm top}(f)>0$. In particular, if $f$ is $\overline{\omega}$-chaotic, then $h_{\rm top}(f)>0$.
\end{cor}

For the proof of the ``if" part, we need the following lemma.

\begin{lem}
Let $X,Y$ be compact metric spaces, let $f\colon X\to X$, $g\colon Y\to Y$ be continuous maps, and let
\[
\pi\colon(X,f)\to(Y,g) 
\]
be a factor map. Then, it holds that $\pi(\overline{\omega}(x,f))=\overline{\omega}(\pi(x),g)$ for every $x\in X$.
\end{lem}

\begin{proof}
The proof is similar to that of property $(2)$ in Lemma 5.2. We define
\[
\pi_\ast\colon\mathcal{M}(X)\to\mathcal{M}(Y)
\]
by $\pi_\ast(\mu)(A)=\mu(\pi^{-1}(A))$ for every $\mu\in\mathcal{M}(X)$ and every Borel subset $A$ of $Y$. Then, we can show that $\pi_\ast(S_f(x))=S_g(\pi(x))$ for all $x\in X$. On the other hand, we have $\pi({\rm supp}(\mu))={\rm supp}(\pi_\ast(\mu))$ for all $\mu\in\mathcal{M}(X)$. Thus, by Lemma 5.1, we obtain
\begin{align*}
\pi(\overline{\omega}(x,f))&=\pi(\overline{\bigcup_{\mu\in S_f(x)}{\rm supp}(\mu)})\\
&=\overline{\bigcup_{\mu\in S_f(x)}\pi({\rm supp}(\mu))}=\overline{\bigcup_{\mu\in S_f(x)}{\rm supp}(\pi_\ast(\mu))}\\
&=\overline{\bigcup_{\nu\in S_g(\pi(x))}{\rm supp}(\nu)}=\overline{\omega}(\pi(x),g),
\end{align*}
proving the lemma.
\end{proof}

We prove the ``if" part of Theorem A.1.

\begin{lem}
For any continuous map $f\colon I\to I$, if $h_{\rm top}(f)>0$, then $f$ exhibits $\overline{\omega}$-chaos.
\end{lem}

\begin{proof}
According to Proposition 5.15 of \cite{Ru}, we know that if $h_{\rm top}(f)>0$, then there are $m>0$, a closed $f^m$-invariant subset $Y$ of $I$, and a factor map
\[ 
\pi\colon(Y,f^m)\to(\{0,1\}^\mathbb{N},\sigma),
\]
here $\sigma\colon\{0,1\}^\mathbb{N}\to\{0,1\}^\mathbb{N}$ is the shift map, such that
\[
Q=\{q\in\{0,1\}^\mathbb{N}\colon|\pi^{-1}(q)|>1\}
\]
is a countable set. As in the proof of Lemma 6.2, we take a set $\mathcal{F}$ of infinite minimal sets for $\sigma$ such that $|\mathcal{F}|=|\mathbb{R}|$. Since
\[
|\{M\in\mathcal{F}\colon q\in M\}|\le1
\]
for every $q\in Q$, letting
\[
\mathcal{G}=\{M\in\mathcal{F}\colon M\cap Q=\emptyset\},
\]
we have $|\mathcal{G}|=|\mathbb{R}|$. Fix $M_0\in\mathcal{G}$. Since $\sigma$ has the shadowing property and is chain mixing, $\sigma$ satisfies DSP; therefore, as shown in the proof of Theorem 1.6, for any $M\in\mathcal{G}\setminus\{M_0\}$, we can take $s_M\in\{0,1\}^\mathbb{N}$ such that
\[
\overline{\omega}(s_M,\sigma)=M_0\cup M.
\]
For each $M\in\mathcal{G}\setminus\{M_0\}$, taking $y_M\in Y$ with $\pi(y_M)=s_M$, by Lemma A.2, we obtain
\[
\pi(\overline{\omega}(y_M,f^m))=\overline{\omega}(s_M,\sigma)=M_0\cup M.
\]
Then, by the choice of $\mathcal{G}$, for any $M\in\mathcal{G}\setminus\{M_0\}$, we have
\[
\overline{\omega}(y_M,f^m)=\pi^{-1}(M_0\cup M)=\pi^{-1}(M_0)\cup\pi^{-1}(M),
\]
and $\pi^{-1}(M)$,$\pi^{-1}(M_0)$ are minimal sets for $f^m$. Let
\[
\mathcal{H}=\{\pi^{-1}(M_0)\}\cup\{\pi^{-1}(M)\colon M\in\mathcal{G}\setminus\{M_0\}\}.
\]
Then, for every $N\in\mathcal{H}$, it holds that $f^m(N)=N$,  $f^i(N)$, $0\le i\le m-1$, are minimal sets for $f^m$, and
\[
\langle N\rangle=\bigcup_{i=0}^{m-1}f^i(N)
\]
is an infinite minimal set for $f$. Similarly as in the proof of Lemma 6.2, we define an equivalence relation $\equiv$ in $\mathcal{H}^2$ by for all $N,N'\in\mathcal{H}$, $N\equiv N'$ if and only if $N'=f^k(N)$ for some $0\le k\le m-1$. Since every equivalence class with respect to $\equiv$ has at most $m$ elements, we have $|\mathcal{H}\slash\equiv|=|\mathbb{R}|$. For each $t\in\mathcal{H}\slash\equiv$, take $N^t\in\mathcal{H}$ such that $N^t\in t$. As in the proof of Lemma 6.2, we see that $t\ne t'$ implies $\langle N^t\rangle\ne\langle N^{t'}\rangle$ for all $t,t'\in\mathcal{H}\slash\equiv$. For the $t_0\in\mathcal{H}\slash\equiv$ with $\pi^{-1}(M_0)\in t_0$, we may assume $N^{t_0}=\pi^{-1}(M_0)$. For every $t\in\mathcal{H}\slash\equiv$ with $t\ne t_0$, we have $N^t=\pi^{-1}(M_t)$ for some $M_t\in\mathcal{G}\setminus\{M_0\}$. Let $\mathcal{K}=(\mathcal{H}\slash\equiv)\setminus\{t_0\}$ and let
\[
S=\{y_{M_t}\colon t\in\mathcal{K}\}.
\]
Then, for every $t\in\mathcal{K}$, by properties (2) and (3) in Lemma 5.2, it holds that
\begin{align*}
\overline{\omega}(y_{M_t},f)&=\bigcup_{i=0}^{m-1}\overline{\omega}(f^i(y_{M_t}),f^m)=\bigcup_{i=0}^{m-1}f^i(\overline{\omega}(y_{M_t},f^m))\\
&=\bigcup_{i=0}^{m-1}f^i(\pi^{-1}(M_0)\cup\pi^{-1}(M_t))=\bigcup_{i=0}^{m-1}f^i(\pi^{-1}(M_0))\cup\bigcup_{i=0}^{m-1}f^i(\pi^{-1}(M_t))\\
&=\bigcup_{i=0}^{m-1}f^i(N^{t_0})\cup\bigcup_{i=0}^{m-1}f^i(N^t)=\langle N^{t_0}\rangle\cup\langle N^t\rangle.
\end{align*}
It follows that for all $t,t'\in\mathcal{K}$ with $t\ne t'$, $\overline{\omega}(y_{M_t},f)\setminus\overline{\omega}(y_{M_{t'}},f)=\langle N^t \rangle$, $\overline{\omega}(y_{M_{t'}},f)\setminus\overline{\omega}(y_{M_{t}},f)=\langle N^{t'} \rangle$, and
\[
\overline{\omega}(y_{M_t},f)\cap\overline{\omega}(y_{M_{t'}},f)=\langle N^{t_0}\rangle.
\]
This implies that $S$ is an $\overline{\omega}$-scrambled subset of $I$ with $|S|=|\mathcal{K}|=|\mathcal{H}\slash\equiv|=|\mathbb{R}|$, and so $f$ exhibits $\overline{\omega}$-chaos. Thus, the lemma has been proved.
\end{proof}

\section{}

In this appendix, we discuss properties of $\overline{\omega}$-limit sets other than those discussed in Section 5. Let $X$ be a compact metric space. For any non-empty subset $S$ of $\mathcal{M}(X)$, following \cite{Sig}, let
\[
{\rm Supp}(S)=\overline{\bigcup_{\mu\in S}{\rm supp}(\mu)}.
\]
Let $K(X)$ denote the set of non-empty closed subsets of $X$. Then, it is easy to see that
\[
{\rm Supp}(S)=\bigcap\{C\in K(X)\colon\mu(C)=1,\,\forall\mu\in S\}.
\]
Let $f\colon X\to X$ be a continuous map. For any $x\in X$ and $C\in K(X)$, we have $d(\{i\ge0\colon d(f^i(x),C)<\epsilon\})=1$ for all $\epsilon>0$, if and only if
\[
\lim_{n\to\infty}\frac{1}{n}\sum_{i=0}^{n-1}d(f^i(x),C)=0.
\]
By this, we easily see that $\overline{\omega}(x,f)\subset C$ is equivalent to
\[
\lim_{n\to\infty}\frac{1}{n}\sum_{i=0}^{n-1}d(f^i(x),C)=0.
\]
Given any non-empty subset $Y$ of $X$, from this equivalence, it follows that
\[
\overline{\bigcup_{x\in Y}\overline{\omega}(x,f)}=\bigcap\{C\in K(X)\colon\lim_{n\to\infty}\frac{1}{n}\sum_{i=0}^{n-1}d(f^i(x),C)=0,\,\forall x\in Y\}.
\]
In particular, when $Y=\{x\}$, this means
\[
\overline{\omega}(x,f)=\bigcap\{C\in K(X)\colon\lim_{n\to\infty}\frac{1}{n}\sum_{i=0}^{n-1}d(f^i(x),C)=0\}.
\]
On the other hand, letting $S_f(Y)=\bigcup_{x\in Y}S_f(x)$, by Lemma 5.1, we obtain
\begin{align*}
\overline{\bigcup_{x\in Y}\overline{\omega}(x,f)}&=\overline{\bigcup_{x\in Y}\overline{\bigcup_{\mu\in S_f(x)}{\rm supp}(\mu)}}=\overline{\bigcup_{x\in Y}\bigcup_{\mu\in S_f(x)}{\rm supp}(\mu)}=\overline{\bigcup_{\mu\in S_f(Y)}{\rm supp}(\mu)}\\
&={\rm Supp}(S_f(Y))=\bigcap\{C\in K(X)\colon\mu(C)=1,\,\forall\mu\in S_f(Y)\}.
\end{align*}
We consider the case where $Y=X$. Let $\mathcal{M}_f(X)$ (resp.\:$\mathcal{M}_f^{\rm erg}(X)$) denote the set of $f$-invariant (resp.\:ergodic $f$-invariant) $\mu\in\mathcal{M}(X)$. Then, for every $\mu\in\mathcal{M}_f^{\rm erg}(X)$, by Birkhoff's ergodic theorem, there is $x\in X$ such that
\[
\lim_{n\to\infty}\frac{1}{n}\sum_{i=0}^{n-1}\delta_{f^i(x)}=\mu
\]
and so $\mu\in S_f(x)\subset S_f(X)$. Given any $C\in K(X)$, it follows that if $\mu(C)=1$ for all $\mu\in S_f(X)$, then $\mu(C)=1$ for every $\mu\in\mathcal{M}_f^{\rm erg}(X)$, and this, in turn, implies $\mu(C)=1$ for all $\mu\in\mathcal{M}_f(X)$. Conversely, if $\mu(C)=1$ for all $\mu\in\mathcal{M}_f(X)$, then because $S_f(X)\subset\mathcal{M}_f(X)$, $\mu(C)=1$ for all $\mu\in S_f(X)$.
Therefore, by the above equalities, we obtain
\begin{align*}
\overline{\bigcup_{x\in X}\overline{\omega}(x,f)}&=\bigcap\{C\in K(X)\colon\lim_{n\to\infty}\frac{1}{n}\sum_{i=0}^{n-1}d(f^i(x),C)=0,\,\forall x\in X\}\\
&=\bigcap\{C\in K(X)\colon\mu(C)=1,\,\forall\mu\in\mathcal{M}_f(X)\}\\
&={\rm Supp}(\mathcal{M}_f(X)).
\end{align*}
This closed $f$-invariant subset of $X$ is often called the {\em measure center} for $f$. For any minimal set $M$ for $f$, we have $M=\overline{\omega}(x,f)$ for every $x\in M$, so
\[
M\subset\overline{\bigcup_{x\in X}\overline{\omega}(x,f)}.
\]
It follows that the measure center for $f$ contains the closure of the union of all minimal sets for $f$. Let $\overline{R}(f)=\{x\in X\colon x\in\overline{\omega}(x,f)\}$. Then, it is obvious that
\[
\overline{\overline{R}(f)}\subset\overline{\bigcup_{x\in X}\overline{\omega}(x,f)}.
\]
Since
\[
\overline{R}(f)=\bigcap_{l=1}^\infty\bigcup_{m=1}^\infty\bigcap_{N=1}^\infty\bigcup_{n\ge N}\{x\in X\colon\frac{1}{n}|\{0\le i\le n-1\colon d(f^i(x),x)<\frac{1}{l}\}|>\frac{1}{m}\},
\]
$\overline{R}(f)$ is a Borel subset of $X$. Given any $\mu\in\mathcal{M}_f^{\rm erg}(X)$, by Birkhoff's ergodic theorem, the set of generic points for $\mu$
\[
G_\mu=\{x\in X\colon\lim_{n\to\infty}\frac{1}{n}\sum_{i=0}^{n-1}\delta_{f^i(x)}=\mu\}
\]
is a Borel subset of $X$ and satisfies $\mu(G_\mu)=1$. Since
\[
x\in{\rm supp}(\mu)=\overline{\omega}(x,f)
\]
for all $x\in G_\mu\cap{\rm supp}(\mu)$, we obtain $G_\mu\cap{\rm supp}(\mu)\subset\overline{R}(f)$ and so 
$\mu(\overline{R}(f))=1$, implying $\mu(\overline{\overline{R}(f)})=1$. Since $\mu\in\mathcal{M}_f^{\rm erg}(X)$ is arbitrary, it follows that $\mu(\overline{\overline{R}(f)})=1$ for every $\mu\in\mathcal{M}_f(X)$. Therefore, by the above equality, we obtain
\[
\overline{\bigcup_{x\in X}\overline{\omega}(x,f)}\subset\overline{\overline{R}(f)}.
\]
In other words, the measure center for $f$ coincides with the closure of $\overline{R}(f)$.

\end{document}